\documentclass[a4paper,10pt]{article}

\usepackage{ucs}
\usepackage[latin1]{inputenc}
\usepackage[german,french,english]{babel}
\usepackage[T1]{fontenc}

\usepackage{caption}
\usepackage{subcaption}

\usepackage{amssymb,amsfonts,amsmath,amsthm,mathrsfs}
\usepackage{paralist}
\usepackage{epsfig}
\usepackage{psfrag}
\usepackage[all]{xy}
\usepackage{graphicx}
\usepackage{color}
\usepackage{multirow}

\usepackage{tikz}
\usetikzlibrary{matrix,arrows,shapes,decorations.pathmorphing}
\usepackage{xifthen}
\usepackage{verbatim}

\newcommand{\NN}{{\mathbb N}} 
\newcommand{\ZZ}{{\mathbb Z}}

\newcommand{\CC}{{\mathbb C}}

\newcommand{\proj}{{\mathbb P}}

\newcommand{\hyp}{{\rm hyp}}

\renewcommand{\tilde}{\widetilde}

\newcommand{\isom}{\cong}

\newcommand{\Res}{{\rm Res}}

\newcommand{\barmoduli}[1][g]{{\overline{\mathcal M}}_{#1}}
\newcommand{\moduli}[1][g]{{\mathcal M}_{#1}}

\newcommand{\omoduli}[1][g]{{\Omega\mathcal M}_{#1}}
\newcommand{\modulin}[1][g,n]{{\mathcal M}_{#1}}
\newcommand{\omodulin}[1][g,n]{{\Omega\mathcal M}_{#1}}
\newcommand{\zomoduli}[1][]{{\mathcal H}_{#1}}

\newcommand{\pomoduli}[1][g]{{\proj\Omega\mathcal M}_{#1}}
\newcommand{\pobarmoduli}[1][g]{{\proj\Omega\overline{\mathcal M}}_{#1}}

\newcommand{\obarmoduli}[1][g]{{\Omega\overline{\mathcal M}}_{#1}}

\newcommand{\omoduliinc}[2][g,n]{{\Omega\mathcal M}_{#1}^{{\rm inc}}(#2)}
\newcommand{\obarmoduliinc}[2][g,n]{{\Omega\overline{\mathcal M}}_{#1}^{{\rm inc}}(#2)}
\newcommand{\omoduliincp}[2][g,\lbrace n \rbrace]{{\Omega\mathcal M}_{#1}^{{\rm inc}}(#2)}
\newcommand{\obarmoduliincp}[2][g,\lbrace n \rbrace]{{\Omega\overline{\mathcal M}}_{#1}^{{\rm inc}}(#2)}

\newcommand{\kbarmodulilim}[1][g,\left\{n\right\}]{{K\overline{\mathcal M}}_{#1}^{{\rm lim}}}
\newcommand{\kbarmoduliplum}[1][g,\left\{n\right\}]{{K\overline{\mathcal M}}_{#1}^{{\rm plb}}}

\newcommand{\spin}[1][g]{{\mathcal{S}}_{#1}}
\newcommand{\barspin}[1][g]{\overline{{\mathcal{S}}_{#1}}}

\newcommand{\hyperell}[1][g]{{\mathcal H}_{#1}}
\newcommand{\barhyperell}[1][g]{{\overline{\mathcal H}}_{#1}}

\newcommand{\permGroup}[1][n]{\mathfrak{S}_{#1}}

\newcommand{\tensor}{\otimes}

\newcommand{\odd}{{\rm odd}}
\newcommand{\even}{{\rm even}}

\newcommand{\dz}{{\rm d}z}
\newcommand{\dw}{{\rm d}w}
\newcommand{\dx}{{\rm d}x}
\newcommand{\dy}{{\rm d}y}

\newcommand{\ind}{{\rm Ind}} 
\newcommand{\Arf}{{\rm Arf}} 

\newcommand{\partder}[1]{\partial_{#1}}

\newcommand{\dualgraph}{\Gamma_{\rm dual}}

\newcommand{\X}{X} 
\newcommand{\Jac}[1][\X]{\mathcal{J}(#1)} 
 
\newcommand{\Ox}[1][\X]{\mathcal{O}_{#1}} 
\newcommand{\dualsheave}[1][\X]{\omega_{#1}}
\newcommand{\holoneform}[1][\X]{\Omega_{#1}^{1}} 
 \newcommand{\divisor}[1]{{\rm div }\left( #1 \right)}

\newcommand{\grd}[2]{\mathfrak{g}^{#1}_{#2}} 
 
\newcommand{\ord}{{\rm ord }} 
\newcommand{\famcurv}{\mathscr{X}} 
\newcommand{\famomega}{\mathscr{W}} 
\newcommand{\seczero}{\mathscr{Z}} 
\newcommand{\nodeSet}[1][\X]{\mathcal{N}_{#1}} 

\newcommand{\PP}{{\mathbb P}}
\newcommand{\Irr}{\mathfrak{Irr}} 

\newcommand{\WP}{{\mathcal{W}\mathcal{P}}}

\newcommand{\Pic}[1][]{\rm{Pic_{#1}}}

\newcommand{\br}[1][i]{{\rm br}_{i}}

\newcommand{\du}[1][i]{{\rm du}_{i}}

\theoremstyle{plain}{
\newtheorem{theorem}{Theorem}[section]
\newtheorem{cor}[theorem]{Corollary}
\newtheorem{conj}[theorem]{Conjecture}
\newtheorem{prop}[theorem]{Proposition}
\newtheorem{lemma}[theorem]{Lemma}

\theoremstyle{remark}{
\newtheorem{rem}[theorem]{Remark}

}

\theoremstyle{definition}{
\newtheorem{defn}[theorem]{Definition}

\newtheorem{ex}[theorem]{Example}

}

\title{The Deligne-Mumford and the Incidence Variety Compactifications of the Strata of
$\omoduli$.}
\author{Quentin Gendron}

\begin{document}

\maketitle

\begin{abstract}
The main goal of this work is to construct and study a reasonable compactification of the strata of the moduli
space
of Abelian differentials.
This allows us to compute the Kodaira dimension of some strata of the moduli
space of Abelian
differentials. 
The main ingredients to study the compactifications of the strata are a version of the
plumbing
cylinder construction for differential forms and an extension of the parity of the connected components
of the strata  to the differentials on curves of compact type. We study in detail the compactifications of the
hyperelliptic minimal
strata and of the odd minimal stratum in genus three.

\end{abstract}

\tableofcontents
\clearpage

\section{Introduction.}

Let $\moduli$ be the moduli space of algebraic curves of genus $g$. In the early 1980s Harris and Mumford
(\cite{MR664324}) proved that $\moduli$ is of general type for $g\geq24$. They used in a crucial way the
compactification of $\moduli$ proposed by  Deligne and Mumford at the end of the 1960s
(\cite{MR0262240}). This compactification is the moduli space $\barmoduli$ of
stable algebraic curves of arithmetic genus $g$.

More recently,  the moduli space of  nonzero holomorphic differentials
$\omoduli$ and its
projectivisation $\pomoduli$ have gained great interest, coming in particular from the theory of
dynamical systems (see \cite{MR2261104}).
The moduli space $\omoduli$ has a natural stratification given by the orders of the zeros of the
differentials.  
For a given tuple
$(k_{1},\cdots,k_{n})$ of positive numbers such that $\sum k_{i} = 2g-2$,  we define the stratum
$$\omoduli(k_{1},\cdots,k_{n}):=\left\{ (\X,\omega): \X\in\moduli, \quad
\divisor{\omega}=\sum_{i=1}^{n} k_{i}Z_{i} 
\right\},$$ 
and their images in $\pomoduli$ are denoted by
$\pomoduli(k_{1},\cdots,k_{n})$.
In analogy with~$\moduli$, it is likely that a good compactification of
$\pomoduli$ should help us to compute the Kodaira dimension of the strata of $\pomoduli$.

In this paper, we first introduce and study two compactifications of the strata of the moduli space of Abelian
differentials. This allows us to compute the Kodaira dimension of some of these strata. The last sections are
devoted to the study of the hyperelliptic minimal strata and the non hyperelliptic minimal stratum in genus
three.

\subsection{The {incidence variety compactification}.}

 The notion of Abelian differentials can be generalised  to the case of stable curves by the
notion of stable differentials. Therefore, we
can prolong $\omoduli$ above $\barmoduli$ simply by looking at the moduli space of stable differentials
$\obarmoduli$. The
closure of the strata inside $\obarmoduli$ are called the {\em Deligne-Mumford compactifications} of these
strata. The main drawback of this method is the loss of information. Indeed, a non vanishing stable
differential may vanish on some irreducible components of the stable curve, losing completely the information
on this component.

In order to keep track of more information, we introduce in Section~\ref{section:CompactDesStrates} another
compactification for the strata. Let us define the closure of the {\em ordered closed incidence variety}
$\PP\obarmoduliinc{k_{1},\cdots,k_{n}}$  inside the moduli space of  marked (semi) stable
differentials by
$$\overline{\left\{(\X,\omega,Z_{1},\cdots,Z_{n}) :
\left(\X,Z_{1},\cdots,Z_{n}\right)\in\modulin,\quad \sum_{i=1}^{n}
k_{i}Z_{i} = \divisor{\omega} \right\}}.$$
Now there is an action of a subgroup $\mathfrak{S}$ of $\permGroup$ permuting the zeros of same order.  The {\em
{incidence
variety compactification}} of $\omoduli(k_{1},\cdots,k_{n})$ is given by
$$\PP\obarmoduliincp{k_{1},\cdots,k_{n}}:=\PP\obarmoduliinc{k_{1},\cdots,k_{n}}/S.$$

The interior of the {incidence variety compactification} is isomorphic to the
strata $\pomoduli(k_{1},\cdots,k_{n})$. But we show that its closure
contains in general much more information than $\pobarmoduli(k_{1},\cdots,k_{n})$. The following theorem
illustrates this
point in the case of the principal stratum (see Theorem~\ref{theoreme:CompStratePrincDansVarIncEtNormal}). 
Let us denote the projection  from the {incidence variety compactification} to the Deligne-Mumford
compactification of the principal stratum by $$\pi:\PP\obarmoduliincp[g,\lbrace 2g-2
\rbrace]{1,\cdots,1}\to\pobarmoduli(1,\cdots,1).$$                                                  

\begin{theorem}
The fibre of 
$\pi$ is positive dimensional above the locus of differentials $(\X,\omega)$, where $\X$ is a reducible stable curve of
genus $g\geq 2$ with two irreducible components connected by one node
and $\omega$ vanishes on one component.
\end{theorem}

\paragraph{}
In order to study the {incidence variety compactification}, we introduce some tools.

In Section~\ref{section:PlomberieCylindrique}, we develop the theory of {\em limit differentials},  which has
a flavour of limit linear series. More precisely, we associate to a family of differentials a limiting object
consisting of a collection of meromorphic differentials parametrised by the irreducible components of the
special curve. For a given component $\X_{i}$, the differential is obtained by rescaling the family in such a
way that it  converges on $\X_{i}$ (see Definition~\ref{definiton:diffLimite}).

To construct examples of limit differentials, we extend the classical plumbing cylinder construction of
curves to the case of differentials (see Lemma~\ref{lemme:PlomberieCylindriqueOk}). In particular, this allows
us to give necessary and sufficient conditions to be a limit differential for an important case (see
Theorem~\ref{theoreme:PlomberieCylindriqueSansResidu}). However, they are not sufficient in full generality
and it remains unclear how to characterise general limit differentials (see nevertheless
Lemma~\ref{lemme:PlomberieCylindriqueAvecResidu} and Lemma~\ref{lemme:noeudsPolaireAvecResiduFaible}).

The second main ingredients are the notions of {\em spin structure} on (semi) stable curves and of {\em Arf
invariant}.  They allow us to generalise the notion of parity of smooth differential to some stable
differentials in Section~\ref{section:Spin}.
In the case of curves of compact type, we associate a canonical spin
structure to a stable pointed differential (see 
Definition~\ref{definition:spinStructureSurCourbeTypeCompact}).
 Using this notion, we show that the
parity of the spin structure above the curves of compact type is invariant by deformations (see 
Theorem~\ref{theoreme:BordDesStratesDansSpin}). 
\begin{theorem}\label{theoreme:BordDesStratesDansSpinIntro}
  Let $n\geq3$ and $(\X,\omega,Z_{1},\cdots,Z_{n})$ be a differential in the closure of the stratum
$\omoduliincp{2l_{1},\cdots,2l_{n}}$. Then the parity of the spin structure~$\mathcal{L}_{\omega}$ associated
to $\omega$ is $\epsilon$ if and only if $(\X,\omega,Z_{1},\cdots,Z_{n})$ is in the closure of
$\omoduliincp{2l_{1},\cdots,2l_{n}}^{\epsilon}$.
\end{theorem}

The notion of spin structure does not seems to be the right one for the irreducible pointed differentials.
However, in
this case, we show that the Arf invariant can be generalised (see
Definition~\ref{definition:extensionInvariantArf}) in such a  way that it stays constant by
deformation (see Theorem~\ref{theoreme:InvDeArfGene}).

It would be very interesting to extend this invariant to the whole boundary of the {incidence variety
compactifications}.  But  we show that, unfortunately, this invariant cannot be extend to the whole {incidence
variety compactification} of the strata (see
Corollary~\ref{corollaire:intersectHypOddGenreTroisIntro}).

\subsection{The Kodaira dimension of strata.}

One of the main motivation for a good compactification of the strata of the moduli space of Abelian
differentials
is the computation of their Kodaira dimensions.
In recent works Farkas and Verra  computed the Kodaira dimension of the moduli space of spin structures
and Bini, Fontonari and Viviani  computed the Kodaira dimension of the universal Picard
variety. They followed the path opened by Harris and Mumford for the moduli space of curves. In
particular, they used in an essential way a nice compactification of these spaces constructed by Cornalba in
the first case and Caporaso in the second.

A second way to compute the Kodaira dimension of algebraic spaces is to use the theory initiated by Iitaka.
We can obtain information about the Kodaira dimension of the total space of an algebraic bundle using
knowledge about the Kodaira dimension of the base and of a generic fibre. 

Using these methods, we want to compute the Kodaira dimension of the strata $S$ of the
moduli space of Abelian differentials for which the forgetful map $\pi:S\to\moduli$ is generically surjective.
We  give a complete description of these strata and, more precisely,  the dimension of the image
of every connected component of each stratum.

\begin{theorem}\label{theoreme:DimensionProjectionStratesIntro}
Let $g\geq2$ and $S$ be a connected component of the stratum $\omoduli(k_{1},\cdots,k_{n})$. The
dimension  of the
projection of $S$ by the forgetful map $\pi:\omoduli\to\moduli$ is 
\begin{equation*}
 \dim\left(\pi( S) \right) =
  \begin{cases}
   2g-1 & \text{if } S=\omoduli(2d,2d)^{\hyp} \\
   3g-4 & \text{if } S=\omoduli(2,\cdots,2)^{\even} \\
   2g-2+n       & \text{if } n< g-1 \text{ and }S\neq\omoduli(2d,2d)^{\hyp} \\
   3g-3   & \text{if } n\geq g-1 \text{ and the parity of $S$ is not even} 
  \end{cases}
\end{equation*}
\end{theorem}

Using this theorem and the fact that the Kodaira dimension of a finite cover is not smaller
than the Kodaira dimension of the base, we deduce the Kodaira dimension of the strata of projective dimension
$3g-3$,  when $\moduli$ is of
general type (see Corollary~\ref{corollaire:dimKodairaStratesProjFini}).
\begin{theorem}
The connected strata $\pomoduli(k_{1},\cdots,k_{g-1})$ are
of general type for $g=22$ and $g\geq 24$.
\end{theorem}

Moreover, for a fibre space $f:X\to Y$  there is the well known inequality
$\kappa(X)\leq\dim(Y)+\kappa(X_{y})$ for a
generic fibre  $X_{y}$ of $f$. This gives the Kodaira dimension of the strata which impose few conditions.
Indeed, by showing that a generic fibre of the forgetful map has negative Kodaira dimension, we obtain the
following result (see Theorem~\ref{theoreme:DimKodairaStraPeutConditions}).
\begin{theorem}
For any  $g\geq2$, let $(k_{1},\cdots,k_{n})$ be a tuple of positive numbers  of the form
$(k_{1},\cdots,k_{l},1,\cdots,1)$ with $k_{i}\geq 2$ for $i\leq l$  such that
$$\sum_{i=1}^{n}k_{i}=2g-2 \text{ and } \sum_{i=1}^{l}k_{i}\leq g-2.$$ 
Then  the Kodaira dimension of the stratum $\pomoduli(k_{1},\cdots,k_{n})$ is $-\infty$.
\end{theorem}

The Iitaka conjecture has been proved by Eckart Viehweg for the fibre spaces $f:X\to Y$, where $Y$ is of
general
type.
So, a similar method could be used to determine the Kodaira dimension of the strata for which the forgetful
map is  generically
surjective to $\moduli$, when $\moduli$ is of general type. However, this method is more subtle for the
remaining strata and  we can only prove that the strata
$\pomoduli(g-1,1,\cdots,1)$ are of general type when  $\moduli$ is of general type (see
Proposition~\ref{proposition:DimKodaira(g-1,1)}). 

To conclude, we compute the Kodaira dimension of both odd (Corollary~\ref{corollaire:dimKodairaOddDeux}) and
even (Proposition~\ref{proposition:dimKodairaEvenDeux}) components of the strata
$\pomoduli(2,\cdots,2)$ and of the hyperelliptic component of $\pomoduli(g-1,g-1)$
(Proposition~\ref{proposition:dimKodairaHypNeg}).

\subsection{Examples.}

We  conclude this work by the explicit description of the {incidence variety compactification} of some strata.
We focus on  the minimal strata $\pomoduli(2g-2)$.  In genus two, there is only one stratum
$\pomoduli[2](2)$ and this stratum has many interpretations. For example, it can be seen as the
Weierstrass divisor in $\moduli[2,1]$ or the moduli space of even spin
structures. 

More generally, the hyperelliptic strata $\pomoduli^{\hyp}(2g-2)$ are  very special and
can be studied with specific tools. They are studied in Section~\ref{section:hyperelliptique} and the
main result is Theorem~\ref{theoreme:isoEntreHypEtWeier} where we show that the fibres of the forgetful map
from the {incidence variety compactification} of $\omoduli^{\hyp}(2g-2)$ to the Weierstrass locus of
hyperelliptic curves inside $\moduli[g,1]$ are projective spaces. 

To be more concrete, let us describe an important locus in the {incidence variety compactification} of the
hyperelliptic
minimal strata (see Theorem~\ref{theoreme:bordHypCasIrr}).
\begin{theorem}\label{theorem:bordHypCasIrrIntro}
Let $\X$ be the union of a smooth curve $\tilde\X$ of genus $g-1$ and a projective line attached to $\tilde\X$
at the points $N_{1}$ and $N_{2}$. 

Then $(\X,\omega,Z)$ is in the {incidence variety compactification}  of the minimal hyperelliptic stratum
$\omoduliincp[g,1]{2g-2}^{\hyp}$ if and only if  the point  $Z$ is in the
exceptional divisor coming from the blow-up, and the differential $\omega$ is the stable differential
with a zero of
order $g-2$ at both $N_{1}$ and $N_{2}$.
\end{theorem}

The first non hyperelliptic minimal stratum is $\pomoduli[3]^{\odd}(4)$.
The description of
the boundary of this stratum gives us the opportunity to illustrate most of the tools developed in this paper.

Let us define a {\em generic curve in the divisor $\delta_{i}$} to be a curve in the divisor
$\delta_{i}$ with a single node. 
The description of the boundary of $\pomoduli[3]^{\odd}(4)$ above the set of curves stably equivalent to
generic curves in $\delta_{0}$
and $\delta_{1}$
is given in Corollary~\ref{corollaire:stabDiffgTroisIrr} and
Corollary~\ref{corollaire:stabDiffgTroisRed}. For example, the  pointed stable differentials in the boundary
of
$\pomoduli[3,1]^{\odd}(4)$ such
that the projection to $\barmoduli[3]$ is stably equivalent to a generic curve of the divisor $\delta_{0}$ is
given by the following theorem.

\begin{theorem}\label{theoreme:bordZeroOrdre4IrrbisIntro}
Let $(\X,\omega,Z)$ be a stable pointed differential in $\PP\obarmoduliinc[3,1]{4}^{\odd}$ such that
$\X$ is the
union of a smooth curve $\tilde\X$ of genus two  and a projective line which meet at two distinct points
$N_{1}$ and $N_{2}$. 

Then $(\X,\omega,Z)$ satisfies that $Z\in \PP^{1}$, the restriction of $\omega$ to $\PP^{1}$ vanishes and the
restriction of $\omega$ to $\tilde\X$ is of one of the following two forms.
\begin{itemize}
 \item  The restriction of $\omega$ to $\tilde\X$ is an
holomorphic differential with a zero of order two at $N_{1}$. 

\item The restriction of $\omega$ to $\tilde\X$ is an
holomorphic differential with two simple zeros at $N_{1}$ and $N_{2}$. 
\end{itemize}
\end{theorem}

This theorem together with Theorem~\ref{theorem:bordHypCasIrrIntro} implies  that the {incidence variety
compactifications} of the hyperelliptic and odd
connected components
of $\pomoduli[3,1](4)$ intersect each other (see Corollary~\ref{corollaire:intersectHypOddGenreTrois}). 
\begin{cor}\label{corollaire:intersectHypOddGenreTroisIntro}
  Let $\X$ be the union of a curve $\tilde{\X}$ of genus two and a projective line glue together at a pair of
points of $\tilde{\X}$ conjugated by the hyperelliptic involution. Let $Z\in E$ and $\omega$ be a
differential
which vanishes on $E$ and has two single zeros at the points which form the nodes on $\tilde{\X}$.

Then the pointed differential $(\X,\omega,Z)$ is in  $\obarmoduliinc[3,1]{4}^{\hyp}$ and
$\obarmoduliinc[3,1]{4}^{\odd}$.
\end{cor}


\section{The {Incidence Variety Compactification} of the Strata of the Moduli Space of Differentials.}
\label{section:CompactDesStrates}

The projectivisation of the Hodge bundle over the moduli space of curves $\pomoduli$ has a natural
compactification given by the moduli space of stable differentials $\pobarmoduli$. The first idea in order to
compactify a stratum is to take its closure inside $\pobarmoduli$. This is called the {\em Deligne-Mumford
compactification } of the stratum. However, this compactification loses lots of information.
To keep track of more information we introduce in Definition~\ref{definition:VarieteeIncidence}
another compactification $\PP\obarmoduliincp{k_{1},\cdots,k_{n}}$ via the closure of the strata inside the
moduli space of marked differentials. This compactification of the strata will be called the {\em
{incidence variety compactification}} of the stratum. The end of this section is devoted to the study of the 
spaces $\PP\obarmoduliincp{k_{1},\cdots,k_{n}}$. We show in
Theorem~\ref{theoreme:ExcesAuBordStratePrincipal} and
Theorem~\ref{theoreme:CompStratePrincDansVarIncEtNormal} that this compactification contains much more
information at the boundary than the one given by the closure inside $\pobarmoduli$.

In this section, all spaces  we consider will be complex orbifolds. 

\paragraph{Background on moduli spaces.}

We begin this section by recalling some basic facts and notations about various moduli spaces.
The {\em moduli space  of curves of genus $g$}, denoted by $\moduli$, is the space of complex structures on a
curve of genus $g$. The {\em moduli space of $n$-pointed curves} is denoted by $\modulin$.
It is well known since Riemann (see for example \cite{MR1288523}) that the dimension of  $\modulin$ is
$3g-3+n$.

A modular compactification of $\modulin$ is given by the
 {\em moduli space $\barmoduli[g,n]$ of $n$-marked stable curves}. This compactification is called the {\em
Deligne-Mumford compactification} of the moduli space of $n$-marked curves.
Recall that a {\em stable curve}  is a connected  nodal curve for which each
irreducible component of the normalisation has not an Abelian fundamental
group.
The {\em dual graph} of a stable curve $\X$ of genus $g$, denoted by $\dualgraph(\X)$, is the
weighted graph such that the vertices  correspond to the irreducible components of
$\X$, the edges correspond to its nodes and the weight at a vertex is given by the geometric genus of
the corresponding component.

 The {\em moduli space of nonzero holomorphic 1-forms $\omoduli$} or {\em Hodge bundle of~$\moduli$} 
parameterises pairs  $(\X,\omega)$, where $\X$ is a smooth curve of genus $g$
and~$\omega$ is a nonzero holomorphic 1-form on $\X$.
Remark that the space $\omoduli$ is sometimes denoted by $\zomoduli[g]$ in the literature (for example
\cite{MR2261104}, \cite{MR2010740},...). We will never use this notation due to the risk of confusion
with the notation of the hyperelliptic locus inside $\moduli$ (see Section~\ref{section:hyperelliptique}).

The space $\omoduli$ has a natural stratification by the multiplicities of zeros of $\omega$.
Let $(k_1,\cdots, k_n)$ be a $n$-tuple of strictly positive numbers such that $\sum_{i=1}^{n}k_i = 2g-2$. The
stratum $\omoduli(k_1,\cdots,k_n)$ is the subspace of
$\omoduli$ consisting of
equivalence pairs $(\X,\omega)$, where $\omega$ has  $n$ distinct zeros of respective orders $(k_1,\cdots,
k_n)$.
In particular, for $g\geq2$ the following decomposition holds (see for example \cite{MR2261104}):
\begin{equation}\label{equation:Strate1}
 \omoduli = \bigsqcup_{\substack{n\in\lbrace 1,\cdots,2g-2 \rbrace\\ 2g-2\geq k_1\geq\cdots\geq k_n\geq 1,\
\sum k_i=2g-2}} 
\omoduli(k_1,\cdots,k_n).
\end{equation}

 The notion of differentials extends to the case of (semi) stable curves in the following way.
 A {\em stable  differential} on a  stable curve $\X$ is a meromorphic 1-form~$\omega$ on $\X$  which
is holomorphic outside of the nodes of $\X$ and has at worst simple poles at the
nodes and the two residues at a node are opposite.
Alternatively, the  stable  differentials could be defined as the global sections of the {\em dualizing sheaf}
$\dualsheave$ of $\X$ (see \cite{MR1631825}).
We can now extend the Hodge bundle $\obarmoduli$ above $\barmoduli$.
 The space $\obarmoduli$ is the {\em moduli space of stable differentials} of genus $g$.

Since the definition of stable differential extends readily to the case of semi stable curves, we can extend
this notion to the case of stable marked curves.
\begin{defn}\label{definition:DiffMarquee}
 A {\em marked  stable differential $(\X,\omega,Q_{1},\cdots,Q_{n})$ of genus $g$} is the datum of a stable
$n$-marked curve  $(\X,Q_{1},\cdots,Q_{n})$ in $\barmoduli[g,n]$ and  a  stable differential $\omega$ on
$\X$. 
\end{defn}
The {\em moduli space of marked stable differentials} will be denoted by  $\obarmoduli[g,n]$. It is the Hodge bundle
above the moduli space of marked curves $\barmoduli[g,n]$. Its restriction to the locus of smooth $n$-marked
curves is the {\em moduli space of $n$-marked Abelian differentials} and is denoted by $\omoduli[g,n]$.

There is a natural $\CC^{\ast}$-action on the moduli space of Abelian differentials given by
\begin{equation}
\CC^{\ast}\times\omoduli\to\omoduli:(\alpha,(\X,\omega))\mapsto(\X,\alpha\omega).
\end{equation} 
The quotient of $\omoduli$ under this action is denoted by $\pomoduli$. Remark that this action preserves the
stratification of $\omoduli$ and the images of  $\omoduli(k_{1},\cdots,k_{n})$ inside $\pomoduli$
are well defined and are denoted by $\pomoduli(k_{1},\cdots,k_{n})$.
Moreover, the group $\CC^{\ast}$ acts in a similar way on $\omoduli[g,n]$ and we denote the quotient under
this action by $\pomoduli[g,n]$.

\paragraph{The {Incidence variety compactification} of the strata of $\omoduli$.}

In order to compactify the strata of $\omoduli$, we define the {\em ordered incidence variety}
$\PP\omoduliinc{k_{1},\cdots,k_{n}}$ to be
the subspace of the moduli space of $n$-marked differentials given by
\begin{equation}
 \left\{(\X, \omega, Z_{1},\cdots,Z_{n}) :\ 
\divisor{\omega} = \sum_{i=1}^{n}k_{i}Z_{i}  \right\}\subset \PP\omodulin.
\end{equation}
Moreover, the {\em closed ordered incidence variety}, denoted by $\PP\obarmoduliinc{k_{1},\cdots,k_{n}}$,
is defined as the closure of the ordered incidence variety inside 
$\pobarmoduli[g,n]$.

 In general, there exists a subgroup of $\permGroup$ acting non-trivially on the closed ordered incidence
variety
$\PP\obarmoduliinc{k_{1},\cdots,k_{n}}$.
 Namely, if $k_{i}=k_{j}$ for $i\neq j$, then the transposition
$(i,j)$ acts  on  $\PP\obarmoduliinc{k_{1},\cdots,k_{n}}$ by permuting the points $Z_{i}$ and $Z_{j}$. Let
$\mathfrak{S}$
be the subgroup of $\permGroup$ generated by these transpositions. It is easy to see that
$\mathfrak{S}\isom\prod\permGroup[l_{i}]$, where $l_{i}:=\#\left\{j| k_{j}=i\right\}$ is the number of
indices $j$
such that the order $k_{j}$ is equal to $i$.

\begin{defn}\label{definition:VarieteeIncidence}
Let $\omoduli(k_{1},\cdots,k_{n})$ be a stratum of $\omoduli$ and let $S$ be one of its connected components.
 The {\em {incidence variety compactification} of $S$} is 
 \begin{equation}
 \PP\obarmoduliincp{k_{1},\cdots,k_{n}}:=\PP\obarmoduliinc{k_{1},\cdots,k_{n}}/\mathfrak{S}.
 \end{equation}

A triple $(\X,\omega,Z_{1},\cdots,Z_{n})\in\PP\obarmoduliincp{k_{1},\cdots,k_{n}}$ will be called a {\em
pointed
differential} or a {\em pointed flat surface}.
\end{defn}

Remark that the notions of pointed differentials and marked differentials (see
Definition~\ref{definition:DiffMarquee}) do not coincide.

Let us remark that the closed ordered incidence variety is a suborbifold of $\pobarmoduli[g,n]$. 
Therefore the {incidence variety compactification} of every stratum is an orbifold as the quotient of an
orbifold by
a finite group.

\paragraph{The forgetful map.}

There is a natural {\em forgetful map} between the {incidence variety compactification} and the corresponding
stratum. Before
defining this map on the whole compactification, we  restrict ourself to its restriction above the smooth
pointed differentials. This restriction is given by
\begin{align*}
\varphi:\PP\omoduliincp{k_{1},\cdots,k_{n}} &\to \PP\omoduli[g](k_{1},\cdots,k_{n})\\
  (\X,\omega,Z_{1},\cdots,Z_{n}) & \mapsto (\X,\omega).
\end{align*}

This map turns out to be an isomorphism.
\begin{lemma}\label{lemme:IsoEntreIncEtStrate}
The forgetful map 
\begin{equation}
\varphi:\PP\omoduliincp{k_{1},\cdots,k_{n}} \to \PP\omoduli[g](k_{1},\cdots,k_{n})
\end{equation}
 is an isomorphism of orbifolds.
\end{lemma} 

In particular, this lemma clearly implies that the dimension of the {incidence variety compactification} 
$\PP\obarmoduliincp{k_{1},\cdots,k_{n}}$ is  $2g-2+n$.

\begin{proof}
It suffices to show that there exists an inverse $\psi$ to $\varphi$. Let
$(\X,\omega)$ be a smooth differential with zeros of order
$(k_{1},\cdots,k_{n})$. We denote by $Z_{1},\cdots,Z_{n}$ the corresponding
zeros. 

Let us define the map 
\begin{align*}
\tilde\psi:\PP\omoduli[g](k_{1},\cdots,k_{n})&\to
\PP\omoduliinc{k_{1},\cdots,k_{n}}\\
(X,\omega)&\mapsto(X,\omega,Z_{1},\cdots,Z_{n}).
\end{align*}
 We define
$\psi$ by the composition of $\tilde\psi$ with the quotient by the action of $\mathfrak{S}$.
It is a routine to prove that both maps are inverse to each other.
\end{proof}

We extend the map $\varphi:\PP\omoduliincp{k_{1},\cdots,k_{n}} \to
\PP\omoduli[g](k_{1},\cdots,k_{n})$ at
the boundary of the strata.
Let $(\X',\omega',Z_{1},\cdots,Z_{n})\in\PP\obarmoduliincp{k_{1},\cdots,k_{n}}$ be a pointed differential. We
denote by $\X$ the image of $\X'$ by
the forgetful map $\pi:\barmoduli[g,n]\to\barmoduli[g]$. Moreover, for every irreducible component $\X_{i}$ 
of $\X$, the corresponding irreducible component of $\X'$ is denoted by $\X_{i}'$.
We obtain a differential $\omega$ on $\X$ in the following way. The restriction of $\omega$ on every
irreducible component $\X_{i}$ of $\X$ is the differential $\omega'|_{\X_{i}'}$. 

This is clearly an extension of the forgetful map $\varphi$, and it remains to show that the image of this
extension lies in $\pobarmoduli$.

 \begin{lemma}\label{lemme:ProlongAppOubli}
The forgetful map 
\begin{align*}
\varphi:\PP\obarmoduliincp{k_{1},\cdots,k_{n}}&\to\PP\obarmoduli[g](k_{1},\cdots,k_{n})\\
 (\X',\omega' , Z_ {1},\cdots,Z_{n})&\mapsto(\X,\omega)
\end{align*}
described in the preceding paragraph is well defined. More precisely, the pair $(\X,\omega)$ is a stable
differential.
\end{lemma}

\begin{proof}
 The forgetful map $\barmoduli[g,n]\to\barmoduli$ is well defined. Hence it is enough to show that the
differential $\omega$ is stable.

Let $E$ be an exceptional component of $\X'$ and
$\X'_{E}$ be the curve obtained from $\X'$ by blowing down $E$. We denote by $\omega_{E}'$ the restriction of
the form $\omega'$ on $\X'_{E}$. We can suppose that the nodal points of $E$ are $0$ and $\infty$.

 Then the restriction of the form $\omega'$ on $E$ is either zero or of the
differential~$\frac{a\dz}{z}$ for some $a\in\CC^{\ast}$.

 If the restriction is zero, then it is clear that $\omega_{E}'$ is
stable on $\X_{E}'$. 

If the differential on $E$ is given $\frac{a\dz}{z}$. Since the residues of the differential at
the nodes have to sum up to zero, the residues at the  points of the node are $-a$ and $a$. This implies
that $\omega'_{E}$ is still a stable differential on $\X'_{E}$.

If we take $E_{1}$ and $E_{2}$ two exceptional components, we can easily verify that
$\X'_{E_{1}E_{2}}=\X'_{E_{2}E_{1}}$ and $\omega_{E_{1}E_{2}}=\omega_{E_{2}E_{1}}$. So by
induction on the set of exceptional components, the limit $\omega$ is a well defined stable differential on
$\X$. And it is easy to check that this limit coincides with the form given by the map $\varphi$.
\end{proof}

\paragraph{Closure of the principal stratum.} 
In this paragraph, we show that the {incidence variety compactification} contains much more
information than the Deligne-Mumford compactification at the boundary of the  principal stratum. This part
uses the results of Section~\ref{section:PlomberieCylindrique} and in particular
Theorem~\ref{theoreme:PlomberieCylindriqueSansResidu}.

\begin{theorem}\label{theoreme:ExcesAuBordStratePrincipal}
Let $(\X,\omega)\in\pomoduli(2g-2)$ be a differential in the minimal stratum. This differential is in the
boundary of principal stratum
$\PP\omoduli(1,\cdots,1)$  and the
dimension of the fibre  of the forgetful
map $$\pi:\PP\obarmoduliinc[g,\left\{ 2g-2 \right\}]{1,\cdots1}\to\PP\obarmoduli[g](1,\cdots,1)$$ above
$(\X,\omega)$ is
$\max(0,2g-4)$.
\end{theorem}

\begin{proof}
 Let $(\X,\omega,Z)\in\PP\omoduliinc[g,1]{2g-2}$ be a pointed differential of genus $g$ and
$(\PP^{1},Q_{1},\cdots,Q_{2g-2},P)$
be a marked rational curve. There exists a meromorphic differential with a single zero at all the $Q_{i}$ and
a pole of order $2g$ at $P$. Indeed, this differential is given up to scalar multiplication by
$$\eta:=\frac{\prod\limits_{i}(z-Q_{i})}{(z-P)^{2g}}\dz.$$
Let us glue the curve $\X$ with this rational curve via the identification of $Z$ with~$P$. It is easy to
verify that we can apply
Theorem~\ref{theoreme:PlomberieCylindriqueSansResidu} in order smooth this differential. The differential
that we obtain has $2g-2$ simple zeros.  This shows that the pointed differential
$$\left(\X\cup\PP^{1}/Z\sim
P,(\omega,0),Q_{1},\cdots,Q_{2g-2}\right)$$ is an element of $\PP\obarmoduliincp[g,\left\{ 2g-2
\right\}]{1,\cdots,1}$ for any tuple
 $(Q_{1},\cdots,Q_{2g-2},P)$. A simple dimension count concludes the proof.
\end{proof}

We now prove an analogous result for the absolute boundary of the stratum
$\PP\obarmoduliinc[g,\left\{ 2g-2 \right\}]{1,\cdots1}$ for curves in the divisor $\delta_{i}$, for $i\geq 2$.

\begin{theorem}\label{theoreme:CompStratePrincDansVarIncEtNormal}
The fibre of the forgetful map  
$$\pi:\PP\obarmoduliinc[g,\left\{2g-2\right\}]{1,\cdots1}\to\PP\obarmoduli[g](1,\cdots,1)$$ is positive
dimensional
over a differential $(\X,\omega)$, where $\X$ is a generic curve in~$\delta_{i}$ for $i\geq 1$ and $\omega$
vanishes on one component of $\X$.
\end{theorem}

\begin{proof}
Let $(\X:=\X_{1}\cup\X_{2}/N_{1}\sim N_{2},\omega)$ be  a differential  of genus $g\geq2$ in
$\pobarmoduli(1,\cdots,1)$ and suppose that $\omega|_{\X_{1}}=0$. Then, the component $\X_{1}$ contains more
than $2g_{1}-2$ marked points.  
The map $h:\X_{1}^{(k)}\to\Jac[\X_{1}]$ from the symmetric product of $\X_{1}$ to the Jacobian of $\X_{1}$
given by
$$ (Q_{1},\cdots,Q_{k})\mapsto \Ox[\X_{1}]\left(\sum_{i} Q_{i} -
(k-2g_{1}+2)N_{1}\right)$$ 
is surjective. Hence the dimension of the fibre of $\pi$ at $(\X,\omega)$ is at least $k-g_{1}$. Such divisors are canonical and
since there is no residue
at $N_{1}$, we apply Theorem~\ref{theoreme:PlomberieCylindriqueSansResidu} to conclude that every such
differential can be smoothed in $\pomoduli(1,\cdots,1)$.
\end{proof}

We are going to present some other results about the closure  of the minimal hyperelliptic strata in
Section~\ref{section:hyperelliptique} and of the closure of $\pomoduli[3,1]^{\odd}(4)$ in
Section~\ref{section:omoduli3,4,odd}.


\section{Limit Differentials and Plumbing Cylinders.}\label{section:PlomberieCylindrique}

In order to remedy the disadvantage of stable differentials that may vanish on some components, we
introduce the notion of {\em limit differential}. It is, in a sense,  similar to the notion of limit linear
series, but for families of pointed differentials in a stratum (see Definition~\ref{definiton:diffLimite}).
In particular, a limit
differential is a collection of differentials parametrised by the set of irreducible components of a marked
curve (such collection will be called {\em candidate differential}). None differential of this collection
identically vanish, but the price to pay is to allow some poles of order greater than one
at the nodes.

This notion is interesting only if the following conditions are satisfied. First, this notion should be manageable, at least
for important cases.
In particular,  we should be able to exhibit limit differentials. To produce  examples, we
extend the classical
plumbing cylinder construction from the case of curves to the case of
differentials (see Lemma~\ref{lemme:PlomberieCylindriqueOk}). This allows us to give necessary and
sufficient conditions for being a limit differential in an important case (see
Theorem~\ref{theoreme:PlomberieCylindriqueSansResidu}). 
Two of these conditions are easily stated: 
 a limit differential $(\X,\omega)$ must satisfy the {\em compatibility
condition} at
every node $N_{i}$ of $\X$
\begin{equation*}
\ord_{N_{i,1}}(\omega)+ \ord_{N_{i,2}}(\omega) =-2,
\end{equation*}
and the {\em residue condition} at every node where $\omega$ has simple poles
\begin{equation*}
\Res_{N_{i,1}}(\omega)+\Res_{N_{i,2}}(\omega) = 0.
\end{equation*}
We refer to Theorem~\ref{theoreme:PlomberieCylindriqueSansResidu} for the other conditions.
The general case is much more complicated, and we show
that the above conditions are not sufficient. However, see Lemma~\ref{lemme:noeudsPolaireAvecResiduFaible} for
a necessary condition and Lemma~\ref{lemme:PlomberieCylindriqueAvecResidu} for a sufficient one (both of them
being non-optimal).

Second, it should be possible to deduce information on the {incidence variety compactification} of the strata
of $\omoduli$ from the limit differentials. We connect the two notions for important cases in
Proposition~\ref{proposition:relationPlumStable}. We will use this relationship intensively in
Section~\ref{section:hyperelliptique} and Section~\ref{section:omoduli3,4,odd}.

\paragraph{Limit Differentials.}
Before defining the notion of limit differential, we prove a preliminary result about families of pointed
differentials. This allows us to introduce the notion of {\em scaling}.

\begin{lemma}
 Let $$\left(f:\famcurv\to\Delta^{\ast},
\famomega:\Delta^{\ast}\to\dualsheave[\famcurv/\Delta^{\ast}],\seczero_{1},\cdots,\seczero_{n}:\Delta^{\ast}
\to\famcurv \right)$$ be a family of pointed differentials inside the stratum
 $\obarmoduliinc{k_{1},\cdots,k_{n}}$ and let $(\X,\omega,Z_{1},\cdots,Z_{n})$ be its
stable limit.  Then, for every irreducible component $\X_{i}$ of
$\X$ there exists a unique $r_{i}\in\ZZ$ such that for a generic section $s:\Delta^{\ast}\to\famcurv$ with
$\bar{s}(0)\in\X_{i}$ we have 
\begin{equation}
\lim_{t\to 0} t^{r_{i}}\famomega\left(t,s(t)\right) \neq 0.
\end{equation}
Moreover, every map $\alpha_{i}:\Delta\to\CC$ satisfying this property is given by
$$t^{r_{i}}(1+t\CC\left[t\right]).$$
The map $t^{r_{i}}$ is called the {\em scaling} of the component $\X_{i}$ for this family. 

The stable limit of the family of differentials is given by $$\lim_{t\to 0}
\left(\alpha(t)\famomega\left(t\right)\right),$$
where $\alpha$ is a scaling such that for every scaling $\alpha_{i}$ the quotient $\alpha/\alpha_{i}$ is
bounded at the origin.
\end{lemma}

Let $\X$ be a (semi) stable curve, we denote by $\Irr(\X)$ the set of irreducible components of $\X$.
\begin{proof}
 Let us define the meromorphic map
$$h:\famcurv\to\CC,(t,x)\mapsto \famomega(t,x),$$
where $\famomega$ is seen as a section of $\Ox[\famcurv]\left( \sum k_{i}\seczero_{i} \right)$. In
particular, the map $h$ is of the form $h(x,t)=h(t)$, where $h$ is never vanishing.
Let us denote its meromorphic continuation on $\bar\famcurv$ by $\bar h$. The divisor of $\bar h$ is of the
form $$\divisor{\bar h} = \sum_{\X_{i}\in\Irr(\X)}l_{i}\X_{i},$$
where $l_{i}\in\ZZ$.
This implies that $\alpha_{i}:=t^{-l_{i}}$ is a scaling for $\X_{i}$. 
The uniqueness and the general description of the map $\alpha_{i}$ having this property clearly follows from
this description.

Now, let $\alpha$ be a map such that $\lim\limits_{t\to 0}
\left(\alpha(t)\famomega(t)\right)$ is a non vanishing stable differential and $\alpha_{i}$ be the scaling of
any component of $\X$.
By definition $\alpha$ is the scaling of some component of $\X$. 
Let us  show that the  quotient $\frac{\alpha}{\alpha_{i}}$ is
bounded in a neighbourhood of $0$. For any section $s:\Delta^{\ast}\to\famcurv$ we have the
equality $$\alpha(x)\famomega(t,s(t))=\frac{\alpha (t)}{\alpha_{i}(t)}\alpha_{i}(t)\famomega(t,s(t)).$$ 
Hence, if $\frac{\alpha}{\alpha_{i}}$ is not bounded at the origin, then the limit is not bounded on the
smooth
locus of
$\X_{i}$. In particular, the limit is not a stable differential.
\end{proof}

Now we introduce the notion of limit differential of a family of pointed differentials. 
\begin{defn}\label{definiton:diffLimite}
 A {\em limit differential $(\X,\omega,Z_{1},\cdots,Z_{n})$ of type $(k_{1},\cdots,k_{n})$} is a tuple
 such that there exists
a family of pointed differentials $$\left(f:\famcurv\to\Delta^{\ast},
\famomega:\Delta^{\ast}\to\dualsheave[\famcurv/\Delta^{\ast}],\seczero_{1},\cdots,\seczero_{n}:\Delta^{\ast}
\to\famcurv \right)$$ inside $\omoduliincp{k_{1},\cdots,k_{n}}$ which satisfies the two following properties. 

First, the marked curve $(\X,Z_{1},\cdots,Z_{n})$ is the stable
limit of the family
$\left(\famcurv,\seczero_{1},\cdots,\seczero_{n}\right)$. 

Second, for every  irreducible component $\X_{i}$
of the curve $\X$ and for every section $s:\Delta^{\ast}\to\famcurv$, we have $$\lim_{t\to 0}
\alpha_{i}(t)\famomega(t,s(t)) =
\omega(\bar{s}(0)),$$
where $\alpha_{i}$ is the scaling of $\X_{i}$.

The set of limit differentials of type $(k_{1},\cdots,k_{n})$ modulo the usual action of the group
$\mathfrak{S}\subset\permGroup$ (see Section~\ref{section:CompactDesStrates})
is denoted by $\kbarmodulilim(k_{1},\cdots,k_{n})$.
\end{defn}

A limit differential is a collection of never identically zero meromorphic differentials
parametrised by the set of irreducible components of
a stable marked curve. Moreover, the sum of the smooth parts of the divisors of these differentials is given
by $\sum k_{i}Z_{i}$.
 In order to avoid confusion with the stable pointed differentials, we call
such objects {\em candidate differentials} of type $(k_{1},\cdots,k_{n})$.

We now give necessary conditions for a candidate differential to be a limit differential.  

\begin{lemma}\label{lemme:CondCompPlomberieCyl}
 Let $(\X,\omega,Z_{1},\cdots,Z_{n})$ be a limit differential and $N_{1}\sim N_{2}$ be a node of the curve
$\X$. Then the differential $\omega$
satisfies the  {\em Compatibility Condition}
\begin{equation}\label{equation:conditionDeCompatibiliteGeneral}
 \ord_{N_{1}}(\omega)+ \ord_{N_{2}}(\omega)=-2.
\end{equation}

Moreover, if the orders of $\omega$ at $N_{1}$ and $N_{2}$ are $-1$, then the differential $\omega$ satisfies
the {\em Residue Condition}
\begin{equation}\label{equation:conditionDeCompatibiliteGeneral2}
 \Res_{N_{1}}(\omega)+ \Res_{N_{2}}(\omega)=0.
\end{equation}
\end{lemma}

\begin{proof} 
 Let $(f:\famcurv\to\Delta^{\ast},\famomega,\seczero_{1},\cdots,\seczero_{n})$ be a family of
pointed differentials which
converges to the limit differential $(\X,\omega,Z_{1},\cdots,Z_{n})$. Let $U$ be  a neighbourhood of the node
$N_{1}\sim N_{2}$ in $\bar\famcurv$.
Without loss of generality, we can assume that $U$ satisfies the following properties. First, the
intersections
$\seczero_{i}\cap U$ are empty for every $i\in\lbrace 1,\cdots,n \rbrace$. In particular, the only possible
zeros and
poles of $\famomega|_{U}$ are contained in $\X|_{U}$. Second, there exists a coordinate system
$(x,y,t)$ of an open subset of $\Delta^{3}$ containing the origin such that
\begin{equation}\label{equation:noeudGeneral}
 U:=\left\{  xy= t^{a}\right\},
\end{equation}
where $a\geq1$. Moreover, we can suppose
that $\X|_{U}$ is given by the equation $\lbrace xy=0 \rbrace$.  In the rest of the proof, we denote by
$\X_{x}$,
$\X_{y}$ and $\X_{U}$ the subset of~$U$ of respective equations $\left\{y=0\right\}$, $\left\{x=0\right\}$
and $\left\{xy=0\right\}$.

We pick a differential $\eta$
that generates $\holoneform[U] / f^{\ast}(\holoneform[\Delta])$ and that vanishes
nowhere on $U$, for example
 $$\eta:=\frac{x\dx-y\dy}{x^{2}+y^{2}}.$$ For
$t\neq 0$, its restriction to the curve $\famcurv_{t}$ is a differential without zeros or poles.
For $t=0$, its  restriction to the component $\X_{x}$ (resp. $\X_{y}$) has a unique
simple pole at $N_{1}$ (resp. $N_{2}$) with residue $1$ (resp. $-1$).

Since $\eta$ generates  $\holoneform[U] / f^{\ast}(\holoneform[\Delta])$, the family of
differentials $\famomega|_{U\setminus\X_{U}}$ is given by 
\[\famomega=h\cdot\eta,\]
where $h$ is a meromorphic function with neither poles nor zeros in $U\setminus\X_{U}$. By multiplying
the function $h$ by a power of $t^{a}$, we obtain a new family of differentials proportional to
$\famomega$
on $U\setminus\X_{U}$. In particular, we can suppose that $h$ is holomorphic on $U$ and vanishes
 on at most one component of $\X_{U}$. This new family will still be denoted by $\famomega$ and the
holomorphic function by $h$.

We have two cases to consider. The first one is the case where $h$ is invertible on $U$. In this case the
limit
differential of $\famomega$ on $\X_{U}$ is simply a scaling of the restriction of $\eta$ on $\X_{U}$. Hence
the residues of $\omega$ at $N_{1}$ and $N_{2}$ are respectively~$h(0)$ and $-h(0)$. In particular, in this
case, both the compatibility and the residue conditions are satisfied.

The second case is where $h$ vanishes on one component. Without loss of generality, we can suppose that
$h|_{\X_{y}}\equiv 0$ and $h|_{\X_{x}}\not\equiv0$. By the Weierstrass preparation
theorem, the function $h$ can be written as
\begin{equation}
 h(x,y)=\left(x^{d} + h_{1}(y)x^{d-1}+\cdots + h_{d}(y)\right)\tilde{h}(x,y),
\end{equation}
where $\tilde{h}$ is invertible and the $h_{i}$ are holomorphic maps vanishing at the origin.
Moreover, since by hypothesis the divisor of $h$ is a multiple of $\X_{y}$, we deduce that the
functions $h_{i}$ are identically zero. Hence the function $h$ is of the form
 \begin{equation}\label{equation:formeLocaleFamilleDeDiffAuxNoeuds}
 h(x,y)=x^{d}\cdot \tilde{h}(x,y).
\end{equation}
This implies that restriction $\omega_{x}$ of $\omega$ to the component $\X_{x}$ is given by
\[
\left.\left( x^{d}\cdot \tilde{h}(x,y)\cdot\frac{x\dx-y\dy}{x^{2}+y^{2}}\right)\right|_{\X_{x}}
=x^{d}\cdot\tilde{h}(x,0)\frac{\dx}{x}.
\] By rescaling the family of differentials  $\famomega$ by the function $(t^{a})^{-d}$, we find
that the restriction $\omega_{y}$ of $\omega$ to the component $\X_{y}$ is given by 
\[y^{-d}\cdot\tilde{h}(0,y)\frac{-\dy}{y}. \]
In particular, since $\tilde{h}(0,0)\in\CC^{\ast}$, the sum of the orders of $\omega_{x}$ and $\omega_{y}$ at
the
origin is $-2$.
\end{proof}

It is convenient to formulate a byproduct of our proof as a
separate lemma.
\begin{lemma}\label{lemme:relationScaleAuNoeud}
Let $(\famcurv,\famomega,\seczero_{1},\cdots,\seczero_{n})$ be a family of pointed differentials  which
converges to the limit differential $(\X,\omega,Z_{1},\cdots,Z_{n})$. Let $N$ be a node between the
irreducible components $\X_{i}$ and $\X_{j}$ (which may coincide), and suppose that the equation of
$\famcurv$ around $N$ is $xy=t^{a}$ for some $a\geq 1$. 

If $\ord_{N}(\omega|_{\X_{i}})=k\geq
-1$, then the scaling $\alpha_{i}$ and $\alpha_{j}$ of $\X_{i}$ and $\X_{j}$ satisfy the equality
\begin{equation}\label{equation:relationScaleAuNoeud}
 \frac{\alpha_{i}}{\alpha_{j}}=(t^{a})^{k+1}.
\end{equation}
\end{lemma}

As an application we can prove that a limit differential $(\X,\omega,Z_{1},\cdots,Z_{n})$ is uniquely
determined up to multiplicative constants by $(\X,Z_{1},\cdots,Z_{n})$. 
\begin{cor}\label{corollaire:uniciteLimDiffSurTypeCompacte}

 Let $(\X,Z_{1},\cdots,Z_{n})$ be a marked curve in the image of the
{incidence variety compactification} $\obarmoduliincp{k_{1},\cdots,k_{n}}$ by the forgetful map. Then there
exists a limit differential on $(\X,Z_{1},\cdots,Z_{n})$ of type
$(k_{1},\cdots,k_{n})$. Moreover for any
 two of such limit differentials $\omega$ and $\omega'$ there exist  constants $c_{i}\in\CC^{\ast}$ such that 
$$\left.\frac{\omega}{\omega'}\right|_{\X_{i}}=c_{i},$$
for every irreducible component $\X_{i}$ of $\X$. 
\end{cor}

\begin{proof}
Let $\X_{i}$ be an irreducible component of $\X$ which corresponds to a leaf of the dual graph of $\X$. Let
$Z_{i,1},\cdots,Z_{i,n_{i}}$ be the marked points in $\X_{i}$. Then the restriction of $\omega$ to $\X_{i}$
has  zeros of order $k_{i,j}$ at $Z_{i,j}$ and at most one other zero or a unique pole which has to be located
at the
node of $\X_{1}$. Moreover, the order at the node is imposed by the fact that the degree of $\omega|_{\X_{i}}$
is $2g_{i}-2$. Hence~$\omega|_{\X_{i}}$ is uniquely determined up to a multiplicative constant.

Now we continue this process on the irreducible components adjacent to the preceding components. The order at
the nodes with the previous components are determined by the
compatibility condition~\eqref{equation:conditionDeCompatibiliteGeneral} and  the order at the marked
points $Z_{l}$ are $k_{l}$. Hence it follows that the order at the last node is imposed by the condition on
the
degree of $\omega$. 

Iterating this process we show that there is at most one limit differential (up to multiplication) on
$(\X,Z_{1},\cdots,Z_{n})$. And since $(\X,Z_{1},\cdots,Z_{n})$ lies in the projection of
$\obarmoduliincp{k_{1},\cdots,k_{n}}$, there exists at least one limit differential on this curve.
\end{proof}

There is a global obstruction to smooth a candidate differential which satisfies the compatibility condition
and the residue condition.Let us
look first at a very simple example.

\begin{ex}\label{exemple:ObstructionPlomberie1}
Let  $\X$ be irreducible with one node and the differential $\omega$ has a zero of
order $k$ and a pole of order $k+2$ at the node. It follows from 
Lemma~\ref{lemme:relationScaleAuNoeud} that the differential cannot be smoothed. Indeed, the scaling of an
irreducible component is unique for a given family of differentials. But in this case, by
Lemma~\ref{lemme:relationScaleAuNoeud}, the scaling $\alpha$ of $\X$ satisfies
$\frac{\alpha}{\alpha}=(t^{a})^{k+1}$ for an $a\geq1$, which is absurd.
\end{ex}

Let us now introduce some definitions.

\begin{defn}\label{definition:ordreNoeds}
Let $(\X,\omega)$ be a candidate differential and $N$ a node of $\X$. The {\em order} of $N$ relatively to
$\omega$ is 
$$\ord(N):=\underset{i=1,2}{\max}(\ord_{N_{i}}(\omega)).$$ 
\end{defn}

\begin{defn}\label{definition:GrapheDualDiffLimites}
 Let $(\X,\omega,Z_{1},\cdots,Z_{n})$ be a candidate differential. 
The {\em dual graph} $\Gamma_{\omega}$ of $(\X,\omega)$ is the partially directed weighted graph given by the
following data.
\begin{itemize}
 \item The graph coincides with the dual graph of $\X$.
 \item An edge is directed from the component with the zero to the component with the pole
of $\omega$ and no orientation in the case of simple poles.
 \item The weight $w(e)$ of an edge $e$ is one greater than the order of the  corresponding node (see
Definition~\ref{definition:ordreNoeds}).
\end{itemize}
 \end{defn}
 
 \begin{ex}
 The graphs of the curves of Example~\ref{exemple:ObstructionPlomberie1}  and
Example~\ref{exemple:ObstructionPlomberie2} are drawn in Figure~\ref{figure:exempleGrapheDualPlomberie}.
 \begin{figure}[ht]
 \centering
\begin{tikzpicture}[shorten >=1pt,auto,node distance=2cm,]

     \node(A) at (1,0)[circle,draw]{$g-1$};
     \node(B) at (3.5,0)[circle,draw]{$1$};
     \node(C) at (5.5,1.5)[circle,draw]{$0$};
     \node(D) at (7.5,0)[circle,draw]{$1$}; 
         
     \draw[loop](A) to node[above]{$k+1\geq 1$} (A) ;

     \draw[->](B) to node[]{$1$} (C);
     \draw[->](D) to node[above]{$1$} (C);

\end{tikzpicture}
\caption{}
\label{figure:exempleGrapheDualPlomberie}
\end{figure}
 \end{ex}

\begin{defn}
 Let $\Gamma$ be an partially oriented graph. A {\em path} $\gamma$ is a finite continuous sequence of pairs
$\left\{(e_{i},\alpha_{i})\right\}_{i\in\left\{1,\cdots,l\right\}}$, where $e_{i}$ is an edge of $\Gamma$ and
$\alpha_{i}\in\left\{0,\pm 1\right\}$ is $0$ if the edge has no orientation, $1$ if the direction coincide
with the orientation of $e_{i}$ and $-1$
otherwise. Such path $\gamma$ will be denoted by $$\gamma:=\sum_{i=1}^{l}\alpha_{i}e_{i}.$$
\end{defn}

We now give another property which is satisfied by the limit differentials. Let us recall that $\nodeSet$
denotes the set of nodes of a curve $\X$.

\begin{lemma}\label{lemme:ConditionPlomberieCheminsFermer}
 Let $(\X,\omega,Z_{1},\cdots,Z_{n})$ be a limit differential.
 There exists a tuple
$(\epsilon_{1},\cdots,\epsilon_{r})\in(\Delta^{\ast})^{\nodeSet}$ such that for every closed path
 $\gamma=\sum_{i=1}^{l}\alpha_{i}e_{i}$ in the dual graph of $(\X,\omega)$
the equation
\begin{equation}\label{equation:paramCylindrique}
\prod_{i=1}^{l}\epsilon_{j_{i}}^{\alpha_{i}w(e_{i})}=1
\end{equation}
  is satisfied, where the node corresponding to $e_{i}$ is $N_{j_{i}}$.
\end{lemma}

\begin{proof}
Let $(\famcurv,\famomega,\seczero_{1},\cdots,\seczero_{n})$ be a family converging to the limit differential
$(\X,\omega,Z_{1},\cdots,Z_{n})$. Let $\gamma=\sum_{i=1}^{l}\alpha_{i}e_{i}$ be a closed path in the dual
graph
$\Gamma_{\omega}$ of the limit
differential $(\X,\omega)$, starting at
the vertex $v_{1}$  and ending at the
vertex $v_{l+1}=v_{1}$. We denote the node corresponding to $e_{i}$ by $N_{i}$. We suppose that the local
equation of $\bar\famcurv$ around $N_{i}$ is given by $xy=t^{a_{N_{i}}}$.
We denote by $\omega_{V_{j}}$ the restriction of $\omega$ to the irreducible component $\X_{V_{j}}$ of $\X$
corresponding to~$V_{j}$. We can suppose (maybe after rescaling) that the family of differentials $\famomega$
converges to $\omega_{V_{1}}$ on $\X_{V_{1}}$. It follows from Lemma~\ref{lemme:relationScaleAuNoeud} that the
 scaling of $\X_{V_{2}}$ for $\famomega$ is $ \left(t^{a_{N_{1}}}\right)^{\alpha_{1}w(e_{1})}$.
Therefore the family of
differentials 
\[  \left(t^{a_{N_{1}}}\right)^{\alpha_{1}w(e_{1})}\famomega   \]
converges to $\omega_{V_{2}}$. Looking at the node $N_{2}$, the family
\[  \left(t^{a_{N_{2}}}\right)^{\alpha_{2}w(e_{2})} \left(t^{a_{N_{1}}}\right)^{\alpha_{1}w(e_{1})}\famomega  
\]
converges to $\omega_{V_{3}}$.
 We iterate this process until $i=l$ and we obtain that the family of differentials 
\begin{equation}
 \prod_{i=1}^{l}(t^{a_{N_{i}}})^{\alpha_{i}w(e_{i})}\famomega
\end{equation}
converges to $\omega_{V_{1}}$. By uniqueness of the scaling for a given irreducible component, the following
equation is satisfied
\begin{equation}
 \prod_{i=1}^{l}(t^{a_{N_{i}}})^{\alpha_{i}w(e_{i})}=1.
\end{equation}
In particular, the tuple $(t^{a_{N_{1}}},\cdots,t^{a_{N_{r}}})\in(\Delta^{\ast})^{\nodeSet}$ satisfies
Equation~\eqref{equation:paramCylindrique} for every closed path $\gamma$ in the dual graph of $(\X,\omega)$.
\end{proof}

\paragraph{Plumbing Cylinder Construction.}

We develop the theory of plumbing cylinders in two steps. First, we introduce the {\em plumbing cylinder
construction} at a single node. Second, we use it to smooth
some limit differentials which will be called {\em plumbable differentials}.

Before extending the plumbing cylinder construction  to the case of differentials, let us recall this
classical result known since (at least) Klein. For a simple proof of the
polar case, which extends to the
holomorphic case, see \cite[Encadr\'e~III.2]{MR2768303}.
\begin{lemma}\label{lemme:formeLocalDesFormesDiff}
 Let $\omega$ be a differential on a Riemann surface $\X$ and $Q\in\X$.  Let~$k$ be the order
and $a_{-1}$ be the residue of $\omega$ at $Q$.

There exists an open neighbourhood $U$ of $Q$ and a coordinate $z$ on $U$ such that $z(Q)=0$ and:
\begin{itemize}
 \item[If $k\leq-2$,] the differential $\omega|_{U}$ is given by the equation $
\left(z^{k}+\frac{a_{-1}}{z}\right)\dz$.

\item[If $k=-1$,]the differential $\omega|_{U}$ is given by the equation
 $\frac{a_{-1}}{z}\dz$.

\item[If $k\geq0$,]the differential $\omega|_{U}$ is given by the equation
$z^{k}\dz$.

\end{itemize}
These equations are called the {\em local normal form} of $\omega$ at $Q$.
\end{lemma}

We can now describe the Plumbing cylinder construction.

\begin{lemma}[Plumbing cylinder construction.]\label{lemme:PlomberieCylindriqueOk}
Let $V:=\left\{z\in\CC: |z|<1\right\}$ and $W:=\left\{w\in\CC: |w|<1\right\}$ be two discs in $\CC$
and $U=V\cup W$ identified at their origins.

 Let $(a,b,k)\in\CC^{2}\times \ZZ$ be a triple of the form $(0,0,-1)$ or $(1,-1,k)$ for $k\neq -1$ and let
$a_{-1}$ be a complex
number. We define the differential $\omega$  on $U\setminus 0$  by 
$$\omega|_{V}=az^{k}\dz+\frac{a_{-1}}{z}\dz \text{, and }
\omega|_{W}=\frac{b}{w^{(k+2)}}\dw-\frac{a_{-1}}{w}\dw.$$ 
Then there exists a differential form $\eta$ on the cylinder of parameter $\epsilon$
\begin{equation}\label{equation:cylindreDeParamEpsilon}
 A_{\epsilon}:=\left\{(x,y)\in\CC^{2}: xy=\epsilon,|x|<1,|y|<1\right\},
\end{equation}
 and a  biholomorphism  
\begin{equation} 
\varphi:U\setminus B(0,\sqrt{|\epsilon|})\to A_{\epsilon}\setminus\left\{(x,y)\in
A_{\epsilon}: |x|=|y|\right\},
\end{equation}
 satisfying the following two properties.
\begin{itemize}
\item[i)] The pair $(A_{\epsilon},\eta)$ is a flat cylinder (i.e., $\eta$ has no zeros or poles in
$A_{\epsilon}$).
\item[ii)] The restrictions of the pull back of $\eta$ by $\varphi$ are  
\begin{equation}
 \varphi^{\ast}(\eta)|_{V\setminus  B(0,\sqrt{|\epsilon|})}=az^{k}\dz+\epsilon^{k+1}\frac{a_{-1}}{z}\dz
\end{equation}
and 
\begin{equation}
 \varphi^{\ast}(\eta)|_{W\setminus  B(0,\sqrt{|\epsilon|})}=\epsilon^{k+1}\omega|_{W\setminus 
B(0,\sqrt{|\epsilon|})}.
\end{equation}
\end{itemize}
\end{lemma}

\begin{proof}
First, we prove the result in the cases $k=-1$ and $a_{-1}=0$.

Let $\epsilon\in\Delta^{\ast}$ and define the following spaces:
\begin{equation*}
A_{\epsilon}=\left\{(x,y)\in\CC^{2}; |x|<1, |y|<1, xy=\epsilon \right\},
\end{equation*}
\begin{equation*}
A_{\epsilon}'=A_{\epsilon}\setminus \left\{(x,y)\in A_{\epsilon}; |x|=|y|\right\},
\end{equation*}
and
\begin{eqnarray*}
B_{\epsilon}' &=& B_{V,\epsilon}'\cup B_{W,\epsilon}'\\
 &=& \left\{ z\in V; |z|>\sqrt{|\epsilon|} \right\} \cup \left\{ w\in W; |w|>\sqrt{|\epsilon|} \right\}.
\end{eqnarray*}

The biholomorphism $\varphi$ is given by the two following restrictions (see
Figure~\ref{figure:tuillauterieCylindrique}):
\begin{eqnarray*}
\varphi_{V,\epsilon} &:&B_{V,\epsilon}' \to A_{\epsilon}',\quad  z\mapsto\left(z,\frac{\epsilon}{z}\right),\\
\varphi_{W,\epsilon} &:&B_{W,\epsilon}' \to A_{\epsilon}',\quad w\mapsto\left(\frac{\epsilon}{w},w\right).
\end{eqnarray*}

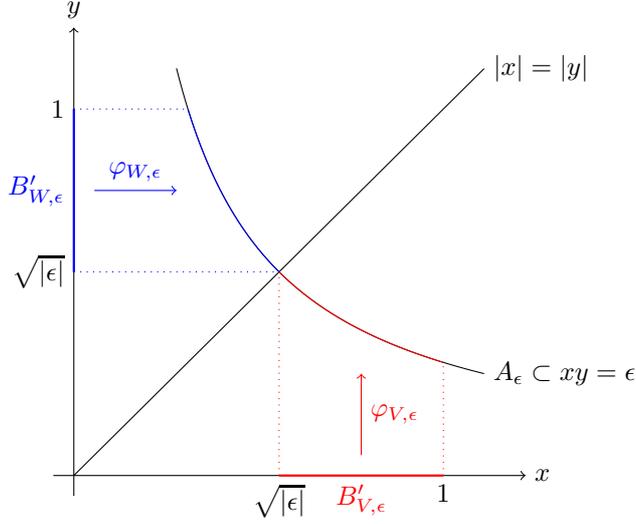
\begin{figure}[ht]
\shorthandoff{:}
\centering
\begin{tikzpicture}[samples=100,scale=2.7]
    \draw[->] (-.1,0)  -- (2.2,0) node[right] {$x$} ;
    \draw[->] (0,-.1) -- (0,2.2) node[above] {$y$} ;

    \draw[domain=.5:2] plot ({\x},{1/\x})     node[right] {$A_{\epsilon}\subset xy = \epsilon$};

    \draw[domain=0:2] plot ({\x},{\x}) node[right] {$|x|=|y|$};

    \draw[red,thick] (1,0) -- (1.8,0);
    \draw[blue,thick] (0,1) -- (0,1.8) ;
      
    \draw[red,domain=1:1.8] plot ({\x},{1/\x});
    \draw[blue,domain=5/9:1] plot ({\x},{1/\x});
      
    \draw[red,dotted] (1,0) -- (1,1);
    \draw[red,dotted] (1.8,0) -- (1.8,5/9);
       
    \draw[blue,dotted] (0,1) -- (1,1);
    \draw[blue,dotted] (0,1.8) -- (5/9,1.8);
       
    \node[below,red] at (1.4,0) {$ B_{V,\epsilon}'$};
    \node[left,blue] at (0,1.4) {$ B_{W,\epsilon}'$};

    \draw[->,red] (1.4,.1) -- (1.4,.5); \node[right,red] at (1.4,.3) {$\varphi_{V,\epsilon}$} ;
    \draw[->,blue] (.1,1.4) -- (.5,1.4); \node[above,blue] at (.3,1.4) {$\varphi_{W,\epsilon}$} ;
    
    \node[below] at (1,0) {$\sqrt{|\epsilon|}$};
    \node[below] at (1.8,0) {$1$};
	\node[left] at (0,1) {$\sqrt{|\epsilon|}$};
    \node[left] at (0,1.8) {$1$};
\end{tikzpicture}
\caption{The maps $\varphi_{V,\epsilon}$ and  $\varphi_{W,\epsilon}$.}
\label{figure:tuillauterieCylindrique}
\end{figure}

Let us now define the differential form $\eta$ on $A_{\epsilon}$ to be the restriction of the
differential of $\CC^{2}$ of equation
\begin{equation}
\frac{x^{k+1}}{x^{2}+y^{2}}\left( x\dx - y\dy \right).
\end{equation}
It is clear that $\eta$ does not vanishes on $A_{\epsilon}$. Therefore $(A_{\epsilon},\eta)$ is a flat
cylinder. 

It remains to compute the pull backs of $\eta$ by  $\varphi_{V,\epsilon}$ and  $\varphi_{W,\epsilon}$.
It is easily verified that the push forward of $\partder{z}$ via $\varphi_{V,\epsilon}$ and $\partder{w}$ via
$\varphi_{W,\epsilon}$ are  respectively
$$\partder{x}-\frac{y}{x}\partder{y}, \text{ and }
-\frac{x}{y}\partder{x}+\partder{y}.$$
Hence the pull backs by $\varphi$ of $\eta$ on $B_{V,\epsilon}'$ and $B_{W,\epsilon}'$ are:
\begin{eqnarray}
\varphi_{V,\epsilon}^{\ast}\left(\eta\right) &=& z^{k}\dz, \\
\varphi_{W,\epsilon}^{\ast}\left(\eta\right) &=& -\frac{\epsilon^{k+1}}{w^{k+2}}\dw.
\end{eqnarray}
In the case $k=-1$, it suffices to multiply this $\eta$ by $a_{-1}$ to obtain all the residues.

Now we prove the general result: let us suppose that $k\neq -1$ and $a_{-1}\neq 0$. The biholomorphism
$\varphi$ is of
course given by the same formula. One can easily verify that the differential $\eta$ is the restriction to
$A_{\epsilon}$ of the differential
\begin{equation}
\frac{x^{k+1}-\epsilon^{k+1}a_{-1}}{x^{2}+y^{2}}\left( x\dx - y\dy \right).
\end{equation}
\end{proof}

Let us now define the subset of the set of limit differentials which can be obtained by plumbing the nodes.
\begin{defn}\label{definiton:diffPlombable}
We say that a limit differential $(\X,\omega,Z_{1},\cdots,Z_{n})$ is {\em plumbable} if there exists a family
of pointed limit differentials 
$$\left(f::\famcurv\to\Delta^{\ast},
\famomega:\Delta^{\ast}\to\dualsheave[\famcurv/\Delta^{\ast}],\seczero_{1},\cdots,\seczero_{n}:\Delta^{\ast}
\to\famcurv \right),$$ 
satisfying the following conditions.
\begin{itemize}
\item The tuple $(\X,\omega,Z_{1},\cdots,Z_{n})$ is the limit differential of this family.
\item For every node $N_{i}$,  there exists a neighbourhood $\mathcal{U}_{i}$ of $N_{i}$ not containing any
other  node or marked point  $Z_{i}$ satisfying the following properties:

 the complement of the union of the $\mathcal{U}_{i}$ is  
$$\famcurv\setminus\bigcup_{i}\mathcal{U}_{i}=\left(\X\setminus\bigcup_{i} U_{i}\right) \times \Delta ,$$ 
where $U_{i}$ denotes the restriction of $\mathcal{U}_{i}$ on $\X$;

the sections $\seczero_{i}$ are given by $Z_{i}\times \Delta$; and

the differentials
$(\mathcal{U}_{i,t},\famomega(t)|_{\mathcal{U}_{i}})$ are given by the plumbing cylinder construction at
$N_{i}$ with a parameter $\epsilon_{i}(t)$.
\end{itemize}
The set of pointed plumbable differentials of type $(k_{1},\cdots,k_{n})$ modulo the action of
$\mathfrak{S}\subset\permGroup$ (see Section~\ref{section:CompactDesStrates})
is denoted by $\kbarmoduliplum(k_{1},\cdots,k_{n})$.
\end{defn}

We now prove that the conditions given in Lemma~\ref{lemme:CondCompPlomberieCyl} and
Lemma~\ref{lemme:ConditionPlomberieCheminsFermer} characterise limit differentials without poles of order
$\geq 2$ with a nonzero residue. Let us recall that for a curve $\X$, we denote by $\nodeSet$ the set of
nodes of $\X$. Moreover, let $e_{i}$ be an edge in the dual graph of $(\X,\omega)$ (see
Definition~\ref{definition:GrapheDualDiffLimites}), we denote by $w(e_{i})$ the weight of $e_{i}$ (which is
one greater than the order of the corresponding node).

\begin{theorem}\label{theoreme:PlomberieCylindriqueSansResidu}
Let $(\X,\omega,Z_{1},\cdots,Z_{n})$ be a candidate differential which has no residue at the poles of
order $k\geq 2$. 

 If $(\X,\omega,Z_{1},\cdots,Z_{n})$ satisfies the three  conditions,
\begin{itemize}
 \item[i)] The Compatibility
Condition $($Equation~\eqref{equation:conditionDeCompatibiliteGeneral}$)$ 
\[
 \ord_{N_{1}}(\omega)+ \ord_{N_{2}}(\omega)=-2,
\]
at every node $N_{1}\sim N_{2}$ of $\X$. 

 \item[ii)] The Residue Condition $($Equation~\eqref{equation:conditionDeCompatibiliteGeneral2}$)$
\[
 \Res_{N_{1}}(\omega)+ \Res_{N_{2}}(\omega)=0,
 \] 
at every node $N_{1}\sim N_{2}$ of $\X$. 

\item[iii)]There exists a tuple
$(\epsilon_{1},\cdots,\epsilon_{r})\in(\Delta^{\ast})^{\nodeSet}$ satisfying
Equation~\eqref{equation:paramCylindrique}, i.e.
\[
\prod_{i=1}^{l}\epsilon_{j_{i}}^{\alpha_{i}w(e_{i})}=1
\]
 for every closed path $\gamma:=\sum_{i=1}^{l}\alpha_{i}e_{i}$ in the
dual graph of $(\X,\omega)$.
\end{itemize}
 then $(\X,\omega,Z_{1},\cdots,Z_{n})$ is a plumbable differential.
\end{theorem}

\begin{proof}
Let $(\X,\omega,Z_{1},\cdots,Z_{n})$ be a candidate differential which has no residue at the poles of order
$k\geq 0$ and
let  $N_{1},\cdots,N_{m}\in\nodeSet$ be the
nodes of $\X$.

It is easily verified that if the parameters $(\epsilon_{1},\cdots,\epsilon_{r})$ satisfy
Equation~\eqref{equation:paramCylindrique} for any closed path, then the same holds for
$(\epsilon_{1}^{1/t},\cdots,\epsilon_{r}^{1/t})$  for any $t\in\Delta^{\ast}$.
Hence it suffices to show that $(\X,\omega,Z_{1},\cdots,Z_{n})$ can be plumbed using the parameters
$(\epsilon_{1},\cdots,\epsilon_{r})$ of Theorem~\ref{theoreme:PlomberieCylindriqueSansResidu}.

According to \cite[page~184]{MR2807457}, there exist neighbourhoods $U_{i}$ of $N_{i}$ which contains neither
any  other  node nor any point  $Z_{i}$. They may be chosen as the unions of the discs
$V_{i}=\left\{z_{i}\in\CC;|z_{i}|<1\right\}$ and
$W_{i}=\left\{w_{i}\in\CC;|w_{i}|<1\right\}$, identified at their origins. Moreover, since the
compatibility
condition and the residue condition are satisfied, we can suppose that the
respective
restrictions
of $\omega$ to $V_{i}$ and $W_{i}$ are of the form $a_{i}z_{i}^{k_{i}}\dz_{i}$ and
$b_{i}w_{i}^{-(k_{i}+2)}\dw_{i}$, where $a_{i}$ and $b_{i}$ are not zero and $a_{i}=-b_{i}$ if $k_{i}=-1$.

 We define for each node $N_{i}$ the following spaces:
\begin{equation*}
A_{i}=\left\{(x_{i},y_{i})\in\CC^{2}; |x_{i}|<1, |y_{i}|<1, x_{i}y_{i}=\epsilon_{i} \right\},
\end{equation*}
\begin{equation*}
A_{i}'=A_{i}\setminus \left\{(x_{i},y_{i})\in A_{i}; |x_{i}|=|y_{i}|\right\},
\end{equation*}
and
\begin{equation*}
B=\left(\X\setminus \underset{i}{\bigcup} U_{i}\right) \bigcup \left(\underset{i}{\bigcup}
B_{i}'\right),
\end{equation*}
where
\begin{eqnarray*}
B_{i}' &=& B_{V_{i}}'\cup B_{W_{i}}'\\
 &=& \left\{ z_{i}\in V_{i}; |z_{i}|>\sqrt{|\epsilon_{i}|} \right\} \cup \left\{ w_{i}\in W_{i};
|w_{i}|>\sqrt{|\epsilon_{i}|} \right\}.
\end{eqnarray*}
Now the curve $\X'$ is the union of $B$ and $\underset{i}{\bigcup}A_{i}$ with the space
$\underset{i}{\bigcup} B'_{i}$ and
$\underset{i}{\bigcup}A'_{i}$ identified via the embeddings:
\begin{eqnarray*}
\varphi_{V_{i}} :B_{V_{i}}'\to A_{i},&
z_{i}\mapsto\left(z_{i},\frac{\epsilon_{i}}{z_{i}}\right)\\
\varphi_{W_{i}} :B_{V_{i}}'\to A_{i},&
w_{i}\mapsto\left(\frac{\epsilon_{i}}{w_{i}},w_{i}\right).
\end{eqnarray*}
We denote the image of $\varphi_{V_{i}}$ by $A_{i}^{V}$ and the image of
$\varphi_{W_{i}}$ by $A_{i}^{W}$. 

Let us remark that the connected components of $\X\setminus
\underset{i}{\bigcup} U_{i}$ and  $\X'\setminus\underset{i}{\bigcup}A_{i}$ are canonically biholomorphic. We 
denote these connected components by $\tilde\X_{j}$ and the corresponding irreducible components of $\X$ by
$\X_{j}$. Moreover, the set of cylinders $A_{i}^{V}$ and
$A_{i}^{W}$ which are at the boundary of $\tilde\X_{j}$ is denoted by $\mathcal{C}(\tilde\X_{j})$.

According to Lemma~\ref{lemme:PlomberieCylindriqueOk}, there exist  differentials
$\omega_{1}',\cdots,\omega_{r}'$ on the cylinders $A_{1},\cdots,A_{r}$ which are proportional to $\omega$ on
$A_{i}^{V}$ and $A_{i}^{W}$. More precisely, if $\omega_{i}'=\alpha_{i}\omega$ for an
$\alpha_{i}\in\CC^{\ast}$ on $A_{i}^{V}$, then $\omega_{i}'=\alpha_{i}\epsilon_{i}^{\pm (k_{i}+1)}\omega$
on $A_{i}^{W}$. The fact that the constants of proportionality are distinct is the key point in the rest of
the proof.

To complete the proof, it suffices to show that we can extend the differentials~$\omega_{i}'$ by a
differential $\omega'$ on $\X'$ which is proportional to $\omega$ on every
component~$\tilde\X_{i}$. Observe that such a differential exists if and only if for every
component~$\tilde\X_{j}$ there exists a common constant of
proportionality between $\omega_{i}'$ and $\omega$ for every cylinder $A_{i}$ in $\mathcal{C}(\tilde\X_{j})$.

Let us construct the constants of proportionality in the following way. Let~$\X_{1}$ be an irreducible
component of $\X$ and $a_{1}\in\CC^{\ast}$. We impose that on every
cylinder of $\mathcal{C}(\tilde\X_{1})$ the relation between $\omega$ and $\omega'_{i}$ is given by
$\omega_{i}'=a_{1}\omega$.

Let $\X_{k}$ be another irreducible component of $\X$. For every path
$$\gamma_{1,k}=\sum_{i=1}^{l_{k}}\alpha^{k}_{i}e^{k}_{i}$$  from $\X_{1}$ to $\X_{k}$ in the dual
graph of $(\X,\omega)$ we assign the following number 
\begin{equation}\label{equation:constanteCheminGraphe}
 a_{k}^{\gamma}:= a_{1}\prod_{i=1}^{l_{k}}\epsilon_{j_{i}}^{\alpha^{k}_{i}w(e^{k}_{i})},
\end{equation}
where $w(e^{k}_{i})$ one greater than the order of the node  corresponding to
$e_{i}^{k}$.

It suffices to prove that under the third condition of
Theorem~\ref{theoreme:PlomberieCylindriqueSansResidu} the $a^{\gamma}_{k}$ do not
depend on the choice of the path $\gamma$. Indeed, if this is the case there exists a differential $\omega'$
on
$\X'$ which coincides with $a_{k}\omega$ on $\tilde\X_{k}$.

Let $\gamma_{1}$ and $\gamma_{2}$ be two paths from $\X_{1}$ to $\X_{2}$ in the dual graph of $(\X,\omega)$.
Then the number associated  by Equation~\eqref{equation:constanteCheminGraphe} to the concatenation
$\gamma_{1}\circ\gamma_{2}^{-1}$ is 
$a_{k}^{\gamma_{1}}(a_{k}^{\gamma_{2}})^{-1}$. Hence it suffices to show that
$a_{k}^{\gamma_{1}}(a_{k}^{\gamma_{2}})^{-1}=a_{1}$ to conclude the proof. Let us denote the path
$\gamma_{1}\circ\gamma_{2}^{-1}$
by $\sum_{i=1}^{l}\alpha_{i}e_{i}$. Then by definition 
\begin{equation*}
 a_{1}^{\gamma}=a_{1}\prod_{i=1}^{l}\epsilon_{j_{i}}^{\alpha_{i}w(e_{i})}.
\end{equation*}
Since the parameters $\epsilon_{i}$ satisfy Equation~\eqref{equation:paramCylindrique}, this quantity is
precisely $a_{1}$.
\end{proof}

As an easy application of this theorem, we have the following remark.
\begin{rem}
Let $(\X,\omega)$ be a holomorphic differential with at least one zero $Z$ of order $k\geq2$.
Moreover,
let $(\PP^{1},0,1,\infty)$ be a rational curve with three marked points and define the differential
$\eta_{i}:=z^{i}(z-1)^{k-i}\dz$ on $\PP^{1}$. Attaching $\X$ to $\PP^{1}$ via the identification of $Z$ with
$\infty$ and using the plumbing cylinder construction of Lemma~\ref{lemme:PlomberieCylindriqueOk}, we obtain
the
construction of \cite{MR2010740} for breaking up a zero of a differential into a pair of zeros.

An advantage of this construction is that it can be easily generalised to the case of breaking up a zero into
more zeros. We use such a generalisation in the proof of
Theorem~\ref{theoreme:ExcesAuBordStratePrincipal}.
\end{rem}

As shown in Lemma~\ref{lemme:formeLocalDesFormesDiff}, there exist differentials which have a pole of order
$k\geq2$ and a nonzero residue. If our candidate differential has such local behaviour at a node, then the
conditions of Theorem~\ref{theoreme:PlomberieCylindriqueSansResidu} are not sufficient to be smoothable. 

\begin{ex}\label{exemple:ObstructionPlomberie2}
Let $(\X,\omega,Z)$ be a candidate differential of genus two such that
$\X:=\X_{1}\cup \PP^{1} \cup \X_{2}$, where $(\X_{1},\omega|_{\X_{1}})$ and $(\X_{2},\omega|_{\X_{2}})$ are
two
flat tori and the projective line has coordinate $z$ such that it is attached to $\X_{1}$ at $0$
and
to $\X_{2}$  at $\infty$. Finally, the restriction of $\omega$ to $\PP^{1}$ is
$\omega_{0}:=\frac{(z-1)^{2}}{z^{2}}\dz$.

The differential $(\X,\omega,Z)$ is not a limit differential. Otherwise, the differential~$\famomega(t)$ of
the
family $\left(f:\famcurv\to\Delta^{\ast},
\famomega:\Delta^{\ast}\to\dualsheave[\famcurv/\Delta^{\ast}],\seczero:\Delta^{\ast}\to\famcurv \right)$ would
have a
zero of order two at $\seczero(t)$. Therefore, the point $\seczero(t)$ would be a Weierstrass point
of~$\famcurv(t)$. Since the limiting position of the Weierstrass points are the $2$-torsion points of both
elliptic curves,  the curve $\famcurv(t)$ would have seven Weierstrass points 
 (see Theorem~\ref{theoreme:limitDesPtsDeWeiCasHyp}), a contradiction.
\end{ex}

We do not give a complet characterisation of the limit differentials having poles of order greater or equal to
$2$ with a nonzero
residue. However, the following lemma gives a necessary condition.

\begin{lemma}\label{lemme:noeudsPolaireAvecResiduFaible}
 Let $(\X,\omega,Z_{1},\cdots,Z_{n})$ be a limit differential. Let $N$ be a node between the irreducible
components $\X_{1}$ and $\X_{2}$ such that the local normal form at $N$ of $\omega|_{\X_{1}}$ is $z^{k}\dz$
and the one of  $\omega|_{\X_{2}}$ is $z^{-k-2}+\alpha z^{-1}\dz$. 
Then, there exists a differential $\eta$ on $\X_{1}$ with a pole of order $-1$ and residue $-\alpha$ at $N$
which
has no poles on the smooth locus of $\X_{1}$.
\end{lemma}

\begin{proof}
Let $N$ be a node where the restriction of $\omega$ to one branch of the node has a pole of
order $d+1\geq 2$ with a  nonzero residue. Let $\left(\famcurv,
\famomega,\seczero_{1},\cdots,\seczero_{n}\right)$ be a family of pointed differentials which converges to
$(\X,\omega,Z_{1},\cdots,Z_{n})$.
Let~$U$ be a neighbourhood of $N$ in $\bar\famcurv$  given inside $\Delta^{3}$ by the equation
$$xy=t^{a},$$
with $a\geq 1$.
Let us suppose that the equation of the branch of $N$ with the pole is $\left\{ x=0 \right\}$ (denoted by
$\X_{y}$) and the restriction on it is
$$\omega_{y}=\frac{1}{y^{d}}(1+\alpha y^{d})\frac{-\dy}{y}$$ and that the family $\famomega$ converges to
$\omega_{y}$ on $\X_{y}$. We denote by $\eta$ the differential $\frac{x\dx-y\dy}{x^{2}+y^{2}}$ on $U$. Observe
that the family of
differentials $\famomega$ is given by 
\[\famomega=g\cdot\eta,\]
where $g$ is a meromorphic function with no poles or zeros in $U\setminus\X|_{U}$. The limit of
$(t^{a})^{d}\famomega$ to
$\left\{ y=0 \right\}$ (denoted by
$\X_{y}$) is $$\omega_{x}=x^{d}\frac{\dx}{x}.$$  The limit
of the family 
\[ \left(f-\frac{x^{d}}{(t^{a})^{d}}\right)\cdot\eta\]
on $\X_{x}$ is $\alpha\frac{-\dx}{x}$. This form can be prolonged on the whole component $\X_{1}$ containing
$\X_{x}$ to a
meromorphic form $\omega_{1}$. 

It remains to show that the poles of $\omega_{1}$ cannot be located in the smooth locus
of $\X_{1}$. Let us denote the family which prolongs respectively $
\left(g-\frac{x^{d}}{(t^{a})^{d}}\right)\eta$ and $\frac{x^{d}}{(t^{a})^{d}}\eta$ on $\famcurv$ by
$\famomega_{1}$ and $\famomega_{2}$. It follows from the fact that $\famomega_{1}$ converges to a meromorphic
differential on $\X_{1}$ and the equality
$$ t^{ad}\famomega=t^{ad}\left(\famomega_{1}+\famomega_{2}\right),$$
 that the limit of $t^{ad}\famomega_{2}$ on $\X_{1}$ is $\omega|_{\X_{1}}$. And it follows by
addition that the differential $\omega_{1}$ has no poles on the smooth locus of $\X_{1}$.
\end{proof}

An interesting application  is the fact that the zero of a differential
in the strata $\omoduli(2g-2)$ cannot
converge to the node of a compact curve with two components. 

\begin{cor}\label{corollaire:zeroJamaisSurLePontFaible}
Let $(\X,\omega,Z)\in\kbarmodulilim[g](2g-2)$ be a limit pointed differential with a single zero. Then
$(\X,Z)$ is not the union of two components attached by a pointed projective line.
\end{cor}

\begin{proof}
If it was the case, then the 
restriction
of the form  $\omega$ to the projective line would be  $$\frac{(z-1)^{2g-2}}{z^{2g_{1}}}\dz$$ in a
coordinate $z$, where the nodes are $z=0$ and $z=\infty$. This form has always a nonzero residue at the
nodes. Let $\X_{1}$ be another irreducible component.
This would implies that $\X_{1}$ has a differential with a single pole, which is of order one.
\end{proof}

We now give conditions which are sufficient to smooth a candidate differential having poles of order $\geq2$
with nonzero residue. They are too
strong to be necessary.
Recall that for an irreducible component $\X_{\alpha}$ of $\X$ we denote by $\nodeSet[\X_{\alpha}]$ the set
of nodes of $\X$ meeting $\X_{\alpha}$. And if $N$ is a node between $\X_{\alpha}$ and $\X_{\beta}$, we
denote by $N_{\alpha}$ the point of $N$ belonging to $\X_{\alpha}$.

\begin{lemma}\label{lemme:PlomberieCylindriqueAvecResidu}
If $(\X,\omega,Z_{1},\cdots,Z_{n})$ is a candidate differential
 which satisfies the following properties, then it is plumbable. 

The Compatibility Condition~\eqref{equation:conditionDeCompatibiliteGeneral}
\[
 \ord_{N_{1}}(\omega)+ \ord_{N_{2}}(\omega)=-2
\]
 holds at every node of $\X$.
The Residue Condition~\eqref{equation:conditionDeCompatibiliteGeneral2}
\[
 \Res_{N_{1}}(\omega)+ \Res_{N_{2}}(\omega)=0,
 \] 
 holds at every node of $\X$ with a simple pole of $\omega$.

 There exists
$(\epsilon_{1},\cdots,\epsilon_{n})\in(\Delta^{\ast})^{\nodeSet}$
satisfying Equation~\eqref{equation:paramCylindrique} for every closed path $\gamma$ in the dual graph of
$(\X,\omega)$.

There exists a differential form $\eta_{\alpha}$ and a constant $c_{\alpha}\in\CC^{\ast}$ on every irreducible
component $\X_{\alpha}$ of
$\X$ satisfying the following properties. The differential~$\eta_{\alpha}$ has simple poles with residue
$-a_{i}$ at every node $N_{i}\in\nodeSet[\X_{\alpha}]$ where $\omega|_{\X_{\beta}}$ has a pole of order $k\geq
0$ with residue $a_{i}\neq0$. At
every other point $Q$ in $\X_{\alpha}$, the residue of $\eta_{\alpha}$ is zero and
$$\ord_{Q}(\eta_{\alpha}) \geq \ord_{Q}(\omega).$$
At every simple pole $N$ of $\eta_{\alpha}$, the parameter $\epsilon_{N}$
satisfies $$\epsilon_{N}^{k_{N}+1}=c_{\alpha}.$$

\end{lemma}

We will prove that on the open set corresponding to $\X_{\alpha}$, the smoothed differential is proportional
to $\omega+c_{\alpha}\eta_{\alpha}$.

\begin{proof}
Let us first remark that if the parameters $\epsilon_{i}$ are small enough, then the orders of
$\omega_{\alpha}$
and $\omega_{\alpha}+c_{\alpha}\eta_{\alpha}$ coincide at every node of $\X$. By replacing the $\epsilon_{i}$
by~$\epsilon_{i}^{r}$, for
some $r>1$, we can
suppose that this is the case.
Let $N$ be a node  of~$\X$ between the components $\X_{\alpha}$ and $\X_{\beta}$. First suppose that the
pole of $\omega$ at $N$ has no residue. Then by plumbing the node $N$, we can
find a
differential which coincide with $$\omega_{\alpha}+c_{\alpha}\eta_{\alpha}$$ on the part of the cylinder
meting $\X_{\alpha}$ and with 
$$\epsilon_{N}^{\pm(k_{N}+1)} \left( \omega_{\beta}+c_{\beta}\eta_{\beta}\right)$$
on the other part of the cylinder.
Now suppose that $\omega_{\beta}$ has a pole of order $k\geq 2$ with a nonzero residue. Then the plumbing
cylinder construction gives a differential which coincide with 
 $$\omega_{\alpha}+\epsilon_{N}^{k_{N}+1}\eta_{\alpha}$$ on the part of the cylinder
meting $\X_{\alpha}$ and with 
$$\epsilon_{N}^{k_{N}+1} \left( \omega_{\beta}+c_{\beta}\eta_{\beta}\right)$$ on the other part.
Since by hypothesis $\epsilon_{N}^{k_{N}+1}=c_{\alpha}$, we can prolong this differential on the component
$\X_{\alpha}$.

Finally, it follows from the fact that the parameters satisfy Equation~\eqref{equation:paramCylindrique}
that the constants of proportionality are globally well defined (see the proof of
Lemma~\ref{lemme:ConditionPlomberieCheminsFermer} for details). 
\end{proof}

\begin{rem}
As said before, the hypotheses of this lemma are to strong. A way to generalise it is to
allow many differentials $(\eta_{\alpha,1},\cdots,\eta_{\alpha,r})$ on a component~$\X_{\alpha}$ with distinct
constants $c_{\alpha,i}$. We do not write the precise condition, because it becomes very technical and messy.

Let us remark that a necessary condition for the existence of the differentials~$\eta_{\alpha}$ is that the
following equation is satisfied
\begin{equation}\label{equation:sommeResidusNoeudsExterieur}
\sum_{N_{i}\in\nodeSet[\X_{\alpha}]}\Res_{N_{i,\beta_{i}}}(\omega)=0,
\end{equation}
where $N_{i}$ is a node between $\X_{\alpha}$ and a distinct component $\X_{\beta_{i}}$.
This condition should be necessary for every limit differentials, but we have no proof of it.
\end{rem}

\paragraph{Relationship with the {incidence variety compactification}.}

To conclude this section, we show how to obtain a stable pointed differential from a limit differential
without any pole of order $k\geq2$ with a nonzero residue. We even prove that this map extends
to the case of  plumbable differentials satisfying the hypotheses of
Lemma~\ref{lemme:PlomberieCylindriqueAvecResidu}.

\begin{defn}\label{definition:composentesPolairesConnectees}
 Let $(\X,\omega)$ be a limit or stable differential. Let $\X_{1}$ and $\X_{2}$ be two irreducible components
of $\X$. We say that $\X_{1}$ and $\X_{2}$ are {\em polarly related by $\omega$} if $\X_{1}=\X_{2}$ or the
differential $\omega$ has simple poles at the
nodes between $\X_{1}$ and $\X_{2}$.

The equivalence classes of this relation are the {\em polarly related components} of~$(\X,\omega)$.
\end{defn}

 There is a map from the set of limit differentials to the space of marked stable differentials 
$$\varphi:\kbarmodulilim[g,n](k_{1},\cdots,k_{n})\to\obarmoduli[g,n],$$
which is given by forgetting the differentials on the polarly related components which contain a pole of order
$\geq 2$. 
\begin{prop}\label{proposition:relationPlumStable}
Let $(\X,\omega,Z_{1},\cdots,Z_{n})$ be a plumbable differential satisfying the hypotheses of 
Lemma~\ref{lemme:PlomberieCylindriqueAvecResidu}.
The marked differential $\varphi \left(\X,\omega,Z_{1},\cdots,Z_{n}\right)$ is contained in
$\obarmoduliinc{k_{1},\cdots,k_{n}}$.
\end{prop}

The point is to show that we do not have to forget the differentials more irreducible components of $\X$.   

\begin{proof}
Let $(\X,\omega,Z_{1},\cdots,Z_{n})$ be a plumbable differential in the closure of the stratum
$\omoduliincp{k_{1},\cdots,k_{n}}$.
Let us first remark that we can suppose that the polarly related components are the irreducible components of
$\X$. Otherwise, we use the plumbing cylinder construction at the nodes with poles of order one. The resulting
differential is still a plumbable differential satisfying the hypotheses of 
Lemma~\ref{lemme:PlomberieCylindriqueAvecResidu} and which lies in the closure of the stratum
$\omoduliincp{k_{1},\cdots,k_{n}}$ if and only if the previous differential lay in its closure.

Let $(\epsilon_{1}:\Delta^{\ast}\to\Delta^{\ast},\cdots,\epsilon_{n}:\Delta^{\ast}\to\Delta^{\ast})$ be
parameters at the nodes of $\X$ which satisfy Equation~\eqref{equation:paramCylindrique}. Let $c_{i}(t)$ and
$\eta_{i}(t)$ be the constants and differentials satisfying the hypotheses of
Lemma~\ref{lemme:PlomberieCylindriqueAvecResidu}. The family of
differentials which
is obtained by
plumbing the nodes with these parameters and such that the limit is a nonzero stable differential
$\tilde\omega$ is denoted
by $(\famcurv,\famomega,\seczero_{1},\cdots,\seczero_{n})$. 

Let $\X_{i}$ be an irreducible component of $\X$ such that $\omega|_{\X_{i}}$ is holomorphic, but the
differential $\tilde\omega|_{\X_{i}}$ is identically zero. Let $V_{i}$ be the subset of $\X_{i}$ which is the
complement of $\X_{i}\setminus\bigcup\limits_{j}(U_{j}\cap\X_{i})$, where the $U_{j}$ are the neighbourhoods
of the nodes in which the plumbing take place. In $V_{i}\times
\Delta$, the differential satisfies 
$$\famomega|_{V_{i}\times \Delta}= h(t)\left(\omega|_{V_{i}}+c_{i}(t)\eta_{i}(t)\right),$$
where $h$ is a function vanishing at the origin.

For every node $N_{i,j}$ of $\X_{i}$, we replace the parameter $\epsilon_{N_{i,j}}$  by
$$h(t)^{1/(k_{N_{i,j}}+1)}\cdot\epsilon_{N_{i,j}}(t),$$  where $k_{N_{i,j}}$ is the order of the zero of
$\omega$
at $N_{i,j}$.  The parameters remain unchanged at the other nodes of $\X$.

Let us show that these new parameters satisfy the conditions given by
Equation~\eqref{equation:paramCylindrique}. Let $\gamma$ be a closed path in the dual graph of
$(\X,\omega)$. Let us denote the vertex corresponding to $\X_{i}$ in the dual graph of $(\X,\omega)$ by
$V_{i}$. Since $\gamma$ is closed, it has the same number of edges which point to~$V_{i}$ as edges which come
from~$V_{i}$.
Using the fact that the component $\X_{i}$ has only holomorphic nodes for $\omega$, we deduce that an incoming
edge and an outgoing edge of
$\gamma$ contribute together to Equation~\eqref{equation:paramCylindrique} by
$$\left(h(t)^{1/(k_{N_{i,j}}+1)}\epsilon_{N_{i,j}}\right)^{(k_{N_{i,j}}+1)} \cdot
\left(h(t)^{1/(k_{N_{i,k}}+1)}
\epsilon_{N_{i,k}}\right)^{-(k_{N_{i,k}}+1)},$$
which is clearly equal to
$$(\epsilon_{N_{i,j}})^{(k_{N_{i,j}}+1)}\cdot(\epsilon_{N_{i,k}})^{-(k_{N_{i,k}}+1)}.$$
So the contribution to Equation~\eqref{equation:paramCylindrique} of the new parameters at the nodes
of~$\X_{i}$ is the same as the contribution with the old ones. This implies that this equation remains
satisfied
by the new parameters.
It is direct to check that the constants~$c_{j}$ coincide with the previous ones when $j\neq i$ and~$c_{i}$ is
replaced by new constants $c_{i}'$. It is easily verified that we may keep
the same differentials~$\eta_{i}$.

According to Lemma~\ref{lemme:PlomberieCylindriqueAvecResidu}, we can smooth the limit differential
$\omega$ using these new parameters. We scale the family of differentials in such a way that the new one
coincides with the old one on $V_{j}\times \Delta$, for  every irreducible component
$\X_{j}$ different
from $\X_{i}$. On the other hand, we claim that in a
neighbourhood of $V_{i}$, we have 
 $$\famomega_{{\rm new}}|_{V_{i}\times \Delta}= \omega|_{V_{i}}+c_{i}'\eta_{i},$$
for the family with the new parameters.
Indeed, let $\X_{k}$ be an irreducible component which meets $\X_{i}$ at $N_{l}$. Let $h_{j}:\Delta\to\Delta$
denotes the
function such that $$\famomega_{{\rm new}}|_{V_{j}\times \Delta}= h_{j}(t)
(\omega|_{V_{j}}+c_{j}'\eta_{j}).$$ 
Since for every irreducible component $\X_{j}$ distinct from $\X_{i}$, this equation holds for the family
$\famomega$, it follows from Lemma~\ref{lemme:PlomberieCylindriqueOk} that
$$\frac{h_{j}}{h}=(\epsilon_{N_{l}})^{(k_{N_{l}}+1)} \text{ and }
\frac{h_{j}}{h_{i}}=h\cdot (\epsilon_{N_{l}})^{(k_{N_{l}}+1)}.$$
It follows from these two equation that $c_{i}=1$ as claimed.
This implies that the stable limit of this family restricts to $\omega$ on $\X_{i}$ and to $\tilde\omega$ on
the
other
irreducible components of $\X$.

The Proposition follows by doing this procedure at every component of $\X$ where $\omega'$ vanishes but
$\omega$ restricts
to a holomorphic differential.
 \end{proof}


\section{Parity at the Boundary of the Strata.}\label{section:Spin}

The notion of theta characteristic is an essential tool for the description of the connected components of the
strata of $\omoduli$. Indeed, every stratum of the form $\omoduli(2l_{1},\cdots,2l_{n})$ has at least two
connected components distinguished by the parity of the theta characteristic associated to the differential.
It would be nice to show that this invariant can be extended for all limit
differentials in the closure of such strata. 
 However, we will show (see Corollary~\ref{corollaire:intersectHypOddGenreTrois})
that such extension is not possible in general. Indeed, the {incidence variety
compactifications} of the even and the odd components of $\pomoduli[3,1](4)$ meet each other.

 In this section, we will nevertheless extend this invariant to two important cases. In the first part, we
treat the case of limit differentials above curves of compact
type (see Theorem~\ref{theoreme:BordDesStratesDansSpin}). This uses the theory of spin structures introduced
by Cornalba, which  we will recall at the beginning of this section. In the second part, we extend this
invariant to the
case of irreducible stable pointed differentials (see Theorem~\ref{theoreme:InvDeArfGene}). For this purpose,
we generalise the Arf invariant to such differentials
(see Definition~\ref{definition:ArfGen}).

\subsection{Differentials of Compact Type.}

Let us begin this section by some preliminary paragraphs about line bundles on (semi) stable curves
 and Cornalba theory of spin structures. 

\paragraph{Some basic facts about line bundles on stable curves.}

The material of this paragraph comes mostly from \cite{MR2807457} and \cite{MR1631825}. Let us recall
that the
normalisation of a (semi) stable curve $\X$ is denoted by $\nu:\tilde{\X}\to\X$. We denote by
$\Irr(\X) := \left\{
\X_{i}
\right\}$ the set of irreducible components of $\X$ and by $\nu_{i}:\tilde{\X_{i}}\to\X_{i}$ the restriction 
of the normalisation to $\X_{i}$. The set of nodes~$\nodeSet$ of~$\X$ is of cardinality $\mathfrak{n}$ and for
each node $N_{i}$
of
$\X$, its preimage by $\nu$ is~$\lbrace N_{i,1},N_{i,2} \rbrace$.

The key to describe the Picard group of $\X$ is the exact sequence
\begin{equation}\label{equation:suiteExCourteFibDroiteNodal}
1\to\Ox^{\ast}\to\nu_{\ast}\Ox[\tilde{\X}]^{\ast}\xrightarrow{e}
\prod_{N\in\nodeSet}\CC_{N}^{\ast}\to 1,
\end{equation}
where the map $e$ is defined in the following way. For every $h\in\nu_{\ast}\Ox[\tilde{\X}]^{\ast}$, the
$\CC_{N_{i}}^{\ast}$-component of $e(h)$ is $\frac{h(N_{i,1})}{h(N_{i,2})}$. The long exact
sequence associated to the short exact sequence \eqref{equation:suiteExCourteFibDroiteNodal} is
\begin{equation}\label{equation:suiteExLongFibDroiteNodal}
1\to\CC^{\ast}\to(\CC^{\ast})^{|\Irr(\X)|}\to (\CC^{\ast})^{\mathfrak{n}}\to
\Pic(\X)\xrightarrow{\alpha^{\ast}} \Pic(\tilde{\X}) \to 1   .
\end{equation}
The interpretation, from the right to the left, of this sequence is the following.
\begin{itemize}
\item[i)] To describe a line bundle $L$ on $\X$ it suffices to give a line bundle $\tilde{L}$ on $\tilde{\X}$
and an identification $\varphi_{N_{i}}:\tilde{L}_{N_{i,1}}\to\tilde{L}_{N_{i,2}}$ of the
fibres above the preimages of each node $N_{i}\in\nodeSet$. The second part of the data are usually called the
{\em descent data} of $L$. Let us remark that the descent data can be interpreted as a condition for a section
of
$L$ to be a lift of a section of $\tilde{L}$.

\item[ii)] If $\tilde{L}$ is trivial, a choice of trivialisation identifies each $\varphi_{N}$ with a well
defined non-zero complex number. So, two line bundles $L_{1}$ and $L_{2}$ such that
$\tilde{L_{1}}=\tilde{L_{2}}$ differ only by a tuple of $\mathfrak{n}$ nonzero complex numbers.
\item[iii)] Let $\tilde{L}$ be a line bundle on $\tilde{\X}$. If two $\mathfrak{n}$-tuple describe in ii)
differ only by 
multiplicative constants on each irreducible component, then the line bundles associated to $\tilde{L}$ and
these descent data
are the same. 
\item[iv)] The descent data are well defined up to a global multiplicative constant.
\end{itemize}

Let us discuss two examples in which we will be particularly interested.
\begin{ex}
If the curve $\X$ is of compact type, then the sequence \eqref{equation:suiteExLongFibDroiteNodal}  implies
that the Picard groups of $\X$ and $\tilde\X$ are isomorphic. Therefore in this case, we will define line
bundles
by specifying their restrictions on every irreducible components of $\X$.
\end{ex}

\begin{ex}
Let us now suppose that the curve $\X$ is an irreducible  nodal curve with $r$ nodes. Then the sequence
\eqref{equation:suiteExLongFibDroiteNodal} gives the sequence
\begin{equation*}
1\to (\CC^{\ast})^{r}\to
\Pic(\X)\xrightarrow{\alpha^{\ast}} \Pic(\tilde{\X}) \to 1   .
\end{equation*}
Hence in this case a line bundle on $\X$ is described by a line bundle on $\tilde\X$ and a $r$-tuple of non
zero
complex
numbers.
\end{ex}

We now give a description of the limit of a line bundle over a smooth family of generically smooth curves such
that the special fibre is of compact type.  The proof is given at
the beginning of \cite[Section~5.C]{MR1631825}. 
\begin{theorem}\label{theoreme:ProlongFibreEnDroite}
 Let $f:\famcurv\to\Delta$ be a smooth family such that for every $t\neq 0$, the curve
$\famcurv(t)$ is a
smooth curve of genus $g$ and $\famcurv(0)$ is a reduced curve of compact type.

Let $\mathscr{L}$ be a line bundle of relative degree $d$ on $\famcurv\setminus\famcurv(0)$ and 
$$\alpha:\Irr(\famcurv(0))\to\ZZ^{|\Irr(\famcurv(0))|}$$ be any map such that
$$\sum_{\X_{i}\in\Irr(\famcurv(0))}\alpha(\X_{i})=d.$$ 
Then there exists a unique extension $\mathscr{L}_{\alpha}$ of~$\mathscr{L}$ to
$\famcurv$ such that 
$$\deg(\mathscr{L}_{\alpha}\otimes\Ox[\X_{i}])=\alpha(\X_{i})$$ 
on every irreducible components $\X_{i}$ of $\famcurv(0)$.

Moreover, if $N$ is a node between two irreducible components $\X_{i}$ and $\X_{j}$, and $\beta$ is
obtained from $\alpha$ by adding $1$ to $\alpha(\X_{i})$ and subtracting $1$ from $\alpha(\X_{j})$, then
\begin{eqnarray}
 \mathscr{L}_{\beta}\otimes\Ox[\X_{i}]=\mathscr{L}_{\alpha}\otimes\Ox[\X_{i}](N),
\label{equation:TenseurFibreDroiteCourbeSpecial1} \\
\mathscr{L}_{\beta}\otimes\Ox[\X_{j}]=\mathscr{L}_{\alpha}\otimes\Ox[\X_{j}](-N).
\label{equation:TenseurFibreDroiteCourbeSpecial2}
\end{eqnarray}
\end{theorem}

If the special fibre is not of compact type, there is not such a precise description. However, the idea at the
beginning of \cite[Section~5.C]{MR1631825} remains true for families of curves with more general special
fibre.

\begin{theorem}
\label{theoreme:ProlongFibreEnDroiteGen}
 Let $f:\famcurv\to\Delta$ be a smooth family such that for every $t\neq 0$, the curve
$\famcurv(t)$ is a
smooth curve of genus $g$ and $\famcurv(0)$ is a (semi) stable curve.

Let $\mathscr{L}$ be a line bundle of relative degree $d$ on $\famcurv$ such that the restriction
of~$\mathscr{L}$ to  $\famcurv(0)$ is a line bundle. 
Let $\X_{i}$ be an irreducible
component and let~$\lbrace N_{j,k} \rbrace_{k}$ be the set of nodes between $\X_{i}$ and $\X_{j}$.

Then, we have the relations
\begin{eqnarray}
 \mathscr{L}\otimes\Ox[\famcurv](\X_{i})|_{\X_{i}}=\mathscr{L}|_{\X_{i}}\otimes\Ox[\X_{i}]\left(\sum_{j,k}-N_{
j,k}\right), \label{equation:TenseurFibreDroiteCourbeSpecialGen1} \\
\mathscr{L}\otimes\Ox[\famcurv](\X_{i})|_{\X_{j}}=\mathscr{L}|_{\X_{j}}\otimes\Ox[\X_{j}]\left(\sum_{k}N_{j,k}
\right).
\label{equation:TenseurFibreDroiteCourbeSpecialGen2}
\end{eqnarray}
\end{theorem}

\paragraph{Abstract spin stable curves.}

Following the article of Cornalba \cite{MR1082361}, we now extend the notion of spin structure to the case of
stable curves.  Recall that a spin structure is a pair $(\X,\mathcal{L})$, where $\X$ is a
smooth curve and $\mathcal{L}$ is a theta characteristic on $\X$.
The following curves are the base of the construction.
\begin{defn}
Let $\bar\X$ be a semi stable curve and $\X$ its stable model. A {\em exceptional component} of $\X$ is an
irreducible component which is contracted by the map $\pi:\bar{\X}\to \X$. The non-exceptional components of
$\bar\X$ are called the {\em stable components}.

  A {\em decent curve} is a semi stable curve in which every exceptional component  meets precisely two other
irreducible components such that two exceptional curves are not allowed to meet. In particular, the
exceptional
components have no self intersection.
\end{defn}

We can think of decent curves as stable curves with some of its nodes blowed up. 
Now we can define the notion of spin structure on decent curves.

\begin{defn}\label{definition:spinCurveb}
 A {\em spin curve} is a triple $(\X,\mathcal{L},\alpha)$, where $\X$ is a decent curve, 
$\mathcal{L}$ is a line bundle of degree $g-1$ on $\X$ and $\alpha$ is a map from $\mathcal{L}^{\otimes 2}$ 
to the dual sheave $\dualsheave$, which satisfies the following two properties.
\begin{itemize}
 \item [i)] The line bundle $\mathcal{L}$ has degree $1$ on every exceptional component of $\X$.
 \item [ii)]  The map $\alpha$ is not zero at a general point of every stable component of $\X$.
\end{itemize}
\end{defn}

Now we explain why this is the right generalisation of the notion of smooth spin curves.

 First of all it is easy to verify that for  smooth curves,
this definition coincide  with the usual one, since $\alpha$ is uniquely determined by $\mathcal{L}$. 

Let $\X$ be a curve of compact type and $\mathcal{L}$ a spin structure on it. It follows easily from the
definition of spin structures that the restriction of $\mathcal{L}$ to every irreducible component $\X_{i}$ of
$\X$ of genus $g\geq 1$ is a theta-characteristic
on $\X_{i}$. But the sum of the degrees of these restrictions is the genus of
$\X$ minus the number of irreducible components of $\X$.
 To have a line bundle of degree $g-1$, the curve $\X$ has to be a decent curve with a projective line at
every node.

An expected property of the notion of spin structure is that there exist $2^{2g}$ isomorphism classes of spin
structures on a given
decent curve.
However, there exist in general infinitely many non isomorphic line bundles $\mathcal{L}$ satisfying the
first
part of Definition~\ref{definition:spinCurveb} (this follows from the exact sequence
\eqref{equation:suiteExLongFibDroiteNodal}). The morphism~$\alpha$ rigidify this notion and the
following proposition shows that it gives the right number of spin structure on a decent curve.

\begin{prop}\label{proposition:nbrThetaChar}(\cite[Paragraph~6]{MR1082361})
Let $\X$ be a stable curve, then the number of non isomorphic spin structures on (the set of decent curves
stably equivalent to) $\X$ is $2^{2g}$. Moreover, the number of even ones is $2^{g-1}(2^{g}+1)$ and the number
of odd ones is  $2^{g-1}(2^{g}-1)$. 
\end{prop}

Before recalling that all these properties are well behaved in families, we discuss a basic but typical
example.

\begin{ex}\label{exemple:spin1}
 Let $\X$ be a curve of genus $g$, which is the union of $\X_{1}$ and $\X_{2}$ of genus
$i$ and $g-i$ meeting at a unique point $N$. 

Let us blow up $\X$ at $N$ and denote by $E$ the exceptional component. Let~$\mathcal{L}$ be a line bundle on
$\bar{\X}$ such that $\mathcal{L}|_{\X_{1}}$ and $\mathcal{L}|_{\X_{2}}$ are theta characteristics on~$\X_{1}$
and $\X_{2}$ respectively, and $\mathcal{L}|_{E}=\Ox[E](1)$.
The degree  of $\mathcal{L}$ is $g-1$ on $\bar\X$. The morphism
$\alpha:\mathcal{L}^{2}\to\dualsheave[\bar{X}]$ vanishes on $E$ and is the
isomorphism between $\mathcal{L}_{i}^{2}$ and~$\dualsheave[\X_{i}]$ on $\X_{i}$.
             
Moreover, the spin structure $\mathcal{L}$ is odd if the parities of $\mathcal{L}|_{\X_{1}}$ and 
$\mathcal{L}|_{\X_{2}}$ are distinct, and even otherwise.
\end{ex}

Let $\barspin$ be the {\em moduli space of stable spin curves}. It is a natural compactification which
projects to
$\barmoduli$.  Let us recall some important properties of $\barspin$.

\begin{prop}\label{proposition:Cor5.2}(\cite[Proposition~5.2]{MR1082361})
 The variety $\barspin$ is normal, projective and  is the disjoint union of the even part $\barspin^{+}$  and
the odd part $\barspin^{-}$. Moreover the  forgetful map $\pi:\barspin\to\barmoduli$ is a finite map.
\end{prop}

In the rest of this section, we will not precise the morphism $\alpha$ and we will suppose that  our spin
structures are square roots of the canonical bundle.

\paragraph{Spin structure associated to limit differentials on curves of compact type.}

In this paragraph we compute the spin structure associated to a limit differential  (see
Definition~\ref{definiton:diffLimite}) on a curve of compact type
which has only zeros and poles of even orders. But a limit differential of type $(2l_{1},\cdots,2l_{n})$
on a stable marked curve of compact type is determined, up to
multiplication by constants, by the marked curve (see
Corollary~\ref{corollaire:uniciteLimDiffSurTypeCompacte}). Hence the invariant that we will construct will
only depends on  the marked curve, and be well defined for the limit pointed differentials of compact type.

On a smooth curve $\X$, we can associate a spin structure to an Abelian differential
with only even orders of zeros by
\begin{equation} \label{fonction:StrateGeneraleVersSpin}
\varphi:\omoduli[g](2l_{1},\cdots,2l_{n})\to\spin;\quad \left(\X,\omega \right)\mapsto
\left(\X,\mathcal{L}_{\omega}:=\Ox\left(\frac{1}{2}\divisor{\omega}\right)\right).
\end{equation}
We extend this definition to the case of limit differentials on curves of compact type.

\begin{defn}\label{definition:spinStructureSurCourbeTypeCompact}
 Let $(\X,\omega,Z_{1},\cdots,Z_{n})$ be a limit differential in the closure of the stratum
$\omoduliincp{2l_{1},\cdots,2l_{n}}$. Let $\pi:\bar\X\to\X$ be the blow-up of $\X$ at every node of $\X$. Then
the {\em spin structure $\mathcal{L}_{\omega}$ associated to $\omega$} is defined by the following
restrictions on $\bar\X$.
\begin{itemize}
 \item If $E$ is an exceptional component of $\bar\X$, then  $\mathcal{L}_{\omega}|_{E}=\Ox[E](1)$.
 \item If $\X_{i}$ is an irreducible component of $\X$, then
$\mathcal{L}_{\omega}|_{\X_{i}}=\Ox[\X_{i}]\left(\frac{1}{2}\divisor{\omega}\right)$.
\end{itemize}
\end{defn}

We now verify that the line bundle $\mathcal{L}_{\omega}$ associated to $\omega$ is indeed a spin structure
in the sense of Definition~\ref{definition:spinCurveb}.
\begin{proof}
Let $\X_{i}$ be an irreducible component of $\X$. The line bundle $\mathcal{L}_{\omega}|_{\X_{i}}$ is by
definition a square root
of the canonical bundle of $\X_{i}$. It remains to check that the degree of $\mathcal{L}_{\omega}$ is $g-1$.
We denote by $\nodeSet[e]\subset\nodeSet$ the subset of nodes of $\X$ which have been blown up to give the
decent
curve. At each node $N_{1}\sim N_{2}$, the compatibility
condition~\eqref{equation:conditionDeCompatibiliteGeneral}
$\deg_{N_{1}}(\omega)+\deg_{N_{2}}(\omega)=-2$ implies that 
\begin{eqnarray*}
\deg(\mathcal{L}_{\omega}) &=  \sum_{\X_{i}\in\Irr(\X)} \deg(\mathcal{L}_{\omega}|_{\X_{i}}) +
\#\nodeSet[e] \\
 &= g-1-\#\nodeSet +\# \nodeSet[e].
\end{eqnarray*}
It follows from this equation that $\deg(\mathcal{L}_{\omega})=g-1$ if and only if every node of~$\X$ is blown
up.
\end{proof}

Of course, this notion could be useful only if it behaves well in families. This is the content of the
following lemma.
\begin{lemma}\label{lemme:spinEnFamille}
Let $(f:\famcurv\to\Delta^{\ast},\famomega,\seczero_{1},\cdots,
\seczero_{n})$ be a family of pointed
differentials in $\omoduliincp{2l_{1},\cdots,2l_{n}}$. Let
$(f:\famcurv\to\Delta^{\ast},\mathcal{L}_{\famomega}\to\famcurv)$ be the associated family of theta
characteristics inside $\spin$.

If the stable limit of $\famcurv$ is of compact type, then the spin
structure associated to the pointed limit differentials of this family  coincides
with the restriction to the special curve of the completion of  $\mathcal{L}_{\famomega}$ inside $\barspin$.
\end{lemma}

\begin{proof}
Let $ \left(\famcurv,\famomega,\seczero_{1},\cdots,\seczero_{n})\right)$ be a family inside
$\omoduliincp{2l_{1},\cdots,2l_{n}}^{\epsilon}$, and $(\X,\omega,Z_{1},\cdots,Z_{n})$ be its limit
differential.
Above $\Delta^{\ast}$, the associated theta characteristics are given by the bundle
$\Ox[\famcurv](\frac{1}{2}\divisor{\famomega})$. Let us remark that according to
Proposition~\ref{proposition:Cor5.2}, there exists an extension of $\mathcal{L}$ above the decent curve
$\bar\X$ in such
a way that $\mathcal{L}|_{\bar\X}$ is a spin structure on $\bar\X$. By
Theorem~\ref{theoreme:ProlongFibreEnDroite}, there exists only one such extension. 
Since the line bundle defined in
Definition~\ref{definition:spinStructureSurCourbeTypeCompact} is such extension, this concludes the proof .
\end{proof}

A direct application of this result, is the fact that the {incidence variety compactifications} of the even
and odd components of $\omoduliincp{2l_{1},\cdots,2l_{n}}$ remain disjunct above the set of curves of
compact type.

\begin{theorem}\label{theoreme:BordDesStratesDansSpin}
  Let $n\geq3$ and $(\X,\omega,Z_{1},\cdots,Z_{n})$ be a stable differential of compact type in 
$\obarmoduliincp{2l_{1},\cdots,2l_{n}}$. Then the parity of the spin structure~$\mathcal{L}_{\omega}$
associated
to $(\X,\omega,Z_{1},\cdots,Z_{n})$ is $\epsilon$ if and only if $(\X,\omega,Z_{1},\cdots,Z_{n})$ is in
$\obarmoduliincp{2l_{1},\cdots,2l_{n}}^{\epsilon}$.
\end{theorem}

Let us remark that Theorem~\ref{theoreme:BordDesStratesDansSpin} remains true with minor modifications even
for $n\leq 2$ zeros. But the fact that in these cases
the strata contain three connected components complicates the statement.
 
\begin{proof}
 By Corollary~\ref{corollaire:uniciteLimDiffSurTypeCompacte}, we can associated a unique (up to multiplicative
constants) limit differential to $(\X,\omega,Z_{1},\cdots,Z_{n})$. By Lemma~\ref{lemme:spinEnFamille}, this
limit differential has parity $\epsilon$ if and only if it lies in the
closure of $\omoduliincp{2l_{1},\cdots,2l_{n}}^{\epsilon}$
\end{proof}

Let us conclude this paragraph by  describing  the spin structures associated to the limit differentials
of
the minimal strata above the generic curves of $\delta_{i}$ for~$i\geq1$.

\begin{prop}\label{proposition:BordStratesMinimalDansSpin}
Let $\X:=\X_{1}\cup\X_{2}/N_{1}\sim N_{2}$ be a curve in $\delta_{i}$ and let
$\bar\X:=\X_{1}\cup E\cup\X_{2}$ the blow-up of $\X$ at the node. 

The spin structure $\mathcal{L}$ associated to the limit differential $(\X,\omega,Z)$ in the
boundary of the minimal stratum is given by
\begin{equation}\label{equation:LimitSpinMinimalCasZeroLisse}
 \mathcal{L}|_{\X_{i}}=\Ox[\X_{i}]((g-1)Z-g_{j}N_{i}),\quad
\mathcal{L}|_{\X_{j}}=\Ox[\X_{j}]((g_{j}-1)N_{j}),\quad
\mathcal{L}|_{E}=\Ox[E](1),
\end{equation}
where $(i,j)=(1,2)$ or $(i,j)=(2,1)$.
\end{prop}

\begin{proof}
The fact that the point $Z$ is not contained in $E$ has been proved in
Corollary~\ref{corollaire:zeroJamaisSurLePontFaible}. So we can suppose that $Z\in\X_{1}$. Then 
$\omega$ is a limit differential with a zero of order $2g-2$ at $Z$, if it has a pole of order $2g_{2}$ at
 $N_{1}$. But by Theorem~\ref{theoreme:PlomberieCylindriqueSansResidu} the form $\omega$ has a zero of order
$2g_{2}-2$ at
$N_{2}$. The description of the restrictions of $\mathcal{L}$ is now given in
Definition~\ref{definition:spinStructureSurCourbeTypeCompact}.
\end{proof}

\subsection{Irreducible Pointed Differentials.}\label{section:InvariantDeArf}

The main purpose of this paragraph is to extend the Arf invariant to the set of irreducible  marked
curves(see
Definition~\ref{definition:ArfGen}). This implies that the {incidence variety
compactifications} of the
even and odd connected components of every strata remain disjoint above this locus of curves (see Theorem~\ref{theoreme:InvDeArfGene}).

We first recall some basic facts about the Arf invariant of Abelian differentials. It was first
investigated in \cite{MR588283} (see also \cite{MR2261104}). 

Through this paragraph, we will use the following notations. The pair $(\X,\omega)$ denotes an Abelian
differential
or an irreducible stable differential with only meromorphic
nodes. For a smooth simple closed path
$\gamma:\left[0,1 \right]\to \X,$
we denote by 
$$G(\gamma):\left[0,1 \right]\to S^{1}$$
  the {\em Gauss map} associated to $\gamma$ by the differential $\omega$ and by
$$\ind(\gamma):=\deg(G(\gamma)) \mod 2$$
 the {\em index} of $\gamma$.

\begin{defn}
 Let $(\X,\omega)$ be an Abelian differential of genus $g$ and let
$\left(a_{1},\cdots,a_{g},b_{1},\cdots,b_{g}\right)$
be a symplectic basis of $H_{1}(\X,\ZZ)$ composed by smooth and simple curves which miss the zeros of
$\omega$. The {\em Arf invariant} of
$(\X,\omega)$ is 
\begin{equation}
 \Arf(\X,\omega):=\sum_{i=1}^{g}(\ind(a_{i})+1)(\ind(b_{i})+1)  \mod 2.
\end{equation} 
\end{defn}

Johnson has shown that for every differential in $\omoduli(2l_{1},\cdots,2l_{n})$, the
Arf invariant is independent of the choice of the symplectic basis. Moreover, he
showed that the Arf invariant coincides with the   parity of the theta characteristic associated to
the differential $\omega$ (see Equation~\eqref{fonction:StrateGeneraleVersSpin}).

We now generalise the Arf invariant in the case of irreducible pointed stable
differentials. Note that such differentials have only poles of order one at every node.

First we define the set of curves which generalises the symplectic basis.
 Let us recall that the normalisation of a nodal curve $\X$ is denoted by $\nu:\tilde{\X}\to\X$ and the
preimages of a node $N_{i}$ by $\nu$ are denoted by $N_{i,1}$ and $N_{i,2}$.
 
 \begin{defn}
 Let $\X$ be an irreducible stable curve of genus $g$ with $k$ nodes $N_{1},\cdots,N_{k}$.
An {\em admissible symplectic system of curves} $(a_{1},\cdots,a_{g},b_{1},\cdots,b_{g})$ on $\X$ is
an ordered set of simple smooth curves on $\X$ satisfying the three following properties. 
 \begin{itemize}
 \item[i)] The curves $(\nu^{\ast}a_{k+1},\cdots,\nu^{\ast}a_{g},\nu^{\ast}b_{k+1},\cdots,\nu^{\ast}b_{g})$
form a basis of $H_{1}(\tilde{\X},\ZZ)$.
 \item[ii)] For every $i,j\in\left\{1,\cdots,g\right\}$ we have
 $$a_{i}\cdot b_{j}=\delta_{ij},\quad a_{i}\cdot a_{j}=0, \text{ and }b_{i}\cdot b_{j}=0.$$
 \item[iii)] For $i\leq k$, we have $\nu^{\ast}a_{i}(0)=N_{i,1}$, $\nu^{\ast}a_{i}(1)=N_{i,2}$ and the limits
$$\lim_{t=0}\frac{\partial\nu^{\ast}a_{i}}{\partial t}(t) \text{ and }
\lim_{t=1}\frac{\partial\nu^{\ast}a_{i}}{\partial t}(t)$$ exist.
 \end{itemize}
  The curve $a_{i}$ is called an {\em admissible path of the node $N_{i}$}.
 \end{defn}

Note that  an admissible symplectic system of curves on a smooth curve $\X$ is a symplectic basis of 
$H_{1}(\X,\ZZ)$.

We now describe the behaviour of the Gauss map of the  admissible paths.
\begin{lemma}\label{lemme:AppGaussAuxNoeuds}
Let $(\X,\omega)$ be an irreducible stable differential with only meromorphic nodes, let $N_{0}$ be a node of
$\X$ and
let $\gamma$ be an admissible path for $N_{0}$.

Then, the limits $$\lim_{t\to 0}G(\gamma)(t) \text{ and }\lim_{t\to 1}G(\gamma)(t)$$ exist and coincide with
the
direction of the flat cylinder associated to $N_{0}$.
\end{lemma}

\begin{proof}
Since the Gauss map of a smooth path is continuous, there exist limits of $G(\gamma)(t)$ for $t\to 0$ and
$t\to 1$. Since the tangent vector of $\gamma$ has a limit, the path cannot turn around the node
infinitely many times. This implies that the limit for the Gauss map is the direction of the flat cylinder
associated to this node. 
\end{proof}

Lemma~\ref{lemme:AppGaussAuxNoeuds} allows us to define the index of the paths intersecting the nodes in an
admissible system of curves.

\begin{defn}\label{definition:ArfGen}
Let $(\X,\omega)$ be an irreducible differential with meromorphic nodes, $N_{0}$ be a node of $\X$ and
$\gamma$
be an admissible path for $N_{0}$.

The {\em index} of $\gamma$ is $$\ind(\gamma):=\deg(G(\gamma)) \mod 2.$$
\end{defn}

We can now extend the notion of Arf invariant. 

\begin{defn}\label{definition:extensionInvariantArf}
Let $(\X,\omega)$ be a stable differential such that $\X$ is irreducible and $\omega$ has a simple pole at every node of $\X$.
Let $\left(a_{1},\cdots,a_{g},b_{1},\cdots,b_{g}\right)$ be an admissible symplectic system of  curves for $(\X,\omega)$.

The {\em generalised Arf invariant} of $(\X,\omega)$ is
\begin{equation}
 \Arf(\X,\omega):=\sum_{i=1}^{g}(\ind(a_{i})+1)(\ind(b_{i})+1) \mod 2.
\end{equation} 
\end{defn} 

We show that the generalised Arf invariant does not depend on the choice of the
admissible system.

\begin{theorem}\label{theoreme:InvDeArfGene}
Let $(\X,\omega,Z_{1},\cdots,Z_{n})\in\obarmoduliincp{2d_{1},\cdots,2d_{n}}$ be a stable differential such
that $\X$ is irreducible with $k$ nodes $N_{1},\cdots,N_{k}$.

Then the generalised Arf invariant only depends on $(\X,\omega)$ and $\Arf(\X,\omega)=\epsilon$ if and only if
$(\X,\omega)$ is in the closure of a component of
$\omoduli(2d_{1},\cdots,2d_{n})$ with associated spin structure of parity $\epsilon$.
\end{theorem}

We prove the result by recurrence on the number of nodes. The main tool for the recurrence step is the
Plumbing cylinder construction of Section~\ref{section:PlomberieCylindrique} (see in particular
Theorem~\ref{theoreme:PlomberieCylindriqueSansResidu}).

\begin{proof}
If $\X$ has no nodes, then the generalised Arf invariant of $\X$ coincides with the usual Arf invariant. This
implies the result for a smooth differential.

Let us suppose that  Theorem~\ref{theoreme:InvDeArfGene} has been proved in the case of $k-1$
nodes and let $(\X,\omega,Z_{1},\cdots,Z_{n})$ be a
differential with $k$ nodes satisfying the hypothesis of Theorem~\ref{theoreme:InvDeArfGene}. Let
$\left(a_{1},\cdots,a_{g},b_{1},\cdots,b_{g}\right)$ be an admissible symplectic system for $(\X,\omega)$.

Let $V$ and $W$ be neighbourhoods of $N_{k,1}$ and $N_{k,2}$ respectively, such that $U:=V\cup W$ and
$\omega|_{U}$ satisfy the
hypothesis of Lemma~\ref{lemme:PlomberieCylindriqueOk}. Without loss of generality, we can suppose that $U\cap
a_{i}=\emptyset$ for all $i\neq k$ and $U\cap b_{j}=\emptyset$ for all~$i\in\lbrace 1,\cdots, g\rbrace$.  
Moreover,
let $\theta_{k}$
be the direction of the cylinders associated to~$\omega$ at $N_{k}$. We can suppose that
$G(a_{k})(t)\in\left]\theta_{k}-\frac{\pi}{4},\theta_{k}+\frac{\pi}{4}\right[$ for every $t$ such that
$a_{k}(t)\in
U$. In particular, the path~$a_{k}$ meets only one time the
boundaries of $V$ and $W$.

Since $(\X,\omega,Z_{1},\cdots,Z_{n})$ verifies the hypotheses of Lemma~\ref{lemme:PlomberieCylindriqueOk},
we can smooth this differential. In particular, the set $U$ is replaced by a
flat cylinder~$U'$ and $a_{k}$
by any simple closed smooth curve which coincide with $a_{k}$ outside of~$U'$.

By induction, the generalised Arf invariant is well defined on this curve. In particular, it does not depend
on the chose of $a_{k}$.
Hence it remains to show that the
index of every  curve in the new admissible symplectic system coincide with the  index of the corresponding
curve in the old admissible system. The indices of every
curve distinct from $a_{k}$ are clearly invariant under the plumbing cylinder construction. It remains to
show that the index of $a_{k}$ and $\tilde{a_{k}}$ coincide. But we can choose $\tilde{a_{k}}$  such that in
$U'$ the Gauss map satisfies  $G(\tilde{a_{k}})(t)\in\left]\theta-\frac{\pi}{2},\theta+\frac{\pi}{2}\right[$.
In particular, it is clear that the index of $\tilde{a_{k}}$ coincide with the index of $a_{k}$.

This shows that the generalised Arf invariant is a well defined invariant of~$(\X,\omega)$ and coincide with
the Arf invariant of any partial smoothing of $(\X,\omega)$ at a node. By induction these smoothings are in
the closure of a component of $\omoduli(2d_{1},\cdots,2d_{n})$ with associated theta characteristic of parity
$\epsilon$.
\end{proof}


\section{Kodaira Dimension of Some Strata of $\pomoduli$.}

In this section, we compute the Kodaira dimension of some strata of $\pomoduli$.  We show in
Theorem~\ref{theoreme:DimKodairaStraPeutConditions} that the
strata which 
 `impose few conditions on the differentials' (see the theorem loc. cit. for a
precise definition) have negative Kodaira dimension.  
In Theorem~\ref{theoreme:DimensionProjectionStrates}, we compute  the dimension of the
projection of every connected component of every stratum of $\omoduli$ to $\moduli$. This result implies that
 the strata $\pomoduli(k_{1},\cdots,k_{g-1})$ different from $\pomoduli^{\even}(2,\cdots,2)$ are of general
type
when $\moduli$ is of general type (see Theorem~\ref{theoreme:DimKodairaRevetementFini}).  
 
The end of this section is devoted to the computation of the Kodaira dimension of other strata. In 
Proposition~\ref{proposition:DimKodaira(g-1,1)}, we show that $\pomoduli(g-1,1,\cdots,1)$ is of general type 
when $\moduli$ is of general type. We give the Kodaira dimension of $\pomoduli^{\hyp}(g-1,g-1)$ in
Proposition~\ref{proposition:dimKodairaHypNeg}.
Moreover, we give the Kodaira dimension of every odd (see Corollary~\ref{corollaire:dimKodairaOddDeux}) and
every even (see Proposition~\ref{proposition:dimKodairaEvenDeux}) component of $\pomoduli(2,\cdots,2)$. 

\paragraph{Generalities.}

We first recall the definition of the Kodaira dimension of complex varieties $Y$ following \cite{MR0506253}.
The (complex) dimension of
$Y$ will be denoted by $\dim Y$.
\begin{defn}
 Let $Y$ be a smooth irreducible compact complex variety. The {\em Kodaira dimension $\kappa(Y)$ of $Y$} is 
\begin{equation}
\kappa(Y) =\left\{
  \begin{array}{l}
   -\infty,  \text{if } H^{0}(Y,mK_{Y})=0  \text{ for all $m\geq0$}\\
  \min  \left\{n\in\NN\cup\lbrace0\rbrace :\frac{h^{0}(Y,mK_{Y})}{m^{n}}  \text{ is bounded} \right\},  
\text{ otherwise.}
  \end{array}\right.
\end{equation}

The variety $Y$ is of {\em general type} if $\kappa(Y)=\dim(Y)$.
\end{defn}

Since we will be mainly interested in singular non-compact varieties, we extend the notion of Kodaira
dimension to
singular and non-compact varieties.
  If $Y$ is a singular compact complex variety, then its {\em Kodaira dimension  $\kappa(Y)$} is the
Kodaira dimension of any non-singular model of $Y$.
  If $Y$ is a non-compact complex variety, then its {\em Kodaira dimension  $\kappa(Y)$} is the Kodaira
dimension of any non singular model of any compactification of $Y$.
Let us remark that, as the Kodaira dimension is a birational invariant, the two preceding definitions make
sense.

The Kodaira dimension of a given complex variety $Y$ is in general difficult to compute. On the other
hand it is easily proved that $\kappa(Y_{1}\times Y_{2})=\kappa(Y_{1})+\kappa(Y_{2})$. One could hope that a
similar statement holds for more general fibre spaces and for maps $\pi:Y\to Z$ which behave like bundle maps.
This
is what we explain now.

The first important notion is the one of {\em fibre space} of complex varieties. This is a proper and
surjective morphism $\pi:Y\to Z$ of reduced analytic spaces such that the general fibre of $\pi$ is connected.
Moreover, a  meromorphic mapping $\varphi:Y\to Z$ is {\em generically surjective} if the projection
$\pi_{\varphi}:G_{\varphi}\to Z$ of the graph of $\varphi$ to $Z$ induced by the projection of $Y\times Z$ to
$Z$ is surjective.

Let us recall, that a fibre space $\pi:Y\to Z$ is {\em uniruled} if a generic fibre $Y_{z}$ of $\pi$ is a
projective line. If a space is uniruled, then its Kodaira dimension is negative.

Let us recall that the Kodaira dimension of a fibre space can not be larger than the Kodaira
dimension of the base plus the Kodaira dimension of a generic fibre (see \cite[Theorem~6.12]{MR0506253}).
\begin{theorem}\label{theoreme:SousAdditiviteKodaria}
Let $\pi:Y\to Z$ be a fibre space of complex varieties. There exists an open dense set $V\subset Z$ such that
for any point $z\in V$ the inequality
\begin{equation}
\kappa(Y) \leq  \dim (Z) + \kappa(\pi^{-1}(z)) 
\end{equation}
holds.
\end{theorem}
In particular, if the Kodaira dimension of a generic fibre or of the basis of a fibre space is
negative, then the total space has negative Kodaira dimension.

A very important open problem is to determine determine the best lower bound in the preceding settings.
\begin{conj}[Iitaka conjecture or $C_{n}$ conjecture.]
   Let   $\pi: Y\to Z$  be a fibre space of an $n$-dimensional
algebraic manifold   $Y$  over an algebraic manifold  $Z$. Then we have
\begin{equation}\label{equation:ConjectureIitaka}
 \kappa(Y) \geq  \kappa (Z) + \kappa(Y_{z}),
\end{equation}
 for a generic fibre $Y_{z}:=\pi^{-1}(z)$.
\end{conj}

Even if the conjecture is known to be false in general (see \cite[Remark~15.3]{MR0506253}), it holds in very
important cases. The first one is when $\pi:Y\to Z$ is a generically surjective map of complex varieties of
the same dimension (see \cite[Theorem~6.10]{MR0506253}).
\begin{theorem}\label{theoreme:DimKodairaRevetementFini}
Let $\pi:Y\to Z$ be a generically surjective meromorphic mapping of complex varieties such that $\dim Y =\dim
Z$. 
Then we have the inequality
\begin{equation}
\kappa(Y)\geq \kappa (Z).
\end{equation}
\end{theorem}

The second important case of this conjecture has been proved by Viehweg. He proved that the
Iitaka conjecture holds as soon as $Z$ is of general type.
\begin{theorem}[\cite{MR641815}]\label{theoreme:KodairaViehweg}
Let $\pi:Y\to Z$ be a generically surjective meromorphic mapping of complex varieties such that $\kappa(Z)
=\dim Z$. 
Then  we have  the inequality
\begin{equation}
 \kappa(Y) \geq  \kappa (Z) + \kappa(Y_{z}),
\end{equation}
for a generic fibre $Y_{z}:=\pi^{-1}(z)$
\end{theorem}

\paragraph{The strata of $\pomoduli$.}
The rest of this section is devoted to the computation of the Kodaira dimension of several strata of the
moduli
space of Abelian differentials.

Let us first remark that the Kodaira dimension of the principal stratum follows directly from
Theorem~\ref{theoreme:SousAdditiviteKodaria}.
\begin{prop}
The Kodaira dimension of the moduli spaces  $\pomoduli$ and the principal strata $\pomoduli(1,\cdots,1)$ is
$-\infty$.
\end{prop}

\begin{proof}
Since $\pobarmoduli\to\barmoduli$ is a bundle with fibre $\PP^{g-1}$, the result follows from
Theorem~\ref{theoreme:SousAdditiviteKodaria}. Since the closure of the principal stratum is $\pobarmoduli$,
this implies the result for the principal stratum.
\end{proof}

In order to apply the Theorem of Iitaka-Viehweg, we have to determine for which strata the forgetful map
$\pi:\omoduli\to\moduli$ is generically surjective. In fact, we compute the dimension of the image of
every connected
component of the strata of $\omoduli$ via the forgetful map.
This theorem greatly generalises a previous result of Chen (see \cite[Proposition~4.1]{MR2746470}).

\begin{theorem}\label{theoreme:DimensionProjectionStrates}
Let $g\geq2$ and $S$ be a connected component of the stratum $\omoduli(k_{1},\cdots,k_{n})$. The
dimension $d_{\pi(S)}$ of the
projection of $S$ by the forgetful map $\pi:\omoduli\to\moduli$ is 
\begin{equation}\label{equation:DimensionProjectionStrates}
d_{\pi(S)} =
  \begin{cases}
   2g-1, & \text{if } S=\omoduli(2d,2d)^{\hyp} \\
   3g-4, & \text{if } S=\omoduli(2,\cdots,2)^{\even} \\
   2g-2+n,       & \text{if } n< g-1 \text{ and }S\neq\omoduli(2d,2d)^{\hyp} \\
   3g-3,   & \text{if } n\geq g-1 \text{ and the parity of $S$ is not even.}
  \end{cases}
\end{equation}
\end{theorem}

This theorem is  proved by degeneration.
The main ingredients are the plumbing cylinder construction of Section~\ref{section:PlomberieCylindrique}, the
explicit description of the spin structures
on the curves of compact type (see Section~\ref{section:Spin}) and the local parametrisation of $\barmoduli$
given by \cite[Theorem~3.17]{MR2807457}.

Before proving the theorem let us introduce the main type of stable curve that we  use in the proof.
\begin{defn}
 Let $(\X_{1},N_{1,1})$ and $(\X_{g},N_{g-1,2})$ be $2$ one-marked elliptic curves and  let
$(\X_{2},N_{1,2},N_{2,1}),\cdots,(\X_{g-1},N_{g-2,2},N_{g-1,1})$ be $g-2$ two-marked elliptic curves.
The {\em snake curve} $\X$ defined by these elliptic curves (see~Figure~\ref{figure:CourbeSerpent}) is 
 $$\X := \left(\bigcup_{i=1}^{g} \X_{i}\right)/\left(N_{i,1}\sim N_{i,2}\right).$$

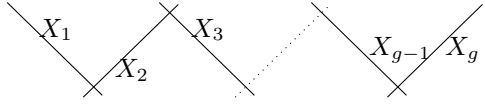
\begin{figure}[ht]
\centering
 \begin{tikzpicture}[>=stealth',shorten >=1pt,auto,node distance=2.8cm]

 \foreach \i in {1,3,...,5}
     \node(A\i) at (\i-.2,0){};
\foreach \i in {1,3,...,5}
     \node[above,right](B\i) at (\i+1.2,1.5){};

 \draw(A1) edge node[below]{$\X_{2}$} (B1);

 \draw[dotted] (A3) to (B3);
 \draw (A5) edge node[right]{$\X_{g}$} (B5);

\foreach \i in {0,2,...,4}
     \node(C\i) at (\i-.2,1.5){};
\foreach \i in {0,2,...,4}
     \node[below,right](D\i) at (\i+1.2,0){};

 \draw (C0) edge node[above]{$\X_{1}$} (D0);
 \draw (C2) edge node[above]{$\X_{3}$} (D2);
\draw (C4) edge node[right]{$\X_{g-1}$} (D4);
\end{tikzpicture}
\caption{The snake curve $\X$.}
\label{figure:CourbeSerpent}
\end{figure}
\end{defn}

\begin{proof} We begin the proof by treating the case of the hyperelliptic strata $\hyperell$. 

\paragraph{The hyperelliptic strata.}
The  hyperelliptic locus $\hyperell\subset\moduli$ of genus $g$ has dimension $2g-1$. Since the projections 
of each of the hyperelliptic strata $\omoduli(2g-2)^{\hyp}$ and 
$\omoduli(2d,2d)^{\hyp}$ to $\moduli$ are $\hyperell$, they have dimension $2g-1$.

\paragraph{}
From now on, $S$ will be a   an non
hyperelliptic connected component of the stratum  $\omoduli(k_{1},\cdots,k_{n})$. 

\paragraph{The strata $\omoduli(k_{1},\cdots,k_{n})$ with $n\geq g$.}
Let us remark that if $n\geq g$, then  the stratum $S':=\omoduli(k_{1}+k_{n},\cdots,k_{n-1})$
lies in the boundary of $S$. So if the  dimension of the projection of $S$ is $d$, the dimension of the
projection of $S'$ is at least $d$. This implies that it suffices to prove the theorem for the strata with at
most $g-1$ zeros. 

From now on, we suppose that $n\leq g-1$. 

\paragraph{The connected strata $\omoduli(k_{1},\cdots,k_{n})$.}
Let $\X$ be the snake curve from above where the points
$N_{i-1,2}$ are points of $2(g-i)$-torsion of $(\X_{i},N_{i,1})$.

Let $\omega$ be the differential  on $\X$
defined by the following restrictions. 
\begin{itemize}
 \item For $i=1$, let $\omega|_{\X_{1}}$ be a differential on $\X_{1}$ with a pole of order $k_{1}$ at
$N_{1,1}$ and
a zero
$Z_{1}$ of order $k_{1}$.
\item For $i\in\left\{ 2,\cdots,n\right\}$, let $\omega|_{\X_{i}}$ be a differential such that the divisor is 
$$\divisor{\omega_{i}}=k_{i}Z_{i}+\left(\sum_{j<i}k_{j}-2(i-1)\right)N_{i-1,2} -\left(\sum_{j\leq
i}k_{j}-2(i-1)\right)N_{i,1},$$ 
where $Z_{i}\in\X_{i}\setminus\left\{N_{i-1,2},N_{i,1}\right\}$.
\item For $i\in\left\{ n+1,\cdots,g-1\right\}$, the differential $\omega|_{\X_{i}}$ is the differential with
divisor 
$$\divisor{\omega_{i}}=2(g-i)N_{i-1,2} -2(g-i)N_{i,1}.$$ 
\item For $i=g$, the differential $\omega_{g}$ is simply the holomorphic differential of $\X_{g}$.
\end{itemize}

Let us remark that the differentials $\omega|_{\X_{i}}$ exist and satisfy the Compatibility
Condition~\eqref{equation:conditionDeCompatibiliteGeneral}, that is 
$\ord_{N_{i,1}}(\omega|_{\X_{i}})=\ord_{N_{i,2}}(\omega|_{\X_{i+1}})=-2$ for every node~$N_{i,1}\sim
N_{i,2}$. 
Moreover, the differentials $\omega_{i}$ have no residues, so according
 to Theorem~\ref{theoreme:PlomberieCylindriqueSansResidu}, they form a limit differential $\omega$ 
which can be smoothed in the stratum $\omoduli(k_{1},\cdots,k_{n})$.

We now construct a neighbourhood of $\X$ of dimension $2g-2+n$ such that every curve in
this neighbourhood  possesses a limit differential of type $(k_{1},\cdots,k_{n})$. 

Let us first give a parametrisation of a small neighbourhood $U$ of $\X$ in $\barmoduli$ (see
\cite[Theorem~3.17]{MR2807457}). Let $(t_{1},\cdots,t_{3g-3})\in\Delta^{3g-3}$ be a
parametrisation of $U$ 
such that the coordinates of $\X$ are $(0,\cdots,0)$ and satisfying the following properties. 
\begin{itemize}
 \item The first $g$ variables $t_{1},\cdots,t_{g}$ parametrise the deformations of the elliptic curves
$(\X_{1},N_{1,1}),\cdots,(\X_{g},N_{g,1})$. 
\item The $g-2$ variables $t_{g+1},\cdots,t_{2g-2}$ parametrise the deformations of the nodes
$N_{1},\cdots,N_{g-1}$. Alternatively, they parametrise the deformations of $(\X_{i},N_{i-1,2},N_{i,1})$
which leave the curve $\X_{i}$ fixed. 
\item The $g-1$ last parameters $t_{2g-1},\cdots,t_{3g-3}$ parametrise the smoothings of the nodes of $\X$.
\end{itemize}

Observe that the existence of a limit differential as previously defined does not depend on the normalisation
of the elliptic curves. Therefore, we can deform the differential $\omega$ above  the curves of parameter
equal to $(t_{1},\cdots,t_{g},0,\cdots,0)$   in such a way that it remains a limit differential of type
$(k_{1},\cdots,k_{n})$.

Now  let us remark that for $i\in
\left\{n+1,\cdots,g-1\right\}$, the points
$N_{i-1,2}$ have to be points of $2(g-i)$-torsion of $(\X_{i},N_{i,1})$. On the other hand, the
points~$N_{i,2}$ and $N_{i+1,1}$ can move freely on $\X_{i}$ for $i\in
\left\{1,\cdots,n\right\}$.
Hence, using the second characterisation
of the
deformations of the nodes, this means that only the deformations of the $n-1$ first nodes of $\X$ are allowed.
 
It follows from Theorem~\ref{theoreme:PlomberieCylindriqueSansResidu}, that the smoothings
of the nodes at the limit differential $(\X',\omega')$ of parameter $(t_{1},\cdots,t_{g+n-1},0,\cdots,0)$
is a differential in $S$.

Summarising this discussion, we have shown, that every curve with coordinates
$(t_{1},\cdots,t_{g+n-1},0,\cdots,0,t_{2g-1},\cdots,t_{3g-3})\subset\Delta^{3g-3}$ has a limit differential
in the closure of $S$. Since this neighbourhood of $\X$ has dimension $2g-2+n$, this proves
Theorem~\ref{theoreme:DimensionProjectionStrates} in the case of connected strata.

\paragraph{The non-connected strata.}
Next, we deal with the non-connected strata of $\omoduli$ determined in \cite{MR2000471}. The problem of the
last
argument is that we do not know a priori in the
boundary of which connected component is the limit differential $(\X,\omega)$ that we have constructed.

Recall from Definition~\ref{definition:spinStructureSurCourbeTypeCompact} that on a curve of compact type
$\X$, a spin structure is
determined by its restrictions on every irreducible component of $\X$. More precisely, if $\omega$ is a limit
differential on $\X$ with only  zeros and poles of even orders, then the
theta characteristic on an irreducible component $\X_{i}$ of $\X$ is $\Ox[\X_{i}]\left(
\frac{1}{2}\divisor{\omega|_{\X_{i}}}\right)$. Moreover, we have shown in
Theorem~\ref{theoreme:BordDesStratesDansSpin} that the parity of a spin
structure is given by the sum of the parities of these restrictions and is invariant under deformation.

\paragraph{The components of the strata $\omoduli(2,\cdots,2)$.}
We first prove that the dimension of the image of $\omoduli^{\odd}(2,\cdots,2)$ under the forgetful map is
$3g-3$. The construction of the differential on the snake curve in the case of connected strata can be
performed in the case of the strata  $\omoduli(2,\cdots,2)$. Hence it suffices to
show that this differential has odd parity to prove this case. On the $g-1$ first curves
$\X_{1},\cdots,\X_{g-1}$, the theta characteristics are given by the line bundles
$\Ox[\X_{i}](Z_{i}-N_{i,1})$. In particular,
they have even parity. On the other hand, the theta characteristic on the  curve $\X_{g}$ is $\Ox[\X_{g}]$,
which has odd parity. Since the parity of $\omega$ is given by the sum of the parities, it has odd parity. 

We now deal with the case of the component  $\omoduli^{\even}(2,\cdots,2)$.
Let us remark that the dimension of the projection of this component is at most
$3g-4$. Indeed, let
$(\X,\omega)\in\omoduli^{\even}(2,\cdots,2)$. Then clearly, $\omega\in H^{0}(\X,\frac{1}{2}\divisor{\omega})$.
This implies that $h^{0}(\X,\frac{1}{2}\divisor{\omega})\geq 2$. The locus of curves having such theta
characteristic is  a divisor of $\moduli$ according to \cite{MR937985}. So it remains to show that
$\dim(\pi(\omoduli^{\even}(2,\cdots,2)))\geq 3g-4$. We prove this by induction on the genus of the curve.

In genus $3$, the even stratum $\omoduli[3]^{\even}(2,2)$ coincides with the hyperelliptic stratum
$\omoduli[3]^{\hyp}(2,2)$. So the claim follows from the description of the hyperelliptic strata.

Let us do the induction step.
Let $\tilde\X$  be generic curve in the image of $\omoduli[g-1]^{\even}(2,\cdots,2)$ under the forgetful map.
Let
$\tilde{N}\in\tilde{\X}$ be a generic point of $\tilde{\X}$. Let
$(\X_{1},N_{1})$ be an elliptic curve. We define the genus $g$ curve $\X$ by 
$$\X:=(\tilde{\X}\cup\X_{1})/(\tilde{N}\sim N_{1}).$$

We now construct a limit differential $(\X,\omega)$ in the closure of the connected component 
$\omoduli^{\even}(2,\cdots,2)$. Let  $(\tilde\X,\tilde\omega)$ be a differential in the connected
component $\omoduli[g-1]^{\even}(2,\cdots,2)$. Let $\omega_{1}$  be a meromorphic differential on $\X_{1}$
which has a pole of order $2$ at
$N_{1}$ and a zero of order two.  The differential $\omega$ is given by the differential $\tilde\omega$ on
$\tilde\X$ and the differential $\omega_{1}$ on $\X_{1}$. Since $\tilde{N}$ is a general point, it is not a
zero of $\tilde\omega$. This implies that $\omega$ verifies the compatibility
condition~\eqref{equation:conditionDeCompatibiliteGeneral}. Hence  $\omega$ is a limit differential in
the closure of the connected component $\omoduli^{\even}(2,\cdots,2)$.


The end of the proof is similar to the case of connected strata. We can parametrise a neighbourhood of
$\X$ by $(t_{1},\cdots,t_{3g-3})\in\Delta^{3g-3}$ such that the locus of nodal curves is given by
$t_{3g-3}=0$. Only the deformations of $\X_{g-1}$ which
stay inside the projection of $\omoduli[g-1]^{\even}(2,\cdots,2)$ are allowed. The dimension of such
deformations is $3(g-1)-1$ by the induction hypothesis.  
To conclude, we use
a similar deformation-smoothing argument as in the case of connected strata. We can deform the point of
attachment on $\X_{g-1}$, the elliptic curve~$(\X_{1},N_{1})$ and the node.  Thus we deduce by induction, that
the dimension of
the projection of the component $\omoduli^{\even}(2,\cdots,2)$ is $3g-4$.

\paragraph{The components of the strata $\omoduli(2l_{1},\cdots,2l_{n})$ for $2\leq n\leq g-2$ and
$(2l_{1},2l_{2})\neq
(g-1,g-1)$.}
Observe that these strata  have only two connected components
which are determined by the parity of the associated theta characteristics.

Let $\X$ be the snake curve defined in the case of connected strata. We show that we can choose a limit
differential in two ways, such that one is in the boundary of the odd component and the other in the even
component of $\omoduli(2l_{1},\cdots,2l_{n})$. Choose a limit differential $\omega$ on $\X$, and
denote by $\omega_{1}$ its restriction on
$\X_{1}$. The divisor of the differential $\omega_{1}$ is $\divisor{\omega_{1}}=2l_{1}Z_{1}-2l_{1}N_{1,1}$. So
the associated theta characteristic is $\mathcal{L}_{\omega_{1}}:=\Ox[\X_{1}](l_{1}Z_{1}-l_{1}N_{1,1})$. 

There are two cases to consider: the first one is when $l_{1}=2$ and the second one when $l_{1}\geq3$. If
$l_{1}=2$, the
theta characteristic $\mathcal{L}_{\omega_{1}}$ is odd if $Z_{1}$ is a $2$-torsion of $(\X_{1},N_{1,1})$ and
even
if $Z_{1}$ is a primitive  $4$-torsion of $(\X_{1},N_{1,1})$.
If $l_{i}\geq 3$, the
theta characteristic $\mathcal{L}_{\omega_{1}}$ is even if $Z_{1}$ is a $2$-torsion of $(\X_{1},N_{1,1})$ and
odd if $Z_{1}$ is a primitive $l_{1}$-torsion of $(\X_{1},N_{1,1})$.

The parity of $\omega$ is the sum of the parities of the restrictions $\omega|_{\X_{i}}$ on every
irreducible curve $\X_{i}$ of $\X$. This implies that fixing $\omega$ on the $g-1$ components~$\X_{i}$ for
$i\geq 2$, we can define a differential $\omega$ in the boundary of both components of 
$\omoduli(2l_{1},\cdots,2l_{n})$ by changing the parity of $\omega_{1}$. 
The deformation-smoothing argument of the connected strata now implies the claim.

\paragraph{The non-hyperelliptic components of $\omoduli(g-1,g-1)$.} 
Since we have already dealt with the hyperelliptic case, it remains  the case of the other connected
component if $g$ is even or the two other components if $g$ is odd.

Let $(\X_{g-1},\omega_{g-1},Z_{g-1},N_{g-1})$  be a generic pointed differential in the stratum
$\omoduli[g-1](g-1,g-3)$ and $(\X_{1},\omega_{1},Z_{1},N_{1})$ be an elliptic curve  with a
differential~$\omega_{1}$ such that $\divisor{\omega_{1}}=(g-1)Z_{1}-(g-1)N_{1}$. Then the pointed
differential
$(\X,\omega,Z_{1},Z_{g-1})$ is a limit differential on the boundary of $\omoduli(g-1,g-1)$. 
Let us remark that the curve $\X_{g-1}$ is not hyperelliptic, because the
dimension of the projection of $\omoduli[g-1](g-1,g-3)=2(g-1)$ is strictly larger
than the dimension of the hyperelliptic locus $\hyperell[g-1]$. In particular, the limit differential
$(\X,\omega,Z_{1},Z_{g-1})$ is not in the
boundary of the hyperelliptic component of these strata.
Moreover, if $g-1$ is
even, then this pointed differential is either in the boundary of the even or in the boundary of the  odd
strata according to the parity
of $\omega_{g-1}$.  

The conclusion of this case uses the same deformation-smoothing argument as previously in this proof.

\paragraph{The non-hyperelliptic minimal strata.}
 The zero of a differential $(\X,\omega)$ in the strata $\omoduli(2g-2)$ is Weierstrass point. Since
there exists only finitely many Weierstrass points on a curve, the projection from every component
of $\pomoduli(2g-2)$ to $\moduli$ is finite. It is known that the dimension of $\pomoduli(2g-2)$ is $2g-2$,
so the dimension of its projection has dimension $2g-2$ too.

This concludes the proof of Theorem~\ref{theoreme:DimensionProjectionStrates}.
\end{proof}

As a corollary of Theorem~\ref{theoreme:DimensionProjectionStrates} we obtain the Kodaira dimension of all the
strata $\omoduli(k_{1},\cdots,k_{g-1})$  different from $\pomoduli^{\even}(2,\cdots,2)$ when $\moduli$ is of
general type.

\begin{cor}\label{corollaire:dimKodairaStratesProjFini}
The strata of the form $\pomoduli(k_{1},\cdots,k_{g-1})$ different from $\pomoduli^{\even}(2,\cdots,2)$  are
of general type for $g=22$ and $g\geq 24$.
\end{cor}

\begin{proof}
 It has been proved that $\moduli$ is of general type for $g\geq 24$ by Harris and Mumford and for $g=22$ by
Farkas. According to Theorem~\ref{theoreme:DimensionProjectionStrates} and
Theorem~\ref{theoreme:DimKodairaRevetementFini} we have
$$\kappa(\moduli)\leq\kappa(\pomoduli(k_{1},\cdots,k_{g-1}))\leq\dim \pomoduli(k_{1},\cdots,k_{g-1}) .$$
Since the left and the right term of this inequality are equal to $3g-3$, the inequalities are equalities and
the corollary follows.
\end{proof}

Using the subadditivity of the Kodaira dimension (see Theorem~\ref{theoreme:SousAdditiviteKodaria}), we can
determine the Kodaira dimension of the strata which impose few conditions on the differential.

\begin{theorem}\label{theoreme:DimKodairaStraPeutConditions}
For any  $g\geq2$, let $(k_{1},\cdots,k_{n})$ be a tuple of positive numbers  of the form
$(k_{1},\cdots,k_{l},1,\cdots,1)$ with $k_{i}\geq 2$ for $i\leq l$  such that
$$\sum_{i=1}^{n}k_{i}=2g-2 \text{ and } \sum_{i=1}^{l}k_{i}\leq g-2.$$ 
Then  the Kodaira dimension of the stratum $\pomoduli(k_{1},\cdots,k_{n})$ is $-\infty$.
\end{theorem}

The proof make an essential use of the following space. 
\begin{defn}\label{definition:incidenceAnulation}
 Let $\X$ be a curve  of genus $g$  and  $i=(i_{1},\cdots,i_{l})\in\NN^{l}$ be a $l$-tuple of positive
numbers. The {\em vanishing incidence of order $i$ of $\X$} is 
 $$I_{i}(\X):=\left\{((Q_{1},\cdots,Q_{l}),\omega)\in\X^{l}\times\PP
H^{0}(\X,\holoneform):\ord_{Q_{j}}(\omega)\geq i_{j}
\right\}.$$
\end{defn}

\begin{proof}[Proof of Theorem~\ref{theoreme:DimKodairaStraPeutConditions}.]
 Let $\X$ be a generic curve of genus $g$. We show that the fibre $\pi^{-1}(\X)$ over $\X$  by the
forgetful map $\pi:\pomoduli(k_{1},\cdots,k_{n})\to\moduli$ is connected and has Kodaira dimension $-\infty$.
The theorem follows readily from
this fact combined with Theorem~\ref{theoreme:SousAdditiviteKodaria}.

Recall that by hypothesis the $n$-tuple $(k_{1},\cdots,k_{n})=(k_{1},\cdots,k_{l},1,\cdots,1)$, with
$k_{i}\geq2$ for $i\leq l$.
Let us denote $k:=(k_{1},\cdots,k_{l})$  and let $r:=\sum_{i=1}^{l}k_{i}$ be the sum of these
orders.
We show that the vanishing incidence of order $k$ is an  algebraic fibre space with generic fibre
$\PP^{g-r}$. Indeed, it follows from Riemann-Roch that for any $l$-tuple of points
$(Z_{1},\cdots,Z_{l})\in\X^{l}$,
the vector space  $$H^{0}\left(\X,\holoneform\left(-\sum_{i=1}^{l}(k_{i}Z_{i})\right)\right)$$ is of dimension
at least $g-r$.
Since $\X$ is generic, the space corresponding to differentials having order exactly $k_{i}$ at $Z_{i}$ and
one otherwise is an open subset of this space. This implies the claim that the vanishing incidence variety
of order~$k$ is an algebraic fibre space with generic fibre isomorphic to $\PP^{g-r-1}$.

Now, the second projection of the vanishing incidence variety of order $k$ to~$\PP^{g-1}$ is clearly
surjective on the
closure of $\pi^{-1}(\X)$ inside $\pobarmoduli(k_{1},\cdots,k_{n})$. Moreover, this map does not factorise
through the first projection. This implies that the generic fibre of
$\pi:\pomoduli(k_{1},\cdots,k_{n})\to\moduli$ is uniruled. Therefore, its Kodaira dimension is $-\infty$.
\end{proof}

\paragraph{Some other strata.}

We determine the Kodaira dimension of some other strata. Let us remark that if $\moduli$ is of general type
and $n\geq g$,  it
suffices to determine the Kodaira dimension of a generic fibre of the map from the stratum
$S:=\pomoduli(k_{1},\cdots,k_{n})$ to $\moduli$ in order to compute the Kodaira dimension of~$S$. However,
this seems to be a quite subtle problem in general.

\paragraph{The strata $\pomoduli(g-1,1,\cdots,1)$, when $\moduli$ is of general type.}

According to
Theorem~\ref{theoreme:DimensionProjectionStrates}, the generic fibres of the forgetful map
 $$\pi:\pomoduli(g-1,1,\cdots,1)\to\moduli$$ are curves. Let us determine these curves.
\begin{lemma}
 Let $\X$ be a generic curve of genus $g\geq 3$. If $g\geq 4$, the closure of the fibre at $\X$ by $\pi$ is a
curve isomorphic to $\X$.
If $g=3$, then the closure of the fibre at $\X$ by $\pi$ is a singular curve such that $\X$ is its stable
model. 
\end{lemma}

The proof uses the vanishing incidence of order $g-1$ of $\X$ that we introduce in 
Definition~\ref{definition:incidenceAnulation}. Let us recall that we denote the set of Weierstrass points of
an algebraic curve $\X$ by $\WP(\X)$.

\begin{proof} 
Let $\X$ be a generic curve in $\moduli$. The preimage of $\X$ by the forgetful map
$\pi:\pomoduli(g-1,1,\cdots,1)\to\moduli$ is
 isomorphic to an open subset of the image of the  projection of $I_{g-1}(\X\setminus\WP(\X))$ into
$\PP^{g-1}$. The closure of this locus is isomorphic to
 the projection in $\PP^{g-1}$ of the  closure  of $I_{g-1}(\X\setminus\WP(\X))$.

 Let $\X$ be a generic curve of genus $3$. Then the fibre of the forgetful map
$\pi:\pomoduli[3](2,1,1)\to\moduli[3]$ above $\X$ is isomorphic to an open subset $U$ of $\X$. The closure of
$U$ has $24$ cusps (at the Weierstrass points of $\X$) and
 $28$ nodes (at the double tangents of order $(2,2)$).
 
 Let $\X$ be a generic curve of genus $g\geq 4$. The fibre of the forgetful map
$\pi:\pomoduli(g-1,1,\cdots,1)\to\moduli$ at $\X$ is an open subset $U$ of $\X$. The points of $\X\setminus U$
is the union of the Weierstrass points of $\X$
together with the points~$Q\in\X$ such that there exist $\omega\in H^{0}(\X,K_{\X})$ and $R\in\X$ such that
$$\divisor{\omega}\geq (g-1)Q + 2R.$$ The closure of $U$ in $\PP^{g-1}$ is also a  curve birationally
equivalent
to $\X$.
\end{proof}
 
It follows that the generic fibres of the forgetful map $\pi$ are of
general type. Therefore,   Theorem~\ref{theoreme:KodairaViehweg} implies that $\pomoduli(g-1,1,\cdots,1)$ is
of
general type when $\moduli$ is of general type:

 \begin{prop}\label{proposition:DimKodaira(g-1,1)}
   The strata $\pomoduli(g-1,1,\cdots,1)$ are of general type for $g\geq 24$ or $g=22$.
 \end{prop}

\paragraph{The hyperelliptic strata $\pomoduli(g-1,g-1)$.}

We show that the hyperelliptic components of the strata $\pomoduli(g-1,g-1)$ are uniruled.

\begin{prop}  \label{proposition:dimKodairaHypNeg}
The connected component $\pomoduli^{\hyp}(2d,2d)$ is uniruled for every genus $g\geq 2$. 
\end{prop}

\begin{proof}
The fibre of the morphism $\pomoduli^{\hyp}(2d,2d)\to\hyperell$ is a projective line without $2g+2$ points
(corresponding to the Weierstrass points). So the closure of the generic fibre is a projective line. The
Kodaira dimension of the component $\pomoduli^{\hyp}(2d,2d)$ follows from
Theorem~\ref{theoreme:SousAdditiviteKodaria}.
\end{proof}

\paragraph{The even connected component of $\pomoduli(2,\cdots,2)$.}

\begin{prop} \label{proposition:dimKodairaEvenDeux}
The connected component $\pomoduli^{\even}(2,\cdots,2)$ is uniruled for every genus $g\geq 2$. 
\end{prop}

\begin{proof}
Let $X$ be a generic curve  in the projection of $\pomoduli^{\even}(2,\cdots,2)$ and $\omega$ an even
differential on $\X$. By definition, we have $h^{0}(\X,\frac{1}{2}\divisor{\omega})=2$. In particular, the
fibre of $\pomoduli^{\even}(2,\cdots,2)\to\moduli$ at $\X$ is a projective line. This proves the proposition.
\end{proof}

\paragraph{The odd connected component of $\pomoduli(2,\cdots,2)$.}

We show that the strata $\pomoduli^{\odd}(2,\cdots,2)$ are birationally equivalent to the moduli space of odd
spin structures $\spin^{-}$. This allows
us to deduce the Kodaira dimensions of these strata using the work of Farkas and Verra \cite{MR2984107}.

\begin{prop}\label{proposition:isoOdd}
 There exists a birational morphism 
 \begin{align*}
\varphi:
\pomoduli^{\odd}(2,\cdots ,2)&\to\spin^{-}\\
(\X,\omega)&\mapsto\left(\X,\Ox\left(\frac{1}{2}\divisor{\omega}\right)\right).
\end{align*}
\end{prop}

\begin{proof}
It is clear that the map $\varphi$ is well defined.
To prove the proposition, we construct a birational inverse for $\varphi$.

Let $\X$ be a curve in $\moduli$ such that the preimage of the forgetful map $\pi:\pomoduli^{\odd}(2,\cdots
,2)\to\moduli$ is finite, and has no differential in the connected components
$\omoduli^{\odd}(2l_{1},\cdots,2l_{n})$ for $n\leq g-2$ or in any even component
$\omoduli^{\even}(2l_{1},\cdots,2l_{n})$. Moreover, we suppose that every theta characteristic $\mathcal{L}$
on $\X$ satisfies $h^{0}(\X,\mathcal{L})=1$. According to Theorem~\ref{theoreme:DimensionProjectionStrates},
this set is an open dense set inside $\moduli$. Hence it suffices to give an inverse to $\varphi$ above this
set
of curves. 

Let $(\X,\mathcal{L})$ be an odd theta characteristic on $\X$. It suffices to show
that there exists  a unique $(g-1)$-tuple $(Q_{1},\cdots,Q_{g-1})$ such that 
$$ 2\sum_{i=1}^{g-1} Q_{i}\sim K_{\X} \text{, and } \mathcal{L}\sim\Ox\left(\sum_{i=1}^{g-1} Q_{i}\right).$$ 
The inverse of $\varphi$ would then be
given by
$$\varphi^{-1}(\X,\mathcal{L})=(\X,\omega),$$
where $\omega$ is the differential with divisor $\divisor{\omega}=\sum 2Q_{i}$. Indeed, by hypothesis on $\X$,
the differential $\omega$ is not in a connected component of the form 
$\omoduli^{\even}(2l_{1},\cdots,2l_{n})$ or $\omoduli^{\odd}(2l_{1},\cdots,2l_{n})$ for $n\leq g-2$. Thus the
differential $\omega$ is in the stratum $\omoduli^{\odd}(2,\cdots,2)$.

Let us remark that since by definition $h^{0}(\X,\mathcal{L})\geq 1$, the line bundle $\mathcal{L}$ is
effective.
Moreover, every effective line bundle of degree $g-1$ on $\X$ can be represented by $\Ox\left(\sum
Q_{i}\right)$ for $Q_{i}\in\X$. Since by definition
$\mathcal{L}^{\tensor 2}= \Ox(K_{\X})$ 
the divisor  $2\left(\sum Q_{i}\right)$ is linearly equivalent to $K_{\X}$. 

To conclude the proof, it suffices to show that this $(g-1)$-tuple is unique up to permutation.
Let  $(R_{1},\cdots,R_{g-1})\in\X^{g-1}$ such that  $2\left(\sum R_{i}\right)\sim K_{\X}$ and
$\mathcal{L}=\Ox\left(\sum R_{i}\right)$.
Since $\mathcal{L}=\Ox(\sum Q_{i})=\Ox(\sum R_{i})$, the line bundle $\Ox(\sum Q_{i}-\sum R_{i})$ is the
trivial bundle $\Ox$.
Applying the Theorem of Riemann-Roch to this line bundle gives:
$$h^{0}\left(\sum Q_{i}-\sum R_{i}\right)-h^{0}\left(K_{\X}-\sum Q_{i}+\sum R_{i}\right)=1-g.$$
Now it follows from the linear equivalence $2\left(\sum Q_{i}\right)\sim K_{\X}$ that
$$1-h^{0}\left(\sum Q_{i} + \sum R_{i}\right)=1-g.$$
 More explicitly, we have the equation 
$$h^{0}\left(\sum Q_{i} + \sum R_{i}\right)=g.$$
In particular, $\Ox(\sum Q_{i} + \sum R_{i})$ is a $\grd{g-1}{2g-2}$ on $\X$, so is the canonical line bundle
of $\X$. But since $\Ox\left(K_{\X}-\sum Q_{i}\right)$ is a theta characteristic, the dimension of its space
of section is by hypothesis
$$h^{0}\left(K_{\X}-\sum Q_{i}\right)=1.$$
 In particular, any differential  which vanishes at the  $Q_{i}$ is proportional to $\omega$.  In
particular, the points $R_{1},\cdots,R_{g-1}$ coincide with the points $Q_{1},\cdots,Q_{g-1}$.
\end{proof}

Therefore, we can deduce the Kodaira dimension of these connected components from the work of Farkas and
Verra
(see \cite{MR2984107}).
\begin{cor}\label{corollaire:dimKodairaOddDeux}
The stratum $\pomoduli(2,\cdots,2)^{\odd}$ is uniruled if $g\leq 11$ and is of general type for $g\geq 12$.
\end{cor}


\section{Hyperelliptic Minimal Strata $\PP\obarmoduliinc[g,1]{2g-2}^{\hyp}$.}
\label{section:hyperelliptique}

The main result of this section is Theorem~\ref{theoreme:isoEntreHypEtWeier}, where we relate the
{incidence variety compactification}
$\PP\obarmoduliinc[g,1]{2g-2}^{\hyp}$ of the
hyperelliptic minimal strata with the locus $\overline{\WP(\hyperell)}$ of Weierstrass points of
hyperelliptic
curves. We show that the fibres of the forgetful map
$\pi:\PP\obarmoduliinc[g,1]{2g-2}^{\hyp}\to\overline{\WP(\hyperell)}$ are projective spaces.

 For sake of concreteness, we describe the hyperelliptic curves with one node in
Theorem~\ref{theoreme:locusHyperell} and the closure of the locus of Weierstrass points of hyperelliptic
curves in $\barmoduli[g,1]$ in Theorem~\ref{theoreme:limitDesPtsDeWeiCasHyp}. 
 Moreover, we describe the pointed differential in
the {incidence variety compactification} of the hyperelliptic minimal strata in the most simple cases in
Theorem~\ref{theoreme:bordHypAvec2Courbes} and Theorem~\ref{theoreme:bordHypCasIrr}.

\paragraph{Admissible covers.}
The key tool to study hyperelliptic curves is the theory of admissible covers. Let us quickly recall its
definition and relationship with hyperelliptic curves. For more details see \cite[Section~3.G]{MR1631825}.
\begin{defn}
Let $(B;Q_{1},\cdots,Q_{n})$ be a stable $n$-pointed curve of arithmetic genus zero and $N_{1},\cdots,N_{k}$
the nodes
of the curve $B$. An {\em admissible cover of the curve $B$} is a nodal curve $\X$ and a regular map
$\pi:\X\to B$ such that the following two conditions hold.
\begin{itemize}
 \item[i)] The preimage of the smooth locus of $B$ is the smooth locus of $\X$ and the restriction of the map
$\pi$ to this open set is simply branched over the points~$Q_{i}$ and otherwise unramified.
 \item[ii)] The preimage of the singular locus of $B$ is the singular locus of $\X$ and for every node $N$ of
$B$ and every node $\tilde N$ of $\X$ lying over it, the two branches of $\X$ near $\tilde N$ map to the
branches of $B$ near $N$ with the same ramification index. 
 \end{itemize}
\end{defn}

This notion is particularly adapted to describe the closure of the loci of $k$-gonal curves inside
$\barmoduli$.
\begin{defn}
 Let $\X$ be a stable curve. We say that $\X$ is $k$-gonal if and only if it is a limit of smooth $k$-gonal
curves.
\end{defn}

The following theorem allows us to characterise the $k$-gonal curves (see \cite[Theorem~3.160]{MR1631825}).
\begin{theorem}\label{theoreme:limitDeCourbekGonale}
A stable curve $\X$ is $k$-gonal if and only if 
there exists a $k$-sheeted admissible cover $\X'\to B$ of a stable pointed curve of genus $0$ which  is stably
equivalent to $\X$.
\end{theorem}

In particular, since the smooth hyperelliptic curves are exactly the smooth $2$-gonal curves, the stable
hyperelliptic curves will by given by the $2$-sheeted admissible covers. 

\paragraph{The hyperelliptic locus $\barhyperell$ in $\barmoduli$.}
A hyperelliptic curve with one node is described in the
following theorem.

\begin{theorem}\label{theoreme:locusHyperell}
  Let $\X\in\barhyperell$ be a hyperelliptic curve of genus $g$ with one node.
  \begin{itemize}
   \item If $\X$ is irreducible,  the normalisation $\tilde\X$ of $\X$ is hyperelliptic and
   the preimage of the node is  a pair of points conjugated by the hyperelliptic involution. 
\item If $\X$ is of compact type,  the curve $\X$ is given by $\X_{1}\cup\X_{2}/(N_{1}\sim N_{2})$, where
the $\X_{j}$
are
hyperelliptic and $N_{j}$ are Weierstrass points of $\X_{j}$ respectively.
  \end{itemize}
\end{theorem}

Let us recall that the {\em Weierstrass locus} inside $\moduli[g,1]$ is defined by $$\WP(\moduli):=\left\{
(\X,W)|\  W \text{ is a Weierstrass point of $\X$} \right\}.$$
The {\em hyperelliptic Weierstrass locus} is simply the restriction of this locus above the hyperelliptic
locus of $\moduli$:
\begin{equation*}
 \WP(\hyperell):=\left\{ (\X,W)\in\WP(\moduli)|\ \X \text{ is hyperelliptic} \right\}.
\end{equation*}
We describe now  the marked curves  in the closure of $\WP(\hyperell)$ which are generic in $\delta_{i}$.

\begin{theorem}\label{theoreme:limitDesPtsDeWeiCasHyp}
Let $(\X,W)\in\overline{\WP(\hyperell)}\subset\barmoduli[g,1]$ be a marked curve in the closure of the
hyperelliptic Weierstrass locus, such that $\X$ is stably equivalent to a generic curve in $\delta_{i}$.

The pair $(\X,W)$ is of one of the following form.
\begin{itemize}
 \item The curve $\X$ is stably equivalent to a curve in $\delta_{0}$. Then $X$ is either irreducible and
$W$
is in
$\WP(\X)$, or $\X$ is the blow-up at the node of an irreducible curve and
$W$ is in the exceptional component. 
\item The curve $\X$ is generic in the divisor $\delta_{1}$ and the point $W$ is one of the  $2g-1$
smooth Weierstrass points of the curve of genus $g-1$ (or
a $2$-torsion point if $g=2$) or a $2$-torsion point of the elliptic curve. 
\item The curve $\X$ is generic in the divisor $\delta_{i}$ for $i\geq 2$ and the points $W$ are smooth
Weierstrass points of the irreducible components of $\X$.
\end{itemize}

\end{theorem} 

These two theorems are consequences of the theory of admissible covers and in particular, we will use
Theorem~\ref{theoreme:limitDeCourbekGonale} in a crucial way.

\begin{proof}[Proof of Theorem~\ref{theoreme:locusHyperell} and
Theorem~\ref{theoreme:limitDesPtsDeWeiCasHyp}.] Let us first suppose that  $1\leq i\leq
\left[\frac{g}{2}\right]$ and let $\X$ be a hyperelliptic curve in  $\delta_{i}$ as given in the theorem. By
Theorem~\ref{theoreme:limitDeCourbekGonale}, the curve $\X$ is stably equivalent to
 an admissible cover $\pi:\X'\to B$ of degree two above a stable marked curve of
genus zero $(B;x_{1},\cdots,x_{2g+2})$. Let $B_{0}$ be  an irreducible
 component of $B$ which meets only one other component and denote $\X_{0}:=\pi^{-1}(B_{0})$.
Remark that there exists such a $B_{0}$  since  $B$ is of compact type. 
Since $(B;x_{1},\cdots,x_{2g+2})$ is a stable marked curve, at least two marked points lie on $B_{0}$. 
Moreover the cardinality of the preimage
 of the node is one because otherwise $\X$ would have a nonseparating node.
 Let us call this point~$N_{0}$. It is a ramification point of the map to $B_{0}$, so
 by Riemann-Hurwitz the curve~$\X_{0}$ has genus at least $1$. And since $\X$ is generic in $\delta_{i}$, the
component~$\X_{0}$ has genus $i$ or $g-i$. We will suppose that $\X_{0}$ has genus $i$. Then the curve~$B_{0}$
has $2i+1$ marked points and the preimages of these points together with $N_{0}$ are the
Weierstrass points of $\X_{0}$. Now there is at least  one other extremal component and the same argument show
that it has genus $g-i$. This concludes the proof of both theorems in the case $1\leq i\leq
\left[\frac{g}{2}\right]$.

The case  $i=0$ is similar. Let $\pi:\X'\to B$ be an admissible cover of degree two stably equivalent to $\X$.
This time, for every irreducible component $B_{0}$ of $B$ which meets one other component of $B$, the
preimage of the node contains two distinct points.
As in the previous case, the curve $B$ has only two components: one of them contains two
marked points
and
the other the $2g$ reminding ones. The curve $\X$ is obtained from $\X'$ by blowing down the preimage of the
projective line which contains only two marked points. The restriction of the projection to this second
component implies that the
two
preimages of the node are conjugated by the hyperelliptic involution.

Since the Weierstrass points are the ramification points of the map to $\PP^{1}$,
their limits are the ramification points of the smooth locus of the admissible cover.
\end{proof}

Let us conclude this paragraph by describing the ramification locus of the forgetful map
$\pi:\WP(\barhyperell)\to\barhyperell$ from the hyperelliptic Weierstrass locus to the hyperelliptic locus.
This is a direct application of Theorem~\ref{theoreme:limitDesPtsDeWeiCasHyp}.
\begin{cor}\label{corollaire:RamWeiLocusHyperell}
The map $\pi:\WP(\barhyperell)\to\barhyperell$ is unramified above the generic locus of the divisors
$\delta_{i}$ for $i\geq1$. On the other hand, above an irreducible curve~$\X$ with $k$
nodes there are $2g-2-2k$ unramified points  and $k$ ramification points of order two.
\end{cor}

\paragraph{The relationship between the hyperelliptic Weierstrass locus and the hyperelliptic minimal
strata.}

We now describe the {incidence variety compactification}
of the hyperelliptic minimal
strata. We will describe precisely its relationship with the hyperelliptic Weierstrass locus. Before, let
us recall that
two irreducible components $\X_{1}$ and $\X_{2}$
of $\X$ are {\em polarly related} by a differential~$\omega$ if $\X_{1}=\X_{2}$ or~$\omega$ has simple
poles at the nodes between $\X_{1}$ and $\X_{2}$ (see
Definition~\ref{definition:composentesPolairesConnectees}).

\begin{theorem}\label{theoreme:isoEntreHypEtWeier}
Let $(\X,Z)\in\overline{\WP(\hyperell)}\subset\barmoduli[g,1]$ be a pair consisting of a hyperelliptic curve
$\X$ together with a Weierstrass point $Z$. 

Then there exists a stable differential $\omega$ on $\X$, such that for every pointed stable differential
$(\X,\omega',Z)$  in $\PP\obarmoduliinc[g,1]{2g-2}^{\hyp}$ we have the following two properties.
\begin{itemize}
 \item  If $\omega\equiv 0$ on an irreducible component $\X_{i}$, then $\omega'\equiv 0$ on $\X_{i}$.
 \item  There exists
$(\alpha_{1},\cdots,\alpha_{r})\in\PP^{r-1}$ such that 
$$\omega|_{\tilde{\X_{i}}}=\alpha_{i}\omega'|_{\tilde{\X_{i}}},$$
 where $\left\{ \tilde{\X_{i}} \right\}_{i=1,\cdots,r}$ is the set of polarly related
components of the differential $(\X,\omega)$ such that $\omega|_{\tilde\X_{i}}\not\equiv 0$.
\end{itemize}

In particular, the fibres of the forgetful map 
\begin{align*}
\pi:\PP\obarmoduliinc[g,1]{2g-2}^{\hyp} &\to \overline{\WP(\hyperell)}\\
(\X,\omega,Z)&\mapsto\left(\X,Z \right).
\end{align*}
are isomorphic to $\PP^{r-1}$.
\end{theorem}

The proof is similar to the one of corollary~\ref{corollaire:uniciteLimDiffSurTypeCompacte}, where we show a
related result for curves of compact type. In fact, since hyperelliptic curves are covers of degree two above
a curve of compact type, many ideas will work in this case.

\begin{proof}
Let $(\X,Z)$ be a hyperelliptic curve together with a Weierstrass point of~$\X$. There exists a family
$(\famcurv,\seczero)$ of hyperelliptic curves with a Weierstrass section which converges to $(\X,\omega)$.
Let
$\famomega$ be a family of differentials on $\famcurv$ such that $\famomega(t)$ has a zero of order $2g-2$ at
$\seczero(t)$. It turns out that the limit differential of this family only depends on $(\X,Z)$ as we show in
the following.

 According to
Theorem~\ref{theoreme:limitDeCourbekGonale}, there exists a semi stable curve $\bar\X$ stably equivalent to
$\X$ such that $\pi:\bar\X\to B$ is an admissible cover of degree two. Moreover, the point $Z$ is a
ramification point of the map $\pi$.
We will now define a differential on $\bar\X$ unique up to scaling on the components of $\bar\X$ such that by
contracting the exceptional components we can associate a limit differential  on
$\X$.

Since $B$ is of compact type, the set of irreducible components of $B$ which meet one other component is not
empty. Let us denote this set of irreducible components by $\Irr_{1}(B)$. The  irreducible
components of $\bar\X$ which map to $\Irr_{1}(B)$ are denoted by $\Irr_{1}(\bar\X)$. By definition, the
irreducible components in $\Irr_{1}(\bar\X)$ have at most two nodes. If a component has one node, then it is a
Weierstrass point of this component.
Otherwise, the two nodes are conjugated by the hyperelliptic involution.

 Let $\X_{1}$ be an irreducible component of genus $g_{1}$ in $\Irr_{1}(\X)$. If $\X_{1}$ is an exceptional
component, then we associate the differential with two simple poles at the nodes and which is holomorphic
outside of the nodes. If $\X_{1}$ is not an exceptional component, there is a unique way (up to scaling) to
associate a
differential which can be the
restriction of a limit differential according to these four cases. 
\begin{itemize}
 \item[i)] If $\X_{1}$  contains the point $Z$ and has a unique node. Then the differential on~$\X_{1}$ is
the differential with a zero of order $2g-2$ at $Z$ and a pole of order $2(g-g_{1})$ at the node.

 \item[ii)] If $\X_{1}$  contains the point $Z$ and has two nodes. Then the differential on~$\X_{1}$ is the
differential with a zero of order $2g-2$ at $Z$ and two
poles of order $(g-g_{1})$ at both nodes.

 \item[iii)] If $\X_{1}$ does not contain the point $Z$ and has a unique node. Then the differential on
$\X_{1}$ is the differential with a zero of order $2g_{1}-2$ at the node.

 \item[iv)] If $\X_{1}$ does not contain the point $Z$ and has two nodes. Then the differential on $\X_{1}$ is
the
differential with two zeros of order $g_{1}-1$ at both nodes.
\end{itemize}
Indeed,  the only zeros and poles of the differentials are contained in the marked locus. Moreover, the fact
that the differential is anti-invariant under the hyperelliptic involution implies that the orders of the
differentials
have to coincide at a pair of points conjugated by the hyperelliptic involution.

Now we can continue this process in the following way. We remove to the dual graph $\Gamma_{B}$ of $B$ the
vertices corresponding to $\Irr_{1}(B)$ and the edges pointing to them. This new graph is denoted by
$\Gamma_{B}^{1}$. Either $\Gamma_{B}^{1}$ is empty and we have achieved the construction of the differential.
Or
 $\Gamma_{B}^{1}$ is a non empty
tree. In this case the set of irreducible components $\Irr_{2}(B)$ of $B$ corresponding to the leafs
of~$\Gamma_{B}^{1}$ is not empty. The irreducible components of $\bar\X$ mapping to the components of 
$\Irr_{2}(B)$ are denoted by $\Irr_{2}(\bar\X)$.

The
description
of the differential on these components is similar to the previous one. To be more precise, because of the
compatibility condition \eqref{equation:conditionDeCompatibiliteGeneral}, the sum of the degrees of the
differentials at the nodes
with the components
of $\Irr_{1}(\bar\X)$ is $-2$. The only other zeros or poles allowed on an irreducible component are at the
marked
points and the orders have to be invariant by the hyperelliptic involution. 

We continue this process and eventually obtain a differential on the curve $\bar\X$. 
Then we can associate a differential $\tilde\omega$ on $(\X,Z)$ by contracting the exceptional components of
$(\bar\X,Z)$. 

Let us remark that at every pair of  points conjugated by the hyperelliptic involution, the
residues of $\tilde\omega$ at these points are opposite. This has two consequences. The first one is that
nodes
corresponding to loops on the dual graph of $\X$ satisfy the residue condition.
 The second consequence is that we can multiply the restrictions on the irreducible components of the form
$\tilde\omega$ by
constants in such a way that the residue condition is satisfied at every node. 

 Hence we obtain a unique differential up to
multiplicative constants on
each polarly related component of $(\X,\tilde\omega,Z)$.

To conclude,  we obtained a stable differential   $\omega$  by imposing $$\omega|_{\tilde\X_{i}}=0$$ when
$\tilde\omega|_{\tilde\X_{i}}$  has a meromorphic node of degree greater or equal to $2$ in the polarly
component $\tilde\X_{i}$ of $(\X,\tilde\omega)$, and otherwise $$\omega|_{\X_{i}}=\tilde\omega|_{\X_{i}}.$$
By an argument similar to the one in Proposition~\ref{proposition:relationPlumStable}, we can deduce that
there exists a family in $\omoduliinc[g,1]{2g-2}^{\hyp}$ which has $(\X,\omega,Z)$ as stable limit.
Moreover, every other stable differential on $(\X,Z)$ in the closure of the
connected component $\omoduliinc[g,1]{2g-2}^{\hyp}$
differs only by multiplicative constants on the polarly related components of $(\X,\omega)$.
\end{proof}

For sake of concreteness, let us describe explicitly the stable differentials inside
$\obarmoduliinc[g,1]{2g-2}                                                           ol^{\hyp}$ when the curve
has at most two irreducible components. First we look at
differentials such that the underlying curve is
in $\delta_{i}$ for~$i\geq 1$. 

\begin{theorem}\label{theoreme:bordHypAvec2Courbes}
Let $(\X,\omega,Z)$ be a stable differential in  $\PP\obarmoduliinc[g,1]{2g-2}^{\hyp}$ such that
$\X:=\X_{1}\cup\X_{2}/(N_{1}\sim N_{2})$ is in the divisor $\delta_{i}$. We suppose without
lose of generality that $Z\in \X_{1}$.

Then $(\X,\omega,Z)$ is characterised by the following three properties.
\begin{itemize}
\item[i)] The curves $\X_{j}$ are hyperelliptic and the points $N_{1}$ and $N_{2}$ are Weierstrass points of
$\X_{1}$ and $\X_{2}$ respectively.
\item[ii)] The point $Z$ is a Weierstrass point of $\X_{1}$. 
\item[iii)] The differential $\omega$ is identically zero on the component  of $\X$ that contains~$Z$ 
and  is the holomorphic differential with a zero of order $2g_{2}-2$ at~$N_{2}$ on the component $\X_{2}$.
 \end{itemize}
\end{theorem}

Now we look at differentials such that the underlying curve is stably equivalent to a curve in $\delta_{0}$.

\begin{theorem}\label{theoreme:bordHypCasIrr}
Let $\X$ be either an irreducible curve or an irreducible curve blown up at a node.

Then $(\X,\omega,Z)$ is in the {incidence variety compactification} of the connected component
$\pomoduli[g,1]^{\hyp}(2g-2)$ if
and only if it is of one of the following two forms. 
\begin{itemize}
\item The point $Z$ is in the smooth locus of the irreducible curve $\X$ and the
differential $\omega$ is a section of $\dualsheave$ which vanishes  at $Z$ with order $2g+2$. 

\item The point  $Z$ is in the exceptional divisor coming from the blow-up of a
node  $N_{1}\sim N_{2}$, and the differential $\omega$ is the stable differential with a zero of
order $g-2$ at both $N_{1}$ and $N_{2}$.
\end{itemize}
\end{theorem}

We omit the proofs of both theorems. They are relatively similar to the proof of
Theorem~\ref{theoreme:isoEntreHypEtWeier}, and the reader can look at the proofs of the main theorems of
Section~\ref{section:omoduli3,4,odd} for similar computations.


\section{The Boundary of $\PP\obarmoduliinc[3,1]{4}^{\odd}$.}\label{section:omoduli3,4,odd}

In this section, we give a precise description of the geometry of the pointed differentials which lie in the
boundary of the {incidence variety compactification} of $\pomoduli[3,1]^{\odd}(4)$. Since this description
depends in an essential way on the dual graph of the underlying curve, we will  restrict ourself to the
most simple cases.
We define a {\em generic curve in the divisor $\delta_{i}$} to be a curve 
 in the divisor $\delta_{i}$ with a single node. 

For a generic curve in $\delta_{1}$, the description of the limit differentials in the boundary of
$\pomoduli[3,1]^{\odd}(4)$ is given in Theorem~\ref{theoreme:linEqOfLimit} and the stable differentials in
Corollary~\ref{corollaire:stabDiffgTroisRed}. This description implies (see
Corollary~\ref{corollaire:compactGenreTroisDeltaZero}) that the
 {incidence variety compactification}  of the connected component $\pomoduli[3,1]^{\odd}(4)$ is better than
the Deligne-Mumford
compactification $\obarmoduli[3](4)$.

For a curve stably equivalent to a generic curve in $\delta_{0}$, the description of the limit differentials
in the boundary of
$\pomoduli[3,1]^{\odd}(4)$ is given in Theorem~\ref{theoreme:bordZeroOrdre4Irr} and
Theorem~\ref{theoreme:bordZeroOrdre4Irrbis} and the stable differentials in Theorem~\ref{theoreme:bordZeroOrdre4Irr} and in
Corollary~\ref{corollaire:stabDiffgTroisIrr}. In the first theorem we investigate the case where the
underlying curve is stable, and in the second only semi stable. 

To conclude, we give two examples of families in $\obarmoduliinc[3,1]{4}$. In the first example, the
underlying
curve is given by a quartic in the projective plane. In the second, we deform the polygonal representation of
differentials belonging to~$\obarmoduli[3](4)$.

\subsection{The underlying curve  is generic in $\delta_{1}$.}

In order to describe the limit differential of $\pomoduli[3,1]^{\odd}(4)$, let us introduce the following
definition.
\begin{defn}
 Let $(\X,Q)$ be an elliptic curve, $k\geq 2$ be an integer and $l$ be a non-trivial divisor of $k$. The
points
of $\X\setminus Q$ which are $k$-torsion but not $l$-torsion of $(\X,Q)$ are {\em primitive $k$-torsion} of
$(\X,Q)$.
\end{defn}

Moreover, let us give the definition of what we mean with a {\em generic curve in a divisor}. 
\begin{defn} Let $\delta_{i}$ be a divisor of $\barmoduli$.
 A {\em generic curve in  $\delta_{i}$} is a curve in the divisor $\delta_{i}$ with a single
node. 
\end{defn}

In this section $\X$ will denote a generic curve in $\delta_{1}$ and will be given by the union of a curve
$\X_{1}$ of genus one and a curve $\X_{2}$ of genus two
meeting together at $N_{1}\in\X_{1}$ and $N_{2}\in\X_{2}$.

We now give a precise description of the limit differentials in the boundary of the connected component
$\omoduliinc[3,1]{4}^{\odd}$ such
that the projection to $\moduli[3]$ is a generic curve of the divisor~$\delta_{1}$.

\begin{theorem}\label{theoreme:linEqOfLimit}
Let $(\X,\omega,Z)$ be a limit differential at the boundary of the odd component of the stratum
$\omoduli[3,1](4)$.

If the curve $\X$
is  stably-equivalent to a generic curve
 in the divisor $\delta_{1}$, then the curve $\X$ is a generic curve in $\delta_{1}$ and $(\X,\omega,Z)$ is of
one of the following two forms.
\begin{itemize}
 \item The point $Z$ is a primitive $4$-torsion point of $(\X_{1},N_{1})$ and the point of attachment
$N_{2}\in\X_{2}$ is a Weierstrass point of $\X_{2}$. 

The restriction of $\omega$ to $\X_{1}$ is the meromorphic
differential with a zero of order $4$ at $Z$ and a pole of order $4$ at $N_{1}$. The restriction of $\omega$
to $\X_{2}$ is the Abelian differential with a zero of order $2$ at $N_{2}$.

\item
The point $Z$ is not a Weierstrass
point of
$\X_{2}$ and the pair $(Z,N_{2})$ satisfies the relation
\begin{equation}\label{equation:LimitDesZerosGenre3}
 4Z-2N_{2}\sim K_{\X_{2}}.
\end{equation}

The restriction of $\omega$ to $\X_{1}$ is an Abelian differential. The restriction of $\omega$
to $\X_{2}$ is the meromorphic differential with a zero of order $4$ at $Z$ and a pole of order $2$ at
$N_{2}$.
\end{itemize}
\end{theorem}

The main tools of the proof consist of the theory of limits differentials and the spin structure on stable
curves.

\begin{proof}
Since $\X$ is stably-equivalent to a generic curve in $\delta_{1}$, the marked curve $(\X,Z)$ must be of one
of the following three forms, where the genus of $\X_{i}$ is $i$.

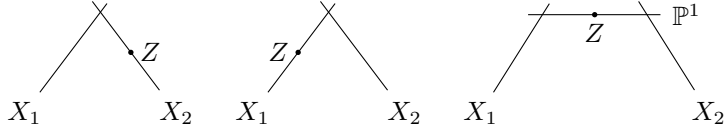
\begin{figure}[ht]
\centering
 \begin{tikzpicture}[>=stealth',shorten >=1pt,auto,node distance=2.8cm]
     
    \node(E1) at (-6,.2){$\X_{1}$};
     \node(E2) at (-5.2,1.8){};

     \node(F1) at (-4.8,1.8){};
     \node(F2) at (-4,.2){$\X_{2}$};

 \draw(E1) edge node[left]{} (F1);
 \draw(E2) edge node[right]{} (F2);

 \node[right](Z2) at (-4.6,1){$Z$};  \fill (-4.6,1) circle (1pt);

     \node(C1) at (-3,.2){$\X_{1}$};
     \node(C2) at (-2.2,1.8){};

     \node(D1) at (-1.8,1.8){};
     \node(D2) at (-1,.2){$\X_{2}$};

 \draw(C1) edge node[left]{} (D1);
 \draw(C2) edge node[right]{} (D2);

 \node[left](Z1) at (-2.4,1){$Z$};  \fill (-2.4,1) circle (1pt);

     \node(A1) at (0,.2){$\X_{1}$};
     \node(A2) at (.5,1.5){};
     \node(A3) at (2,1.8){};

     \node(B1) at (1,1.8){};
      \node[right](B2) at (2.4,1.5){$\PP^{1}$};
     \node(B3) at (3,.2){$\X_{2}$};

 \draw(A1) edge node[left]{} (B1);
 \draw(A2) edge  (B2);
 \draw(A3) edge node[right]{} (B3);

     \node[below](Z) at (1.5,1.5){$Z$};  \fill (1.5,1.5) circle (1pt);
\end{tikzpicture}
\end{figure}

The third case does not occur according to 
 Corollary~\ref{corollaire:zeroJamaisSurLePontFaible}.
 
Let us remark that since $\omega|_{\X_{i}}$ has at most one pole, this pole cannot have a residue. Therefore,
the
limit differentials on the curve $\X$ are characterised in
Theorem~\ref{theoreme:PlomberieCylindriqueSansResidu}. In the case at hand, observe that the only relevant
condition of Theorem~\ref{theoreme:PlomberieCylindriqueSansResidu} is the Compatibility
Condition~\eqref{equation:conditionDeCompatibiliteGeneral}
$$\ord_{N_{1}}(\omega)+ \ord_{N_{2}}(\omega) =-2,$$
 at the node of $\X$.

Let us now treat the case where $Z\in\X_{1}$. Since $Z$ is a limit differential in the boundary of
$\omoduliinc[3,1]{4}$ the restriction of $\omega$ to $\X_{1}$ has a zero of order  $4$ at $Z$ and a pole
of the same order at $N_{1}$. It follows from the Compatibility
Condition~\eqref{equation:conditionDeCompatibiliteGeneral}  that
the order of $\omega|_{\X_{2}}$ at $N_{2}$ is
$2$. Thus $N_{2}$ is a Weierstrass point of $\X_{2}$.

It remains to show that $Z$ is a primitive $4$-torsion point of $(\X_{1},N_{1})$. By the continuity of the
parity of the spin structure (see Theorem~\ref{theoreme:BordDesStratesDansSpin}) the parity of the spin
structure associated to $\omega$ has to be odd. But since the parity of~$\omega|_{\X_{2}}$ is odd, the parity
of $\omega|_{\X_{1}}$ has to be even. This concludes the first case by
observing that for a $4$-torsion $Z$, we have $h^{0}(\X_{1},\Ox[\X_{1}](2Z-2N_{1}))=0$ if $Z$ is primitive
and $h^{0}(\X_{1},\Ox[\X_{1}](2Z-2N_{1}))=1$ otherwise.

The case where $Z\in\X_{2}$ is very similar, hence we do not write every detail. Since $\omega$ has
a zero of order $4$ at $Z$, it has to have a pole of order
$2$ at $N_{2}$. Therefore the points $Z$ and $N_{2}$  satisfy Equation~\eqref{equation:LimitDesZerosGenre3}.

Let us now show that the
point $Z$ cannot be a Weierstrass point. First let us remark that in this case, the point $N_{2}$ would be a
Weierstrass point too. Indeed Equation~\eqref{equation:LimitDesZerosGenre3} would be equivalent to 
$$2Z\sim 2N_{2} \sim K_{\X_{2}},$$
which clearly implies that $N_{2}$ is a Weierstrass point.
Now the claim follows again from the continuity of the spin structure.
Since in this case the restriction of $\omega$ to $\X_{1}$ is odd, the restriction $\omega|_{\X_{2}}$ has to
be even. Since the associated theta characteristic is $\Ox[\X_{2}](2Z-N_{2})$, it would have exactly one
section if $Z$ (and therefore $N_{2}$) were a Weierstrass point, contradicting
Theorem~\ref{theoreme:BordDesStratesDansSpin}.
\end{proof}

\begin{rem}
An interesting fact is that there are only a finite number of points in $\X_{1}$ which are in the closure of
the zero of
order $4$ of $\omoduliinc[3,1]{4}$. This has to be compared with \cite[Theorem~5.45]{MR1631825}
 which tells us that  when $N_{2}$ is a Weierstrass
point, then every point of $\X_{1}$ is in the closure
of the Weierstrass locus.
\end{rem}

We can characterise the pointed differentials in this case from  Theorem~\ref{theoreme:linEqOfLimit}
and
Proposition~\ref{proposition:relationPlumStable}.
\begin{cor}\label{corollaire:stabDiffgTroisRed}
Let $(\X,\omega,Z)$ be a stable pointed differential in $\obarmoduliinc[3,1]{4}^{\odd}$.

If the curve $\X$
is  stably-equivalent to a generic curve
 in the divisor $\delta_{1}$, then~$\X$ is a stable curve in $\delta_{1}$ and $(\X,\omega,Z)$ is of one of the
following two forms.
\begin{itemize}
 \item 
The point $Z$ is a primitive $4$-torsion point of $(\X_{1},N_{1})$ and $N_{2}$ is a Weierstrass point of
$\X_{2}$.

The restriction of $\omega$ to $\X_{1}$ vanishes identically. The restriction of $\omega$
to~$\X_{2}$ is the Abelian differential with a zero of order $2$ at $N_{2}$.

\item 
The point $Z$ is not a Weierstrass
point of
$\X_{2}$ and the pair $(Z,N_{2})$ satisfies  the relation
$ 4Z-2N_{2}\sim K_{\X_{2}}$.

The restriction of $\omega$ to $\X_{1}$ is an holomorphic differential. The restriction of $\omega$
to $\X_{2}$ vanishes identically.
\end{itemize}
\end{cor}

These properties illustrate  that the {incidence variety compactification} of the connected component
$\omoduli[3]^{\odd}(4)$ is
better than its
Deligne-Mumford compactification.

\begin{cor}\label{corollaire:compactGenreTroisDeltaZero}
Let $\X$ be a generic curve in $\delta_{1}$ such that the nodal point of the curve of genus two is a
Weierstrass point. Let $(\X,\omega)$ be a differential in $\obarmoduli[3](4)$ where $\omega$ is of one of the
following two kinds. 
\begin{itemize}
\item[i)] The restriction of $\omega$ is identically zero on $\X_{1}$ and is a holomorphic
differential with a zero of order two at $N_{2}$ on $\X_{2}$.
\item[ii)] The restriction of $\omega$ is identically zero on $\X_{2}$ and is
holomorphic on $\X_{1}$. 
\end{itemize}
Then the stable differential $(\X,\omega)$ lies in the boundary of
both connected components of the minimal strata in $\obarmoduli[3]$.
However, the closure of the two connected components of $\obarmoduliinc[3,1]{4}$ are disjoint over the generic
locus of
$\delta_{1}$.
\end{cor} 

This corollary follows readily from Theorem~\ref{theoreme:linEqOfLimit} and the description of the boundary of
the closure of the hyperelliptic minimal strata as given in Theorem~\ref{theoreme:isoEntreHypEtWeier}.

\subsection{The underlying curve  is generic in $\delta_{0}$.}

In this section we denote a generic curve in $\delta_{0}$ by $\tilde\X/(N_{1}\sim N_{2})$, where $\tilde\X$
is a smooth curve of genus and $N_{1}$, $N_{2}$ are distinct points of $\tilde\X$.

The following two theorems give the description of the limit differentials which lie in the {incidence
variety compactification} $\PP\obarmoduliinc[3,1]{4}^{\odd}$ such
that the underlying curve is generic in $\delta_{0}$.

 First we give the case
where the zero of the differential lies in the smooth part. Observe that in this case the limit differentials
in the closure of $\omoduliinc[3,1]{4}^{\odd}$ coincide with the stable differentials in 
$\obarmoduliinc[3,1]{4}^{\odd}$.
In the following theorem, we denote by $\X$ the
curve $\tilde\X/(N_{1}\sim N_{2})$.

\begin{theorem}\label{theoreme:bordZeroOrdre4Irr}
Let $Z$ be anon Weierstrass point of $\tilde\X$. There exists a unique pair of distinct points
$(N_{1},N_{2})\in\tilde{\X}^{2}$ and a unique differential $\omega\in
H^{0}(\X,\dualsheave)$ with a zero of order $4$ at $Z$ and a simple pole at $N_{1}$ and $N_{2}$ such that the
triple $(\X,\omega,Z)$ is in $\PP\obarmoduliinc[3,1]{4}^{\odd}$.

The set of triples $$C:=\left\{(N_{1},N_{2},Z): (\X,Z)\in\pi\left(\obarmoduliinc[3,1]{4}^{\odd}\right)\right\}$$
 is a curve in $\tilde\X^{3}$.
Moreover,
for a given pair among the three points $N_{1}$, $N_{2}$ and~$Z$ from the curve $C$, there
exists exactly one point of
$\tilde\X$ such that the triple lies in $C$.
\end{theorem}

Now we describe the case where the zero of the differential lies on a bridge joining the two points of the
node.

\begin{theorem}\label{theoreme:bordZeroOrdre4Irrbis}
 Let $(\X,\omega,Z)$ be a limit differential at the boundary of the stratum $\omoduli[3,1]^{\odd}(4)$ such that $\X$ is the
union of a smooth curve $\tilde\X$ of genus two  and a projective line $\PP^{1}$ which meet at two distinct
points
$N_{1}$ and $N_{2}$. 

Then the point $Z$ is in the projective line $\PP^{1}$,  and $(\X,\omega,Z)$ is of one of the following two
forms.                                                   
\begin{itemize}
\item
The restriction of $\omega$ on $\PP^{1}$
has a zero of order~$4$ at $Z$, a pole of order~$4$ at $N_{1}$ and a pole of order $2$ at $N_{2}$. The
restriction of $\omega$ to $\tilde\X$ is an
holomorphic differential with a zero of order two at $N_{1}$. In particular, $N_{1}$ is a Weierstrass point
of $\tilde\X$.

\item  The restriction of $\omega$ on
$\PP^{1}$ has a zero of order $4$ at $Z$ and two  poles of order $3$ at $N_{1}$ and $N_{2}$. The restriction
of $\omega$ to $\tilde\X$ is a
holomorphic differential with two simple zeros at $N_{1}$ and $N_{2}$. In particular, $N_{1}$ and $N_{2}$ are
conjugated by the hyperelliptic involution of $\tilde\X$.
\end{itemize}

\end{theorem}

We can easily deduce the form of the pointed differentials in this case from
Theorem~\ref{theoreme:bordZeroOrdre4Irrbis} and
Proposition~\ref{proposition:relationPlumStable}.
\begin{cor}\label{corollaire:stabDiffgTroisIrr}
 Let $(\X,\omega,Z)$ be a stable differential in $\obarmoduliinc[3,1]{4}^{\odd}$ such that the curve $\X$ is
the
union of a smooth curve $\tilde\X$ of genus two  and a projective line~$\PP^{1}$ which meet at two distinct
points
$N_{1}$ and $N_{2}$. 

Then the point $Z$ is in the projective line $\PP^{1}$.  The restriction of $\omega$ on $\PP^{1}$ vanishes
everywhere.
Either $N_{1}$ is a Weierstrass point of $\tilde\X$ and the restriction of $\omega$ to $\tilde\X$ is an
holomorphic differential with a zero of order two at $N_{1}$
or the points $N_{1}$ and $N_{2}$ are conjugated by the hyperelliptic involution of $\tilde\X$ and the
restriction of $\omega$ to $\tilde\X$ is a
holomorphic differential with two simple zeros at $N_{1}$ and $N_{2}$. 
\end{cor}

The proofs of Theorem~\ref{theoreme:bordZeroOrdre4Irr} and
Theorem~\ref{theoreme:bordZeroOrdre4Irrbis} are relatively similar. In particular, the main steps will be the
following. The first one is to determine all the possible candidates as triples at the boundary. Then we show
that we can smooth them using the plumbing
cylinder construction of Section~\ref{section:PlomberieCylindrique}. The last step consists of
determining the
cases such that the
smoothing occurs in the odd  component and the ones where the smoothing occurs in the hyperelliptic one.

\begin{proof}[Proof of Theorem~\ref{theoreme:bordZeroOrdre4Irr}]
Let $(\X,Z)$ be an irreducible marked curve of genus two. Then the pointed differentials $(\X,\omega,Z)$ which could
appear in the boundary of the stratum $\PP\obarmoduliinc[3,1]{4}$ are stable differentials $\omega$ with a
zero
of order $4$
at $Z$ and  poles at the nodes of $\X$.

Let us now suppose that $Z$ is not a Weierstrass point of $\tilde\X$.
We want to show that there exists a pair $(N_{1},N_{2})$ on
$\tilde\X$ such that $h^{0}(K_{\tilde\X}+N_{1}+N_{2}-4Z)=1$ and moreover that this pair is unique. 
Since $Z$ is not a Weierstrass point of~$\tilde\X$, the divisor $4Z-K_{\tilde\X}$ is not  canonical. Indeed,
this
would be equivalent to the fact that
$2(Z-\iota Z)$ is principal, where $\iota$ is the hyperelliptic involution. But this would give the existence
of a
function with a pole of order two at $Z$,
contradicting the fact that $Z$ is not  a Weierstrass point. Now
let us consider the locus $E$ inside~$\tilde\X^{(2)}$ consisting of pairs $(Q,\iota Q)$. Then the
Jacobian $\Jac[\tilde\X]$ of $\tilde\X$
 is the quotient $\tilde\X^{(2)}/E$. And since $4Z-K_{\tilde\X}$ is not canonical, this implies that for
each point $Z\notin\WP$ there is a unique pair $(N_{1},N_{2})$ such that
$$\Ox[\tilde\X](K_{\tilde\X}+N_{1}+N_{2}-4Z)=\Ox[\tilde\X].$$

It remains to show that the projection of the set of triples $(N_{1},N_{2},Z)$ to the first coordinate is
finite. Since $\tilde\X$ is a curve, it is enough to
show that there are no pairs $(Q_{1},Q_{2})\in\tilde{\X}$ such that for an open set of $Q\in\tilde\X$ the
equality $K_{\tilde\X}+Q_{1}+Q_{2}-4Q\sim 0$ holds. But this is clearly the case, because the map of
$\tilde\X\to\Jac[\tilde\X]$ is nondegenerate and the pairs are never conjugated by the hyperelliptic
involution.

Now using the plumbing cylinder construction of Theorem~\ref{theoreme:PlomberieCylindriqueSansResidu}, we can
smooth every of these differentials, preserving the zero of order four.
Moreover, the curves that we obtain are clearly not hyperelliptic since the special fibre is not hyperelliptic. 

Suppose now that $Z$ is a Weierstrass point of $\tilde\X$.
We have to show that every smoothing of such a curve which preserves the zero of order $4$ is hyperelliptic.
An
analogous using the Riemann-Roch Theorem implies that the points~$N_{1}$ and~$N_{2}$ are conjugated by the
hyperelliptic involution. But then the continuity of the parity of the generalised Arf invariant proved in
Theorem~\ref{theoreme:InvDeArfGene} concludes the proof.
\end{proof}

We now prove  Theorem~\ref{theoreme:bordZeroOrdre4Irrbis} following a similar scheme.

\begin{proof}[Proof of Theorem~\ref{theoreme:bordZeroOrdre4Irrbis}.]
 First we prove that it is necessary that the differentials are of the form given in  Theorem~\ref{theoreme:bordZeroOrdre4Irrbis}. 
 
 It is clear that the point $Z$ is on the bridge between $N_{1}$ and $N_{2}$ since otherwise $(\X,Z)$ would
not be stable. Moreover, the points which form the node are conjugated by the hyperelliptic
involution or one of them is a Weierstrass point. Otherwise, the differential would have a zero at a smooth
point of $\tilde\X$. But this zero would be preserved by any
deformation, contradicting the fact that the differential is in the boundary of $\pomoduli[3,1](4)$ and that $Z\notin\tilde\X$.
 
 Let us suppose that we are in the first case: the restriction of $\omega$ to $\tilde\X$ has a zero of order
two at $N_{1}$. Let us take a coordinate $z$ on $\PP^{1}$ such that $0$ is identified to $N_{2}$ and $\infty$
to
$N_{1}$. We define the differential form $\eta:=\frac{(z-1)^{4}}{z^{2}}\dz$. 
We want to use the plumbing cylinder construction with parameters $(\epsilon_{1},\epsilon_{2})$ at the nodes.
By Lemma~\ref{lemme:ConditionPlomberieCheminsFermer}, they have to satisfy $\epsilon_{1}=\epsilon_{2}^{3}=:c$.
We can find a differential $\eta$ on~$\tilde\X$ with simple poles at $N_{1}$ and $N_{2}$ and holomorphic
otherwise.
 By Lemma~\ref{lemme:PlomberieCylindriqueAvecResidu}, we can plumb the differential and obtain an
holomorphic differential with a zero of order $4$. Moreover, this
differential is not hyperelliptic since the special fibre is not hyperelliptic.
 This proves the first point.
 
 Let us now suppose that the differential has a single zero at both $N_{1}$ and~$N_{2}$. We can still use
Lemma~\ref{lemme:PlomberieCylindriqueAvecResidu} to plumb this differential. But this
time, there are two distinct ways (up to isomorphisms) to plumb the nodes.
Let $\epsilon_{1}$ be the
parameter of the cylinder at the node $N_{1}$, then according to
Lemma~\ref{lemme:ConditionPlomberieCheminsFermer}, the
parameter of the cylinder at $N_{2}$ has to be of the form $\epsilon_{2}=\pm\epsilon_{1}$. 
To conclude the proof, it suffices to show that the case $\epsilon_{1}=\epsilon_{2}$ leads to a hyperelliptic
curve and that the case $\epsilon_{1}=-\epsilon_{2}$ leads to a non hyperelliptic curve. 

From now on, we will use the notations of  Lemma~\ref{lemme:PlomberieCylindriqueOk}.
The hyperelliptic involution $\iota$ on $\X$ restricts to the hyperelliptic involution on $\tilde\X$  and  to
the involution which fixes $Z$ and permutes $N_{1}$ and $N_{2}$ on
the component $\PP^{1}$. Hence we can suppose that
the two open sets $U_{1}$ and $U_{2}$ and the
coordinates $z_{1}$, $w_{1}$ on $U_{1}$ and $z_{2}$, $w_{2}$ on $U_{2}$ are chosen such that
$\iota(z_{i})=z_{j}$ and
$\iota(w_{i})=w_{j}$ for $i\neq j$. 
 
Let us suppose that the cylinder plumbed at the node $N_{1}$ is given by the equation
$x_{1}y_{1}=\epsilon_{1}$ and at the node $N_{2}$ it is given by
$x_{2}y_{2}=\pm\epsilon_{1}$. Then on the cylinders, the hyperelliptic involution has to be
of the form $\iota(x_{1})=x_{2}$ and
$\iota(y_{1})=\pm y_{2}$ in order to coincide with the hyperelliptic involution on the part of the smoothed
curve coming  from $\tilde\X$. But it is easy to verify that this map can be prolonged to a holomorphic map
on the whole smoothed curve if and only if the sign is positive. Moreover, in this case one can easily
verify that this map is the hyperelliptic involution of the smoothed curve. And in the other case, the
uniqueness of the hyperelliptic involution implies that the smoothed curve cannot be hyperelliptic. 
\end{proof}

We can deduce from Theorem~\ref{theoreme:bordZeroOrdre4Irrbis} the surprising fact that the odd and
hyperelliptic components  of the {incidence variety compactifications} of $\pomoduli[3,1](4)$ meet at their
boundaries.

\begin{cor}\label{corollaire:intersectHypOddGenreTrois}
 Let $\X$ be the union of a curve $\tilde{\X}$ of genus two and a projective line glue together at a pair of
points of $\tilde{\X}$ conjugated by the hyperelliptic involution. Let $Z\in \PP^{1}$ and $\omega$ be a
differential
which vanishes on $\PP^{1}$ and has two single zeros at the points which form the nodes on $\tilde{\X}$.

Then the pointed differential $(\X,\omega,Z)$ is in  $\obarmoduliinc[3,1]{4}^{\hyp}$ and
$\obarmoduliinc[3,1]{4}^{\odd}$.
\end{cor}

\paragraph{Examples.}

We
give two examples of concrete families in
$\omoduli[3,1]^{\odd}(4)$ which degenerates to a curve stably equivalent to an irreducible curve with one
node.
The first one is given as family of curves in $\PP^{2}$ with a hyperflex. The second is a family of flat
surfaces given as a family of polygons with identifications. 

\begin{ex}\label{exemple:FamilleCourbesPlanesVersDelta0}
 We define in $\PP^{2}\times\Delta$ the family of curves given by:
$$P(x,y,z;t):= xyz^{2}+y^{4}+x^{3}z+tz^{4}.$$
Each curve has a hyperflex of order 4 at $(1,0,0;t)$, thus the differential corresponding to the line at
infinity has a
zero of order 4 at this point. The special curve is irreducible with only one node as singularity. Moreover
the
differential associated to the tangent has a simple pole at the node. Now the Weierstrass form of the
normalisation
is $y^{2}+4x^{5}-1$ and the  preimages of the node are over~$x=0$ and $x=\infty$. In particular, the point
which is over $x=\infty$ is a Weierstrass point. We can show that the Igusa invariant of this curve is
zero. 

More generally, let us consider the family
$$\left\{ xyz^{2}+y^4+a_{1}x^3z+a_{2}x^2yz+a_{3}xy^2z+a_{4}y^3z+tz^4=0 \right\}\subset\PP^{2}\times \{t\},$$
where the $a_{i},\ i=1,\cdots,4$ are complex numbers. This gives us examples where the special curve has any
given Igusa invariants.

Let us now take a look at the family given by the equation
$$P(x,y,z;t):= x^2yz+y^{4}-x^{3}z+tz^{4}.$$
Moreover, the differential associated to the line at infinity has a zero of order~$4$ at $(1,0,0;t)$.
The singularity of the special curve  is a cusp meeting a smooth
branch. It follows from the classification of Kang
\cite[Corollary~2.5]{MR1760371} and the fact that the family is smooth, that the stable limit of this
family is an irreducible curve with one node. The limit of the zeros of order $4$ is in  the node. The limit
stable differential has a zero of order two at one of the
preimages of the node, which is also a Weierstrass point. In this example,  the other preimage of the node is
a
Weierstrass point of the normalisation.
\end{ex}

Let us now give examples using the polygonal representation of the flat surfaces. Since a complete
classification of the cylinder decompositions of the flat surfaces in $\pomoduli[3,1](4)$ is given in
\cite[Proposition~3.1]{arXiv5879A}, these examples could lead to another proof of
Theorem~\ref{theoreme:bordZeroOrdre4Irrbis} using degeneration of these diagrams.

\begin{ex}\label{exemple:SurfacesPlattesVersDelta0}
First we give in Figure~\ref{figure:BordOM3(4)IrrWP} an example of a curve such that a zero of order two is
identified with another point of the curve. In this figure and in the following one, the vertical segment are
identified by an horizontal translation.
In this example, it is not difficult to see that the second point which forms the node is a Weierstrass point
of the curve. However, it is not difficult to construct examples where this point is not a Weierstrass
point.
\begin{figure}[ht]
\centering
\begin{tikzpicture}[scale=1.2]
  \draw (0,0) coordinate (a1) -- node [below] {$1$} (1,0) coordinate (a2) --node [below] {$2$} (2,0)
coordinate (a3) --node [below] {$x$} (3,0) coordinate (a4) --node [below] {$3$} (4,0) coordinate (a5) -- (4,1)
coordinate (a6) --node [above] {$x$} (3,1) coordinate (a7) --node [above] {$3$} (2,1) coordinate (a8) -- (2,2)
coordinate (a9) --node [above] {$2$} (1,2) coordinate (a10) -- (1,1) coordinate (a11) --node [above] {$1$}
(0,1) coordinate (a12) -- cycle;
  \foreach \i in {1,2,...,12}
  \fill (a\i) circle (1pt);

  \draw[->] (4.2,1) -- node [below] {$x\to 0$} (4.8,1);
  
    \draw (5,0) coordinate (b1) -- node [below] {$1$} ++(1,0) coordinate (b2) -- node [below] {$2$} ++(1,0)
coordinate (b3) --node [below] {$3$} ++(1,0) coordinate (b4) -- ++(0,1) coordinate (b5) --node [above] {$3$}
++(-1,0) coordinate (b6) -- ++(0,1) coordinate (b7) --node [above] {$2$} ++(-1,0) coordinate (b8) -- ++(0,-1)
coordinate (b9) --node [above] {$1$} ++(-1,0) coordinate (b10) -- cycle;
  \foreach \i in {1,4,5,10}
  \fill (b\i) circle (1pt);
  \foreach \i in {2,3,6,7,8,9}
  \filldraw[fill=white,draw=black] (b\i) circle (1pt);
\end{tikzpicture}
\caption{A family of curve in $\obarmoduli[3](4)$ degenerating to an irreducible curve with one of the points
of the node a Weierstrass point.}
\label{figure:BordOM3(4)IrrWP}
\end{figure}
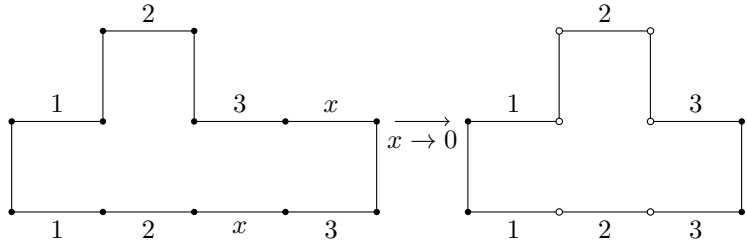

More interesting is the case where the special curve is irreducible and the nodal points are conjugated by
the hyperelliptic involution. In this case, we can produce a smoothing in both connected components of
$\omoduli[3,1](4)$. The  Figure~\ref{figure:BordOM3(4)IrrConjug} shows such a smoothing. One can easily
verify
that the smoothing are in the correct stratum using the Arf invariant of these curves. A consequence of this
is that the Arf Invariant of the nodal curve depends on the choice of a basis of the homology.

\begin{figure}[ht]
\centering
\begin{tikzpicture}[scale=1.1]
   \draw (0,0) coordinate (a1) -- node [below] {$1$} ++(1,0) coordinate (a2) -- ++(0,-1) coordinate (a3) --
node [below] {$2$} ++(1,0) coordinate (a4) -- ++(0,1) coordinate (a5) -- ++(0,1) coordinate (a6) --node
[above] {$2$} ++(-1,0) coordinate (a7) -- ++(0,1) coordinate (a8) --node [above] {$1$} ++(-1,0) coordinate
(a9) -- ++(0,-1) coordinate (a10) -- cycle;
  \foreach \i in {1,2,5,8,9}
  \fill (a\i) circle (1pt);
  \foreach \i in {3,4,6,7,10}
  \filldraw[fill=white,draw=black] (a\i) circle (1pt);

  \draw[->] (0.9,-1.2) -- node [below,sloped] {odd} ++(-1.5,-1);
  \draw[->] (1.1,-1.2) -- node [below,sloped] {hyperelliptic} ++(1.5,-1);
  
    \draw (-2,-3.5) coordinate (b1) -- node [below] {$1$} ++(1,0) coordinate (b2) -- node [below] {$x$}
++(.5,0) coordinate (b3) -- ++(0,-1) coordinate (b4) --node [below] {$2$} ++(1,0) coordinate (b5) -- ++(0,1)
coordinate (b6) -- ++(0,1) coordinate (b7) --node [above] {$2$} ++(-1,0) coordinate (b8) --node [above] {$x$}
++(-.5,0) coordinate (b9) -- ++(0,1) coordinate (b10) --node [above] {$1$} ++(-1,0)coordinate (b11) --
++(0,-1)coordinate (b12) -- cycle;
  \foreach \i in {1,2,...,12}
  \fill (b\i) circle (1pt);

   \draw (3.3,-3.5) coordinate (c1) -- node [below] {$1$} ++(1,0) coordinate (c2) -- ++(0,-1) coordinate (c3)
-- node [below] {$2$} ++(1,0) coordinate (c4) -- ++(0,1) coordinate (c5) -- ++(0,1) coordinate (c6) --node
[above] {$2$} ++(-1,0) coordinate (c7) -- ++(0,1) coordinate (c8) --node [above] {$1$} ++(-1,0) coordinate
(c9) --node [above] {$x$} ++(-.5,0) coordinate (c10) -- ++(0,-1) coordinate (c11) --node [below] {$x$}
++(.5,0) coordinate (c12)   -- cycle;
 
      \foreach \i in {1,2,...,12}
  \fill (c\i) circle (1pt);
\end{tikzpicture}

\caption{Two smoothings of an irreducible curve with a node of conjugated points, one of them in
$\omoduli[3,1]^{\odd}(4)$ and the other in $\omoduli[3,1]^{\hyp}(4)$.}
\label{figure:BordOM3(4)IrrConjug}
\end{figure}
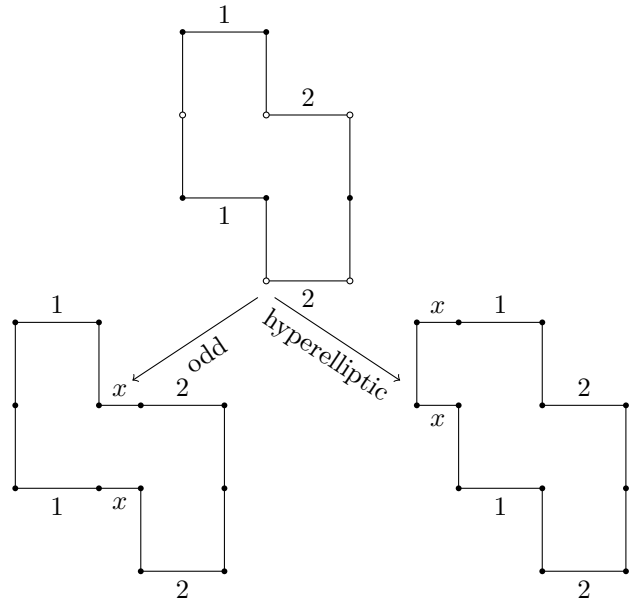
\end{ex}

\section*{Acknowledgements.}                 
This work consists of a part of my doctoral dissertation. I am grateful to Martin M\"oller for his
guidance. I would like to thank Sam Grushevsky for helpful suggestions. Dawei Chen has informed the author
that he has obtained some of the results in Sections \ref{section:Spin} and \ref{section:omoduli3,4,odd}
independently, the methods being in some case related, in some cases disjoint.
During
my PhD I have been supported by the ERC-StG 257137.


\addcontentsline{toc}{section}{References.}
\bibliographystyle{alpha}
\bibliography{biblio}

\bigskip

\noindent
\small{\textsc{Quentin Gendron}, Institut f\"ur Mathematik,\\
 Goethe-Universit\"at,
Robert-Mayer-Str. 6-8,
D-60325 Frankfurt am Main \\
{\em E-mail:} gendron@math.uni-frankfurt.de}

\end{document}